\documentclass[a4paper,reqno,11pt]{amsart}
\usepackage{amsmath}
\usepackage{amssymb}
\usepackage{amsthm}
\usepackage{microtype}
\usepackage{mathrsfs}
\usepackage{graphicx}
\usepackage{textcomp}
\usepackage{lmodern}
\usepackage{mathtools}
\usepackage{xcolor}
\usepackage{supertabular}
\usepackage{booktabs}
\usepackage{url}
\usepackage[section]{placeins}

\newtheorem{thm}{Theorem}[section]
\newtheorem{prop}[thm]{Proposition}
\newtheorem{lemma}[thm]{Lemma}
\newtheorem{cor}[thm]{Corollary}
\theoremstyle{remark}

\newtheorem{remark}[thm]{Remark}
\theoremstyle{definition}

\newtheorem{conjecture}[thm]{Conjecture}

\numberwithin{equation}{section}

\let\phi=\varphi
\let\epsilon=\varepsilon

\title{Nodal domains in the square---the Neumann case}

\author{Bernard Helffer}
\author{Mikael Persson Sundqvist}
\address[Bernard Helffer]{ Laboratoire de
Math\'{e}matiques UMR CNRS 8628\\ Universit\'{e} Paris-Sud - B\^{a}t 425\\
F-91405 Orsay Cedex\\ France  and Laboratoire Jean Leray, Universit\'e de Nantes, France.}
\email{bernard.helffer@math.u-psud.fr}
\address[Mikael Persson Sundqvist]{Lund University, Department of Mathematical 
Sciences, Lund, Sweden.}
\email{mickep@maths.lth.se}

\begin{document}
\begin{abstract}
\AA. Pleijel has proved that in the case of the Laplacian on the square with 
Neumann condition, the equality in the Courant nodal theorem (Courant sharp 
situation) can only be true for a finite number of eigenvalues. 
We identify five Courant sharp eigenvalues for the Neumann Laplacian in
the square, and prove that there are no other cases.
\end{abstract}
\maketitle

\section{Introduction}

For an eigenfunction $\Psi_n$ corresponding to the $n$-th eigenvalue $\lambda_n$
(counted with multiplicity) of the Laplace operator in a bounded regular 
domain $\Omega$, we
denote by $\mu(\Psi_n)$ the number of nodal domains of $\Psi_n$. A famous
result by Courant (see~\cite{Cou}) states that $\mu(\Psi_n)\leq n$. If
$\mu(\Psi_n)=n$, then we say that the eigenpair $(\lambda_n,\Psi_n)$ (or just
the eigenvalue $\lambda_n$) is Courant sharp.
It is proved in~\cite{Pl,Pol} that, for general planar domains, and with
Dirichlet or Neumann boundary conditions, the Courant
sharp situation occurs for a finite number of eigenvalues only. 
Note that in the case of Neumann the additional assumption that the boundary is 
piecewise analytic should be imposed due the use of a theorem by 
Toth--Zelditch~\cite{ToZe} counting the number of nodal domains whose closure 
is touching the boundary.

In the recent years, the question of determining the Courant sharp cases 
reappears in connection with the determination of minimal spectral partitions 
in the work of Helffer--Hoffmann-Ostenhof--Terracini~\cite{HHOT}. The Courant 
sharp situation was analyzed there in the case of the irrational rectangle and 
in the case of the disk for Dirichlet boundary condition. The case of 
anisotropic (irrational) tori is solved in~\cite{HH}.

Recently the Courant sharp cases were identified in the cases of $\Omega$ being 
a square with Dirichlet boundary conditions imposed~\cite{Pl,BH}, and $\Omega$ 
being the two-sphere~\cite{BHSphere}. Here, our aim is to do the same detailed 
analysis in the case of $\Omega$ being a square with Neumann boundary conditions.

We let $\Omega=\{(x,y)\in\mathbb{R}^2~|~0<x<\pi,\ 0<y<\pi\}$ and
denote by $L$ the self-adjoint Neumann Laplacian in $L^2(\Omega)$.
This operator has eigenvalues
\[
0=\lambda_1<\lambda_2\leq\cdots \leq \lambda_n\leq \cdots,
\]
generated by the set $\{p^2+q^2~|~p,q\in\mathbb{N}\cup\{0\}\}$. 
A basis for the eigenspace corresponding to the eigenvalue
$\lambda=p^2+q^2$ is given by
\[
\{\cos px\cos qy~|~p,q\in\mathbb{N}\cup\{0\},\ p^2+q^2=\lambda\}.
\]
\AA. Pleijel was in particular referring to  figures appearing in the book of 
Pockels~\cite{Po} 
(and partially reproduced in \cite{CH}) like in the Figure~\ref{Pockel}.
\begin{figure}[htbp]
\centering
\includegraphics[width=10cm]{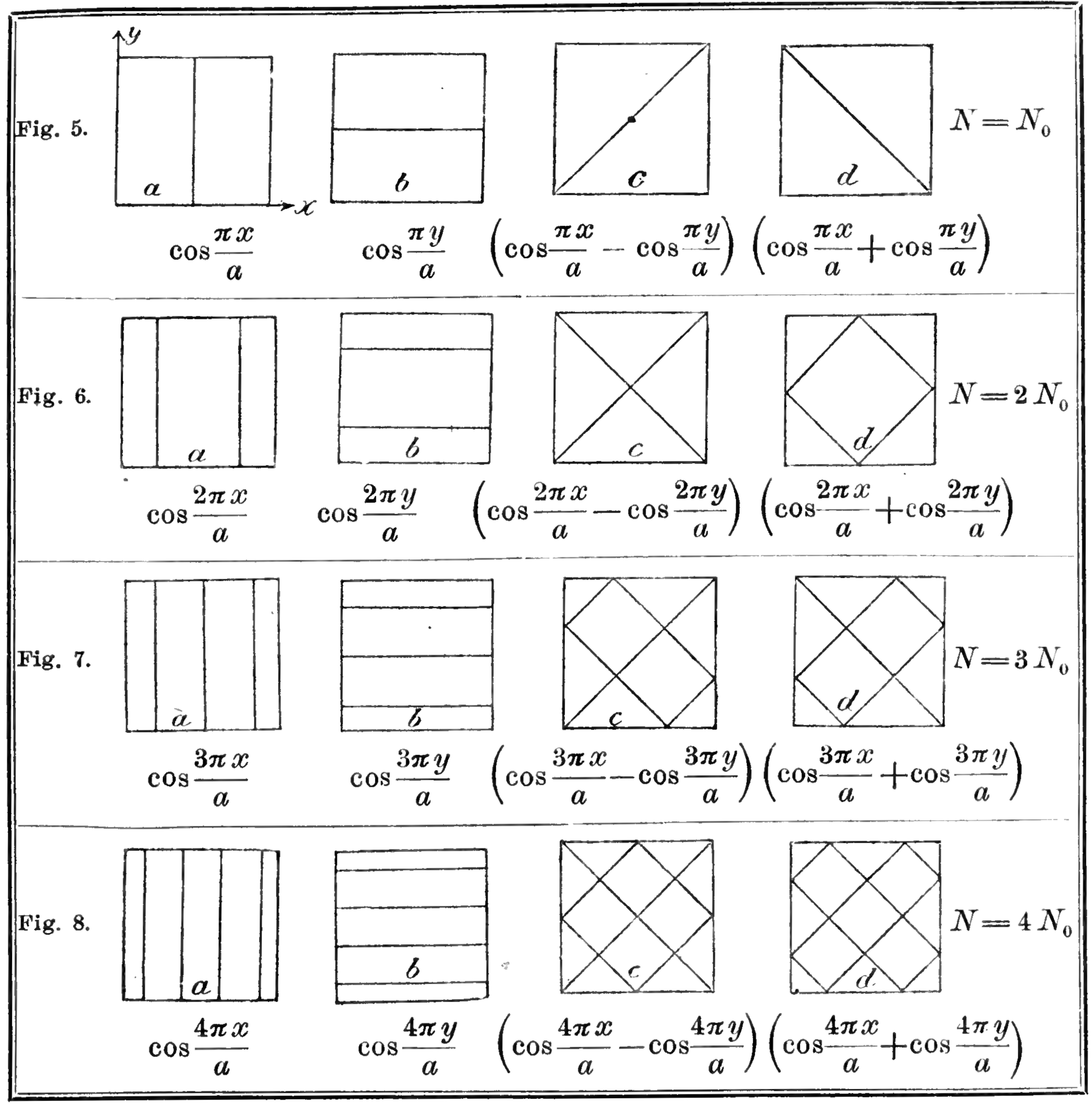} \label{Pockel}
\caption{Nodal patterns for the Neumann eigenfunctions in the square 
$(0,1)^2$ in the book of Pockels (1891).}
\end{figure}

\begin{thm}
\label{thm:main} There exists a Courant sharp eigenpair 
 $(\lambda_n,\Psi_n)$  if and only if $n\in\{1,2,4,5,9\}$.
\end{thm}

 The Courant sharpness of eigenvalues $\lambda_1$, $\lambda_2$ and $\lambda_5$ follows from Lemma~\ref{lem:p0} and the Courant sharpness of 
$\lambda_4$ and $\lambda_9$ follows from Lemma~\ref{lem:pp}. These cases are
illustrated in Figure~\ref{fig:CS}. They 
correspond to the zero sets of the following eigenfunctions:
\begin{itemize}
\item $n=1$ : $(x,y)\mapsto 1$\,;
\item $n=2$ : $(x,y)\mapsto \cos \theta \cos x + \sin \theta \cos y\,$ (with $\theta=1$ in Figure~\ref{fig:CS}); 

\item $n=4$ : $(x,y) \mapsto \cos x \cos y\,$;
\item $n=5$ : $(x,y) \mapsto \cos 2x + \cos 2y\,$;
\item $n=9$ : $(x,y) \mapsto \cos 2x \cos 2y\,$.
\end{itemize}

\begin{figure}[htbp]
\centering
\includegraphics[width=2cm]{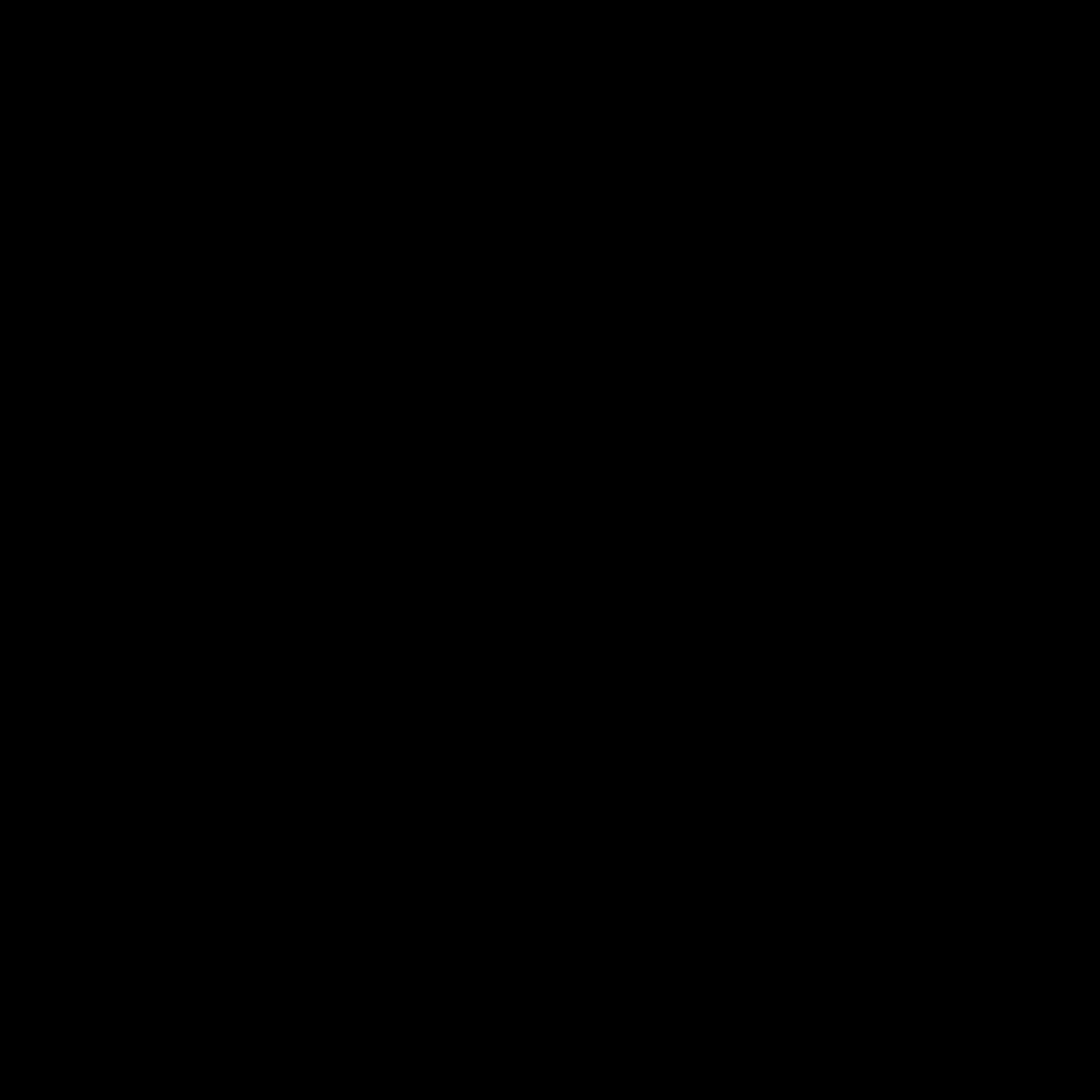}
\hskip 0.5cm
\includegraphics[width=2cm]{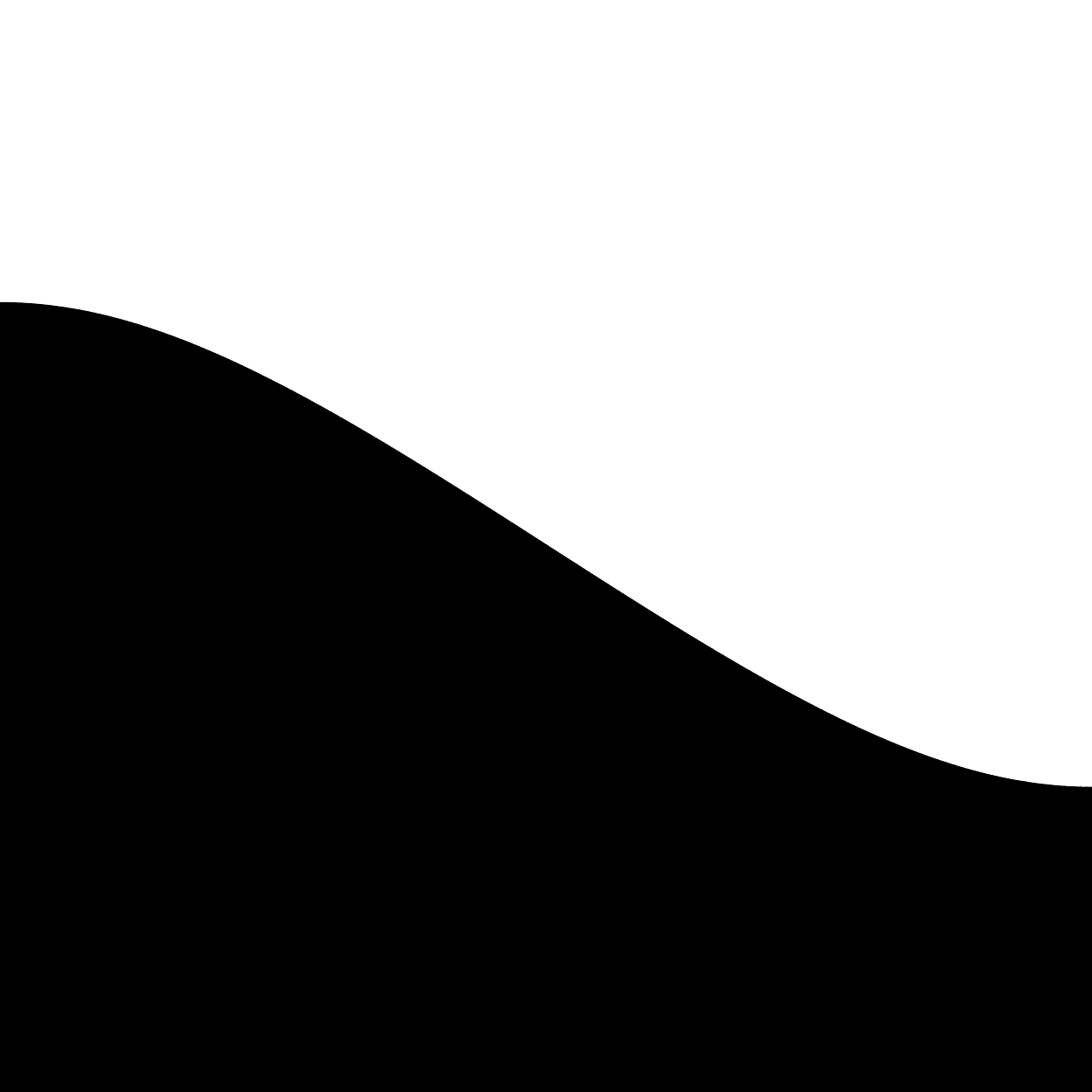}
\hskip 0.5cm
\includegraphics[width=2cm]{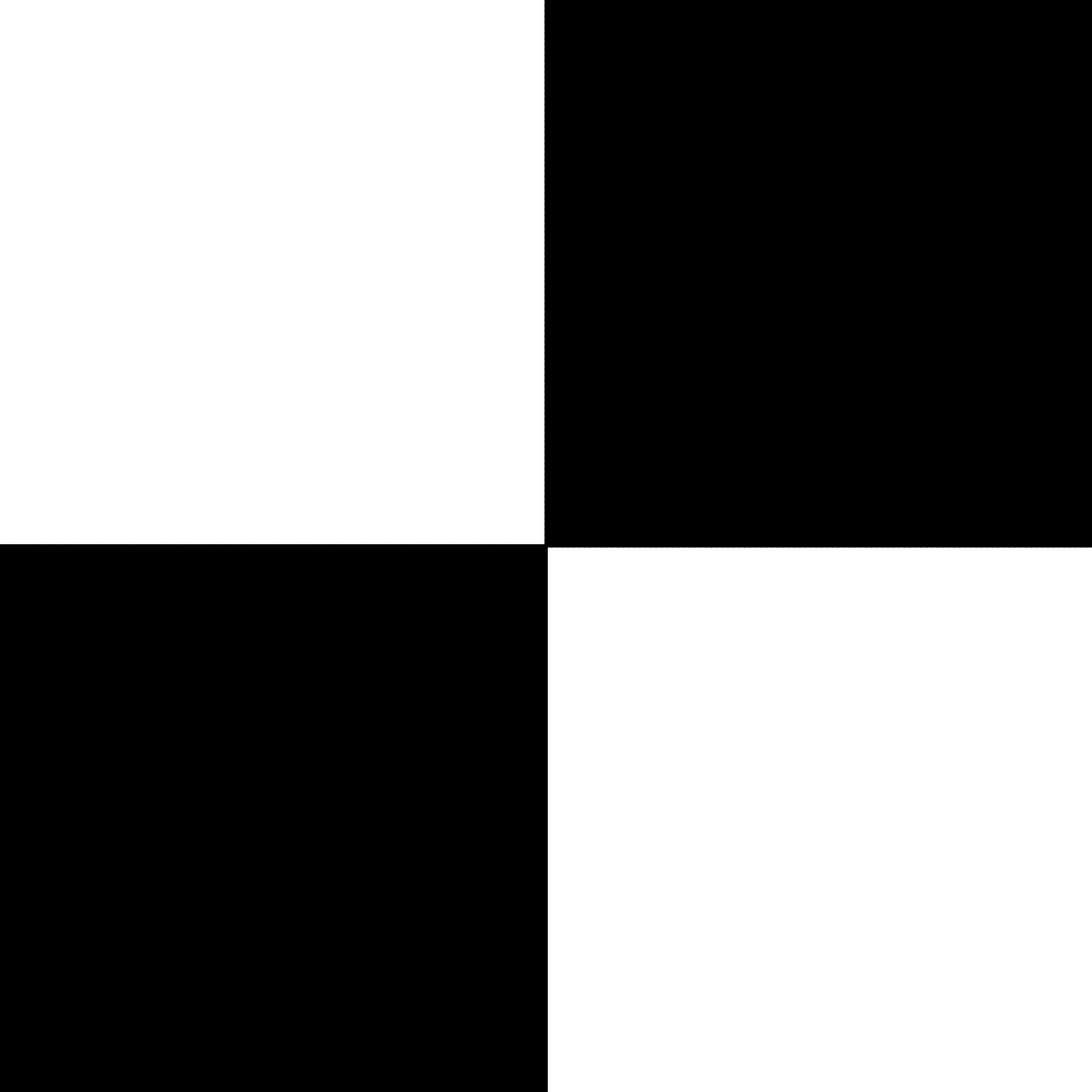}
\hskip 0.5cm
\includegraphics[width=2cm]{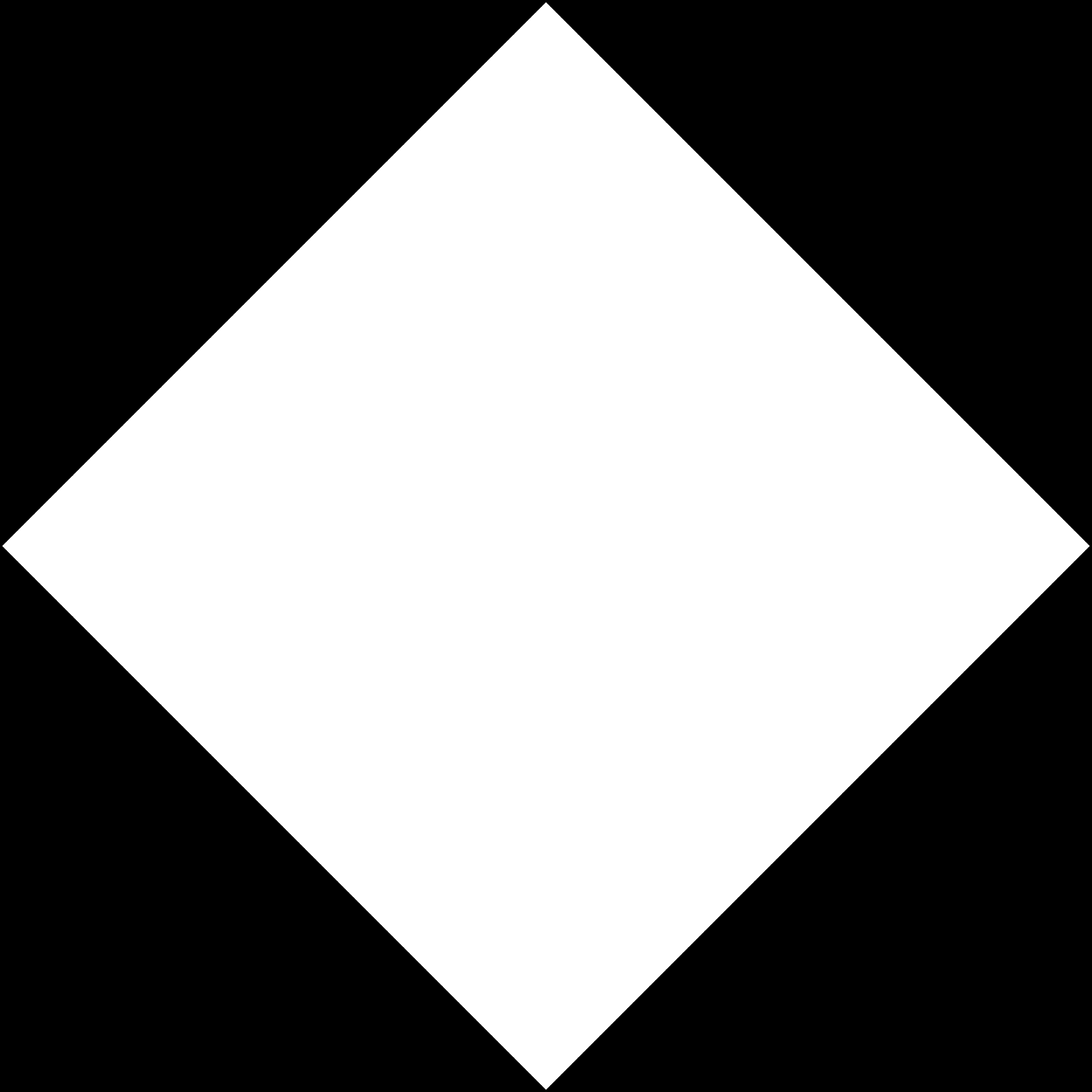}
\hskip 0.5cm
\includegraphics[width=2cm]{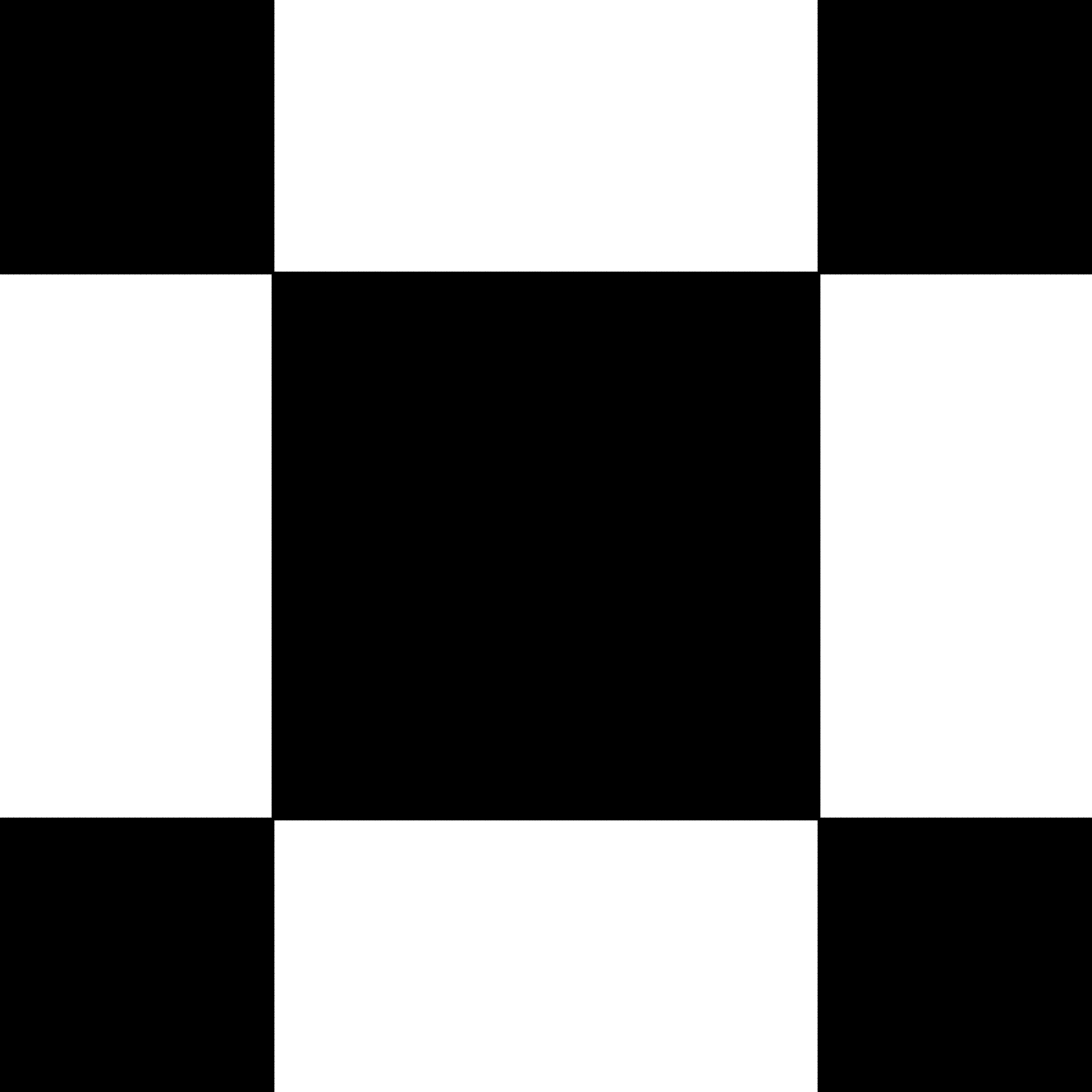}
\caption{The figure shows the nodal sets  in the five Courant sharp cases. From 
left to right, $n=1$, $n=2$, $n=4$, $n=5$, $n=9$. In each example the black
and white areas represent nodal domains where the function has different sign.}
\label{fig:CS}
\end{figure}

The proof of Theorem~\ref{thm:main} is divided into several lemmas and 
propositions. Although following the general scheme proposed by 
\AA. Pleijel~\cite{Pl} and completed in~\cite{BH} for the Dirichlet case, the 
realization of the program in the case of Neumann is more difficult and finally 
involves a combination of arguments present 
in~\cite{Pl},~\cite{St},~\cite{Ley0},~\cite{Ley},~\cite{HHOT},~\cite{HH} and~\cite{BH}.

First we reduce to a finite number of possible Courant sharp cases in 
Section~\ref{Section2}. In Section~\ref{Section3} we use different symmetry
arguments. In Section~\ref{Section4}, we consider two families of eigenfunctions 
corresponding to $\lambda = p^2$ and $\lambda = 2 p^2$ for which a complete 
description is easy.

Section~\ref{Section5} gives the general approach for the analysis of the 
critical points and the boundary points together with a rough localization of 
the zero set initiated by A. Stern: the chessboard localization. The rest of the 
cases, which needs a separate treatment, are taken 
care of in Sections~\ref{Section6} and~\ref{Section7}. In Section~\ref{Section8} 
we indicate 
how one can improve the estimates, if striving for optimal results. Finally,
in Section~\ref{sec:table} we give a list of all eigenvalues together with a
reference to the lemma in which they are treated. We conclude by a short 
discussion on open problems.

Proposition~\ref{prop:red1} below reduces our study to a finite number of
eigenvalues. We provide animations showing the nodal domains in all finite cases studied where the eigenspace is two-dimensional\footnote{See \url{http://www.maths.lth.se/matematiklth/personal/mickep/nodaldomains/}}.

\section{Necessary conditions for Courant sharpness and first reductions}
\label{Section2}
Given an eigenfunction $\Psi$, we introduce the subset 
$\Omega^{\text{inn}}\subseteq\Omega$ as the union of nodal domains of $\Psi$
that do not touch the boundary of $\Omega$, except at isolated points. We also
introduce $\Omega^{\text{out}}\subseteq\Omega$ as the union of nodal domains of 
$\Psi$ not belonging to $\Omega^{\text{inn}}$. We also denote by 
$\mu^{\text{inn}}(\Psi)$ and $\mu^{\text{out}}(\Psi)$ the number of 
nodal domains of $\Psi$ restricted to $\Omega^{\text{inn}}$ and 
$\Omega^{\text{out}}$, respectively. It is clear that
\[
\mu(\Psi)=\mu^{\text{inn}}(\Psi)+\mu^{\text{out}}(\Psi).
\]
From~\cite{Pl} we know that if $(\lambda_n,\Psi_n)$ is an eigenpair of $L$ then
\begin{equation}
\label{eq:muout}
\mu^{\text{out}}(\Psi_n)\leq 4\sqrt{\lambda_n}.
\end{equation}

Moreover, we can write $\Omega^{\text{inn}}=\bigcup_i \Omega^{\text{inn}}_{i}$ 
as a finite union of pairwise disjoint nodal domains for $\Psi_n$. 
The Faber--Krahn inequality~\cite{F,K}
for each inner nodal domain $\Omega^{\text{inn}}_{i}$ says
\begin{equation}
\label{eq:FK}
\frac{1}{\lambda_n}
\leq
\frac{\bigl|\Omega^{\text{inn}}_i\bigr|}{\pi j_{0,1}^2},
\end{equation}
where $\bigl|\Omega^{\text{inn}}_i\bigr|$ denotes the area of 
$\Omega^{\text{inn}}_i$ and $j_{0,1}$ the first positive zero of the Bessel
function $J_0$. Summing, we get
\begin{equation}
\label{eq:FK2}
\mu^{\text{inn}}(\Omega) 
\leq \,  \frac{\bigl|\Omega^{\text{inn}}\bigr|}{\pi j_{0,1}^2}\lambda_n.
\end{equation}

\begin{prop}
\label{prop:red1}
Assume that $(\lambda_n,\Psi_n)$ is a Courant sharp eigenpair. Then $n\leq 208$.
\end{prop}

\begin{proof}
Let $N(\lambda)$ denote the number of eigenvalues strictly less than 
$\lambda$, counting multiplicity. The Weyl law~\cite{W}  says 
that $N(\lambda)\sim \frac{\pi}{4}\lambda$ but we need the following universal 
lower bound for the Neumann problem in the square obtained by direct counting 
(see~\cite{Pl} for the Dirichlet case with the correction mentioned 
in~\cite{BH}):
\begin{equation}
\label{eq:Weyl}
N(\lambda)>\frac{\pi}{4}\lambda.
\end{equation}
Assume that $(\lambda_n,\Psi_n)$ is a Courant sharp eigenpair. The theorem 
of Courant implies that $\lambda_n>\lambda_{n-1}$ and $N(\lambda_n)=n-1$.
Inserting this into \eqref{eq:Weyl} gives us
\[
\lambda_n<\frac{4}{\pi}(n-1).
\]
Combining this with~\eqref{eq:muout} and~\eqref{eq:FK2}, and the
estimate $|\Omega^{\text{inn}}|\leq|\Omega|=\pi^2$,
\[
n=\mu(\Psi_n) \leq \frac{\bigl|\Omega^{\text{inn}}\bigr|}{\pi j_{0,1}^2}\lambda_n + 4\sqrt{\lambda_n}
 < \frac{4}{j_{0,1}^2}(n-1)+\frac{8}{\sqrt{\pi}}\sqrt{n-1}.
\]
 This inequality is false if $n\geq 209$.
\end{proof}

Depending on the cases, we can consider many variants of the  intermediate 
steps in the proof of Proposition~\ref{prop:red1} and introduce small useful 
improvements which can be used directly.

\begin{lemma}
\label{lemma2.2}
Assume that $(\lambda_n,\Psi_n)$ is an eigenpair of $L$. Then
\begin{equation}
\mu(\Psi_n)\leq  \frac{\bigl|\Omega^{\text{inn}}\bigr|}{\pi j_{0,1}^2}\lambda_n + 4\lfloor\sqrt{\lambda_n}\rfloor \,.
\end{equation}
\end{lemma}

\begin{proof}
This follows immediately from~\eqref{eq:muout} and~\eqref{eq:FK2} together with
the fact that $\mu^{\text{out}}(\Psi_n)$ must be an integer.
\end{proof}

For $n\geq 1$, we denote by $P_n:=\max\{p~|~p^2+q^2=\lambda_n,\ p,q\in\mathbb{N}\cup\{0\}\}$. 

\begin{lemma}
\label{lemma2.3}
Assume that $(\lambda_n,\Psi_n)$ is an eigenpair of $L$. Then
\[
\mu(\Psi_n)\leq \frac{\pi}{j_{0,1}^2}\lambda_n+ \max{(4 P_n, 1)}\,.
\]
\end{lemma}

\begin{proof}
We observe that  $\mu^{\text{out}}(\Psi_n)\leq \max{(4 P_n, 1)}$.
Hence,
\[
\mu(\Psi_n)=\mu^{\text{inn}}(\Psi_n)+\mu^{\text{out}}(\Psi_n)\leq 
\frac{\pi}{j_{0,1}^2}\lambda_n+\max{(4 P_n, 1)}.\qedhere
\]
\end{proof}

\begin{cor}
\label{cor2.3}
The eigenvalues $\lambda_n$, where $n$ is one of 86, 95--96, 99--100, 103--104, 
113, 118--119, 120--121, 128--142, 147--208,
are not Courant sharp.
\end{cor}

\begin{proof}
Assume that $n$ is such that $\lambda_{n-1}<\lambda_n$. Then, a numerical
calculation shows that
\[
\frac{\pi}{j_{0,1}^2}\lambda_n+  4P_n <n
\]
for the $n$ mentioned in the statement.
\end{proof}

\section{Reduction by symmetry}
\label{Section3}

\subsection{Preliminaries}
Symmetry arguments will play an important role in the analysis of the Courant 
sharp situation. These ideas appear already in the case of the harmonic 
oscillator and the sphere in contributions by J.~Leydold~\cite{Ley0,Ley}.

We introduce the notation
\begin{equation}
\label{eq:Phidef}
\Phi_{p,q}^{\theta}(x,y) = \cos\theta\cos px\cos qy+\sin\theta\cos qx\cos py.
\end{equation}
We will often write just $\Phi(x,y)$ or $\Phi_{p,q}(x,y)$. For eigenvalues
of $L$ of multiplicity two, the family $\Phi_{p,q}^{\theta}(x,y)$, 
$0\leq \theta<\pi$ 
will give all possible eigenfunctions (up to multiplication by a non zero 
constant). Moreover, the basis of our arguments
are the rich symmetry of the trigonometric functions. The role of the 
antipodal map in the case of the sphere is now replaced in the case of the 
sphere by the map:
\[
(x,y) \mapsto (\pi-x,\pi -y)\,.
\]
A finer analysis will involve the finite group generated by the maps 
$(x,y)\mapsto (\pi-x,y)$ and $(x,y)\mapsto (x,\pi-y)$.

\subsection{Odd eigenvalues} 
We introduce $L^{\text{ARot}}$, the Neumann Laplacian restricted to the 
antisymmetric space 
\[
\mathcal H^{\text{ARot}} = \{\psi~|~\psi(\pi-x,\pi-y)=-\psi(x,y)\}\,.
\]
The spectrum of this Laplacian is given by $p^2 + q^2$ with $p+q$ odd. We 
denote by $(\lambda_n^{\text{ARot}})_{n=1}^{+\infty}$ the sequence of 
eigenvalues of $L^{\text{ARot}}$, counted with multiplicity. Then each odd
$\lambda_n$ equals $\lambda_m^{\text{ARot}}$ for some $m$. The next lemma is 
an adaptation of Courant's theorem in this subspace.
\begin{lemma}
\label{lem:antisymmetric}
Assume that $(\lambda_n,\Psi_n)$ is an eigenpair of $L$, with $\lambda_n$ odd,
and let $m$ be such that $\lambda_n=\lambda_m^{\text{ARot}}$. 
Then $\mu(\Psi_n)$ is even, and 
\[
\mu(\Psi_n) \leq 2 m\,.
\]
\end{lemma}
The proof is inspired by a proof of Leydold~\cite{Ley} (used in the case of the 
sphere). See also Leydold~\cite{Ley0} and  B\'erard-Helffer~\cite{BH2} for the 
case of the harmonic oscillator.

\begin{proof}
By assumption we have
\[
\Psi_n(\pi-x,\pi-y)=-\Psi_n(x,y)\,.
\]
This implies that $\mu(\Psi_n)$ is even and that the family of nodal domains 
is the disjoint union of $r$ pairs, each pair consisting of two disjoint open 
sets exchanged by $(x,y) \mapsto (\pi -x, \pi -y)$. Restricting $\Psi_n$ 
to each pair, we obtain an $r$-dimensional antisymmetric space whose energy 
is bounded by $\lambda_n$. Hence $\lambda_n \geq \lambda_r^{\text{ARot}}$ by the 
min-max principle and $m \geq r$. Thus $\mu(\Psi_n)=2r\leq 2m$.
\end{proof}

\begin{cor}
\label{cor:antisymmetric}
The eigenvalues 
$\lambda_7=\lambda_8$, 
$\lambda_{23}=\lambda_{24}=\lambda_{25}=\lambda_{26}$, 
$\lambda_{29}=\lambda_{30}$, 
$\lambda_{36}=\lambda_{37}$, 
$\lambda_{40}=\lambda_{41}$, 
$\lambda_{51}=\lambda_{52}$, 
$\lambda_{55}=\lambda_{56}$, 
$\lambda_{59}=\lambda_{60}=\lambda_{61}=\lambda_{62}$, 
$\lambda_{72}=\lambda_{73}$, 
$\lambda_{76}=\lambda_{77}=\lambda_{78}=\lambda_{79}$, 
$\lambda_{91}=\lambda_{92}$, 
$\lambda_{97}=\lambda_{98}$, 
$\lambda_{99}=\lambda_{100}$, 
$\lambda_{103}=\lambda_{104}$, 
$\lambda_{109}=\lambda_{110}=\lambda_{111}=\lambda_{112}$, 
$\lambda_{120}=\lambda_{121}$, 
$\lambda_{124}=\lambda_{125}=\lambda_{126}=\lambda_{127}$, 
$\lambda_{132}=\lambda_{133}$, 
$\lambda_{143}=\lambda_{144}=\lambda_{145}=\lambda_{146}$, 
$\lambda_{151}=\lambda_{152}$, 
$\lambda_{157}=\lambda_{158}$, 
$\lambda_{159}=\lambda_{160}=\lambda_{161}=\lambda_{162}$, 
$\lambda_{163}=\lambda_{164}$, 
$\lambda_{169}=\lambda_{170}$, 
$\lambda_{176}=\lambda_{177}=\lambda_{178}=\lambda_{179}$ and 
$\lambda_{186}=\lambda_{187}=\lambda_{188}=\lambda_{189}$ 
are not Courant sharp.
\end{cor}

\subsection{Even eigenvalues}
We let $L^{\text{SRot}}$ denote the Neumann Laplacian  restricted to the 
symmetric space 
\[
\mathcal H^{\text{SRot}} = \{\psi~|~\psi(\pi-x,\pi-y)=\psi(x,y)\}\,.
\]
The spectrum of this Laplacian is given by $p^2 + q^2$ with $p+q$ even. We 
denote by $(\lambda_n^{\text{SRot}})_{n=1}^{+\infty}$ the sequence of 
eigenvalues of $L^{\text{SRot}}$, counted with multiplicity. Each even 
$\lambda_n$ equals $\lambda_m^{\text{SRot}}$ for some $m$. The next lemma is 
an adaptation of Courant's theorem in this subspace.
\begin{lemma}
\label{lem:symmetric}
Let $(\lambda_n,\Psi_n)$ be an eigenpair of $L$, with even $\lambda_n$, and let
$m$ be such that $\lambda_n=\lambda_m^{\text{SRot}}$. Then
\[
\mu(\Psi_n) \leq 2 m\,.
\]
\end{lemma}
It is again inspired by a proof of  Leydold \cite{Ley} (used in the case of the sphere). 

\begin{proof}
By assumption we have $\Psi_n(\pi-x,\pi-y)= \Psi_n(x,y)$.
This implies  that the family of nodal domains is the disjoint union of $r$ 
pairs, each pair consisting of two disjoint open sets exchanged by 
$(x,y) \mapsto (\pi -x, \pi -y)$ and of $s$ symmetric open sets. Hence we have:
\[
\mu(\Psi_n) = 2 r + s\,.
\]
Restricting $\Psi_n$ to each pair, or to each symmetric open set, we obtain 
an $(r+s)$-dimensional antisymmetric space whose energy is bounded by $\lambda_n$. 
Hence $\lambda_n \geq \lambda_{r+s}^{\text{SRot}}$ by the min-max principle and 
$m \geq  r +s$. Thus $\mu(\Psi)=2r+s\leq 2r+2s\leq 2m$.
\end{proof}

\begin{cor}
\label{cor:symmetric}
The eigenvalues 
$\lambda_{27}=\lambda_{28}$, 
$\lambda_{46}=\lambda_{47}=\lambda_{48}$, 
$\lambda_{63}=\lambda_{64}$, 
$\lambda_{82}=\lambda_{83}$, 
$\lambda_{86}$, 
$\lambda_{87}=\lambda_{88}=\lambda_{89}=\lambda_{90}$, 
$\lambda_{107}=\lambda_{108}$, 
$\lambda_{113}$, 
$\lambda_{114}=\lambda_{115}=\lambda_{116}=\lambda_{117}$, 
$\lambda_{138}=\lambda_{139}$, 
$\lambda_{147}=\lambda_{148}=\lambda_{149}=\lambda_{150}$, 
$\lambda_{165}=\lambda_{166}$, 
$\lambda_{167}=\lambda_{168}$, 
$\lambda_{171}=\lambda_{172}=\lambda_{173}$, 
$\lambda_{194}=\lambda_{195}$, 
$\lambda_{198}=\lambda_{199}$, 
$\lambda_{202}=\lambda_{203}$, 
$\lambda_{206}$, and 
$\lambda_{207}=\lambda_{208}$
are not Courant sharp.
\end{cor}

Next, we let $L^{\text{AMir}}$ denote the Neumann Laplacian restricted to the 
anti-symmetric space 
\[
\mathcal H^{\text{AMir}} = \{\psi~|~\psi(\pi-x,y)=-\psi(x,y),\ \psi(x,\pi-y)=-\psi(x,y)\}\,.
\]
The spectrum of this Laplacian is given by $p^2 + q^2$ with $p$ and $q$ odd. 
We denote by $(\lambda_n^{\text{AMir}})_{n=1}^{+\infty}$ the sequence of 
eigenvalues of $L^{\text{AMir}}$, counted with multiplicity. The next lemma is 
an adaptation of Courant's theorem in this subspace.

\begin{lemma}
\label{lem:antimirror}
Assume that $(\lambda_n,\Psi_n)$ is an eigenpair of $L$, with $\lambda_n$ even and  $\Psi_n \in \mathcal H^{\text{AMir}} $. 
Then
\[
\mu(\Psi_n) \leq 4 m\,,
\]
for $m$ such that  $\lambda_n=\lambda_m^{\text{AMir}}$.

Moreover, $\mu(\Psi_n)$ is divisible by $4$.
\end{lemma}

\begin{proof}
By assumption $\Psi_n(\pi-x,y)=-\Psi_n(x,y)$ and 
$\Psi_n(x,\pi-y)=-\Psi_n(x,y)$. This implies that the nodal domains is the 
disjoint union of $r$ quadruples. Hence we have:
\[
\mu(\Psi_n) = 4 r.
\]
This proves the last statement. Restricting $\Psi_n$ to each quadruple, we 
obtain an $r$-dimensional space whose energy is bounded by $\lambda_n$. 
Hence $\lambda_n \geq \lambda_{r}^{\text{AMir}}$ by the min-max principle and 
$m \geq  r$. Thus $\mu(\Psi_n)=4r\leq 4m$.
\end{proof}

\begin{remark}
If, for all pairs $(p,q)$ of non-negative integers such that 
$p^2+q^2=\lambda_n$, it holds that $p$ and $q$ are odd, then there exists an
$m$ such that $\lambda_n=\lambda_m^{\text{AMir}}\,$.
\end{remark}

\begin{cor}
\label{cor:antimirror}
The eigenvalues 
$\lambda_{12}=\lambda_{13}$, 
$\lambda_{20}$, 
$\lambda_{27}=\lambda_{28}$, 
$\lambda_{32}=\lambda_{33}$, 
$\lambda_{46}=\lambda_{47}=\lambda_{48}$, 
$\lambda_{53}=\lambda_{54}$, 
$\lambda_{68}=\lambda_{69}$, 
$\lambda_{74}=\lambda_{75}$, 
$\lambda_{80}=\lambda_{81}$, 
$\lambda_{86}$, 
$\lambda_{95}=\lambda_{96}$, 
$\lambda_{107}=\lambda_{108}$, 
$\lambda_{114}=\lambda_{115}=\lambda_{116}=\lambda_{117}$, 
$\lambda_{128}=\lambda_{129}$, 
$\lambda_{140}$, 
$\lambda_{147}=\lambda_{148}=\lambda_{149}=\lambda_{150}$, 
$\lambda_{153}=\lambda_{154}$, 
$\lambda_{165}=\lambda_{166}$, 
$\lambda_{174}=\lambda_{175}$, 
$\lambda_{184}=\lambda_{185}$, 
$\lambda_{194}=\lambda_{195}$, 
$\lambda_{202}=\lambda_{203}$ and 
$\lambda_{206}$
are not Courant sharp.
\end{cor}

\begin{lemma}
\label{lem:pandqeven}
Assume   that
$\Phi_{p,q}^\theta(\pi,y)=0$ has at least $k$ solutions for $0<y<\pi$ 
($k\geq 0$) and that $\Phi_{p,q}^\theta(x,\pi)=0$ has at least $\ell $ solutions 
($\ell \geq 0$)  for $0<x<\pi$\,.
Then
\[
\mu(\Phi_{2p,2q}^\theta)\leq 4\mu(\Phi_{p,q}^\theta)-(2(k+\ell)+3)\,.
\]
If, moreover, $\Phi_{p,q}^\theta(\pi,\pi)= 0\,$,
\[
\mu(\Phi_{2p,2q}^\theta)\leq 4\mu(\Phi_{p,q}^\theta)-(2(k+\ell)+4)\,.
\]
\end{lemma}

\begin{proof}
The function $\Phi_{2p,2q}^{\theta}$ is even in the lines $x=\pi/2$ and 
$y=\pi/2$. We note that for each zero described in the statement (except the 
biggest one), we count for $\Phi_{2p,2q}^\theta$ one nodal domain two times. 
The one in the middle is subtracted three times if 
$\Phi_{p,q}^\theta(\pi,\pi)\neq 0$ and four
times if $\Phi_{p,q}^\theta(\pi,\pi)=0$.
\end{proof}

\begin{cor}
\label{cor:pandqeven}
The eigenvalues $\lambda_{38}=\lambda_{39}$ and $\lambda_{93}=\lambda_{94}$ 
are not Courant sharp.
\end{cor}

\subsection{Reduction for the domain of definition of the parameter $\theta$}

\begin{lemma}
\label{Lemma3.8}
For odd eigenvalues of multiplicity two, to get the maximum number of possible
nodal domains, it is sufficient to study $\Phi_{p,q}^{\theta}(x,y)$ for
$0\leq\theta\leq\pi/4$.
\end{lemma}

\begin{proof}
As we have seen the odd eigenvalues correspond to the case $p+q$ odd. 
Assume, without loss of generality, that $p$ is even and $q$ is odd. Then 
the statement follows directly from the relations
\begin{align}
\label{eq:thetared1}\Phi_{p,q}^{\pi-\theta}(x,\pi-y) 
& = \Phi_{p,q}^{\theta}(x,y)\,,\\
\label{eq:thetared2}\Phi_{p,q}^{\pi/2-\theta}(y,x) 
& = \Phi_{p,q}^{\theta}(x,y)\,.
\end{align}
\end{proof}

\begin{remark}
Note that~\eqref{eq:thetared2} holds for all $p$ and $q$, not only for $p+q$ 
odd.
\end{remark}

\section{ The cases $(p,0)$ and $(p,p)$}\label{Section4}

\subsection{The case $(p,0)$}
In this case we are able to calculate exactly the maximum number of nodal
domains. We start with the first non-trivial case.

\begin{lemma}
\label{lem:10}
Let $\Psi_2$ be an eigenfunction corresponding to $\lambda_2=1\,$. 
Then $\mu(\Psi_2)=2$. Moreover, the nodal line will either go from the side 
$y=0$ to the side $y=\pi$ or from the side $x=0$ to the side $x=\pi$ 
(or be a diagonal). In any case it will not be a loop.
\end{lemma}

\begin{proof}
Since $\lambda_2=1$ is the second eigenvalue, it follows directly that 
$\mu(\Psi_2)=2$. The eigenfunction $\Psi_2$ will have the form
\[
\Psi_2(x,y)=\Phi_{1,0}^\theta(x,y)
=\cos\theta\cos x + \sin\theta\cos y\,,\quad 0\leq \theta<\pi\,.
\]
If $\theta\notin\{0,\pi/4,\pi/2,3\pi/4\}$ then
\begin{align}
\label{eq:1-0-1}
\Psi_2(x,0)=0 &\iff \cos x = -\tan\theta\,,\\
\label{eq:1-0-2}
\Psi_2(x,\pi)=0 & \iff \cos x=\tan\theta\,,\\
\label{eq:1-0-3}
\Psi_2(0,y)=0 &\iff \cos y = -\cot\theta\,,\\
\label{eq:1-0-4}
\Psi_2(\pi,y)=0 &\iff \cos y = \cot\theta\,.
\end{align}
If $0< \theta<\pi/4$ or $3\pi/4<\theta<\pi\,$, then the equations~\eqref{eq:1-0-1}
and~\eqref{eq:1-0-2} has exactly one solution each, and the 
equations~\eqref{eq:1-0-3} and~\eqref{eq:1-0-4} has no solutions. 
If $\pi/4<\theta<\pi/2$ or $\pi/2<\theta<3\pi/4\,$, then the opposite situation 
holds.
In the remaining cases the nodal lines are just straight lines. If $\theta=0$ 
then the nodal line is just $x=\pi/2$. If $\theta=\pi/4$ then it is $y=\pi-x$.
If $\theta=\pi/2$ then it is $y=\pi/2$ and if $\theta=3\pi/4$ then it is $y=x\,$.
\end{proof}

\begin{lemma}
\label{lem:p0}
Assume that $(\lambda,\Psi)$ is an eigenpair of $L$ of multiplicity two, 
corresponding to $(p,0)$  and $(0,p)$. Then
\[
\mu(\Psi)\leq
\begin{cases}
\frac{(p+1)^2}{2}\,, & \text{if $p$ is odd,}\\
\frac{(p+1)^2+1}{2}\,, & \text{if $p$ is even.}
\end{cases}
\]
Moreover, in each situation, equality holds for some function $\Psi$ in the 
eigenspace.
\end{lemma}

\begin{proof}
The case $(0,0)$ is clear since then the eigenfunction is just a constant, 
having one nodal domain. The case $(1,0)$  (and $(0,1)$) was taken care of 
in Lemma~\ref{lem:10}.

For $(p,0)$, $p>1$, the eigenfunctions looks like
\[
\Psi(x,y)=\Phi_{p,0}^\theta(x,y)
=\cos\theta\cos px + \sin\theta\cos py,\quad 0\leq\theta<\pi.
\]
Note that, for all $\theta$, the function $\tilde\Psi(x,y):=\Psi(x/p,y/p)\,$,
$0<x<\pi$, $0<y<\pi$ is exactly the function in the eigenspace corresponding 
to the case $(1,0)$, whose nodal domains we know of from Lemma~\ref{lem:10}. 
The function $\Psi(x,y)$ is reconstructed by taking its
values in the square $0<x<\pi/p$, $0<y<\pi/p\,$, and then ``folding'' it evenly
over the whole square. Indeed, for integers $k$,
\[
\begin{aligned}
\Psi(k\pi/p+x,y)&=\Psi(k\pi/p+(2\pi/p-x),y)\,,\quad\text{and}\\
\Psi(x,k\pi/p+y)&=\Psi(x,k\pi/p+(2\pi/p-y))\,.
\end{aligned}
\]

If $\theta\notin\{\pi/4,3\pi/4\}$ then the $\tilde\Psi$ has only 
one nodal line, going from one side to its opposite side. When folding, 
this results in exactly $p+1$ nodal domains. See Figure~\ref{fig:3-0-1} 
for a typical case.

If $\theta=\pi/4$, then, in the square $0<x<\pi/p\,$, $0<y<\pi/p\,$, $\tilde\Psi$ 
has one nodal line, $y=\pi/p-x$. Folding this square gives indeed 
$(p+1)^2/2$ nodal
domains if $p$ is odd, and $((p+1)^2+1)/2$ nodal domains if $p$ is even. This
is illustrated (as the left pictures) in Figure~\ref{fig:3-0-special} 
and~\ref{fig:4-0-special}.

If $\theta=3\pi/4$ then, in the square $0<x<\pi/p\,$, $0<y<\pi/p\,,$ $\tilde\Psi$ 
has one nodal line, $y=x\,$. Folding this square gives indeed $(p+1)^2/2$ nodal
domains if $p$ is odd, and $((p+1)^2-1)/2$ nodal domains if $p$ is even. This
is illustrated (as the right pictures) in Figures~\ref{fig:3-0-special} and~\ref{fig:4-0-special}.
\end{proof}

\begin{figure}[htp]
\centering
\includegraphics[width=3cm]{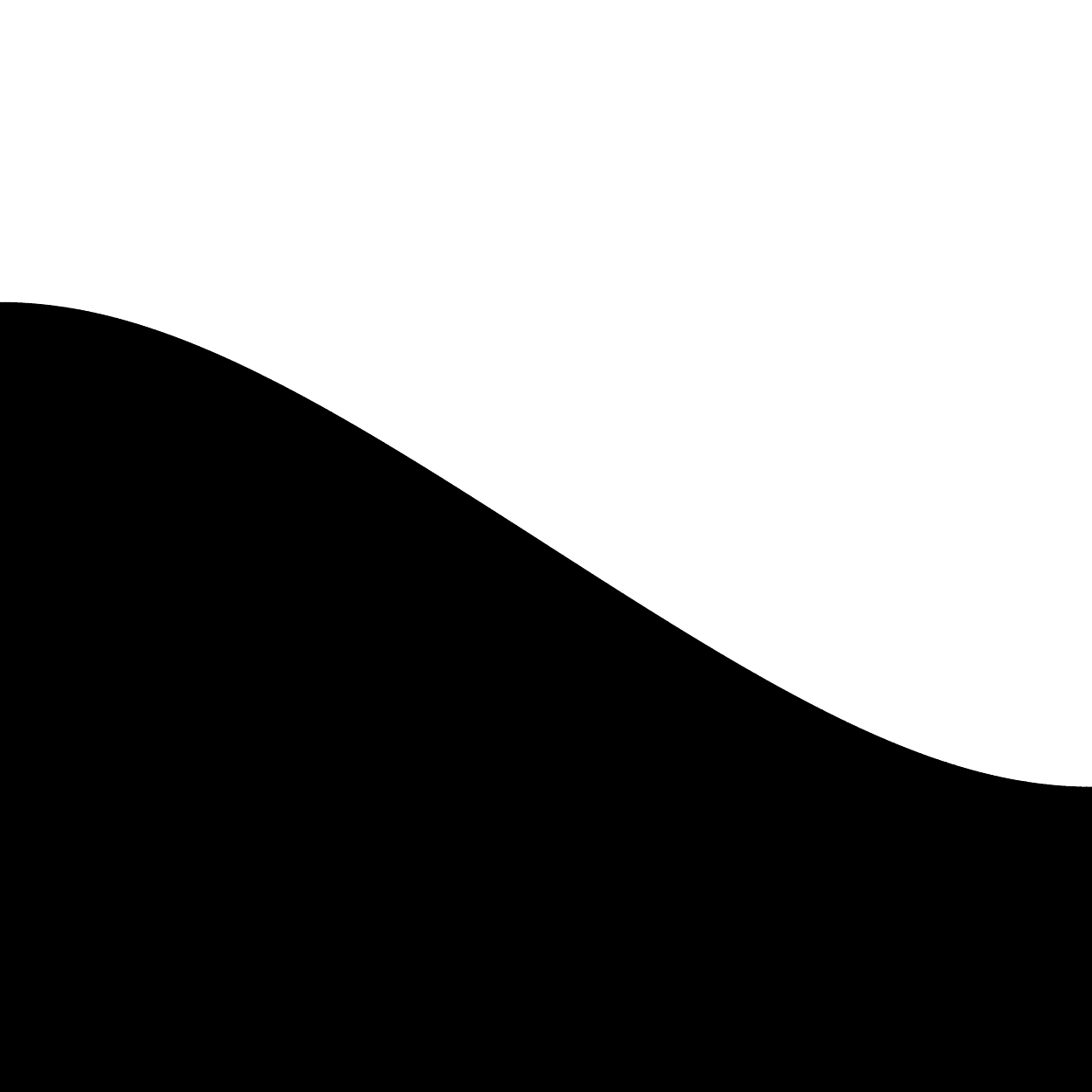}
\hskip 2cm
\includegraphics[width=3cm]{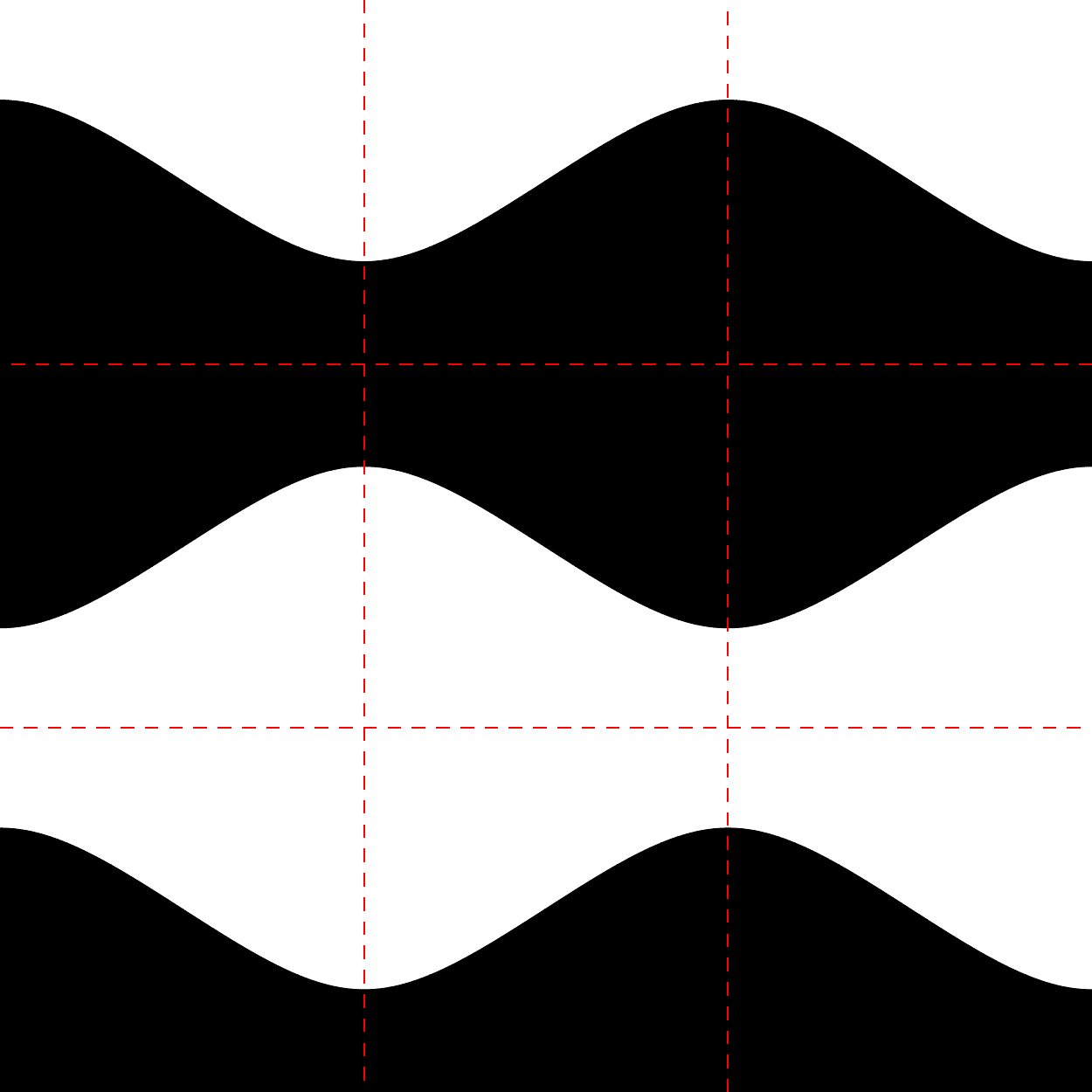}
\caption{Left $(p,q)=(1,0)$, right $(p,q)=(3,0)$. In each case $\theta=1$. Note
how the eigenfunction in the case $(3,0)$ is constructed by folding the
(scaled) $(1,0)$ eigenfunction evenly.}
\label{fig:3-0-1}
\end{figure}

\begin{figure}[htp]
\centering
\includegraphics[width=3cm]{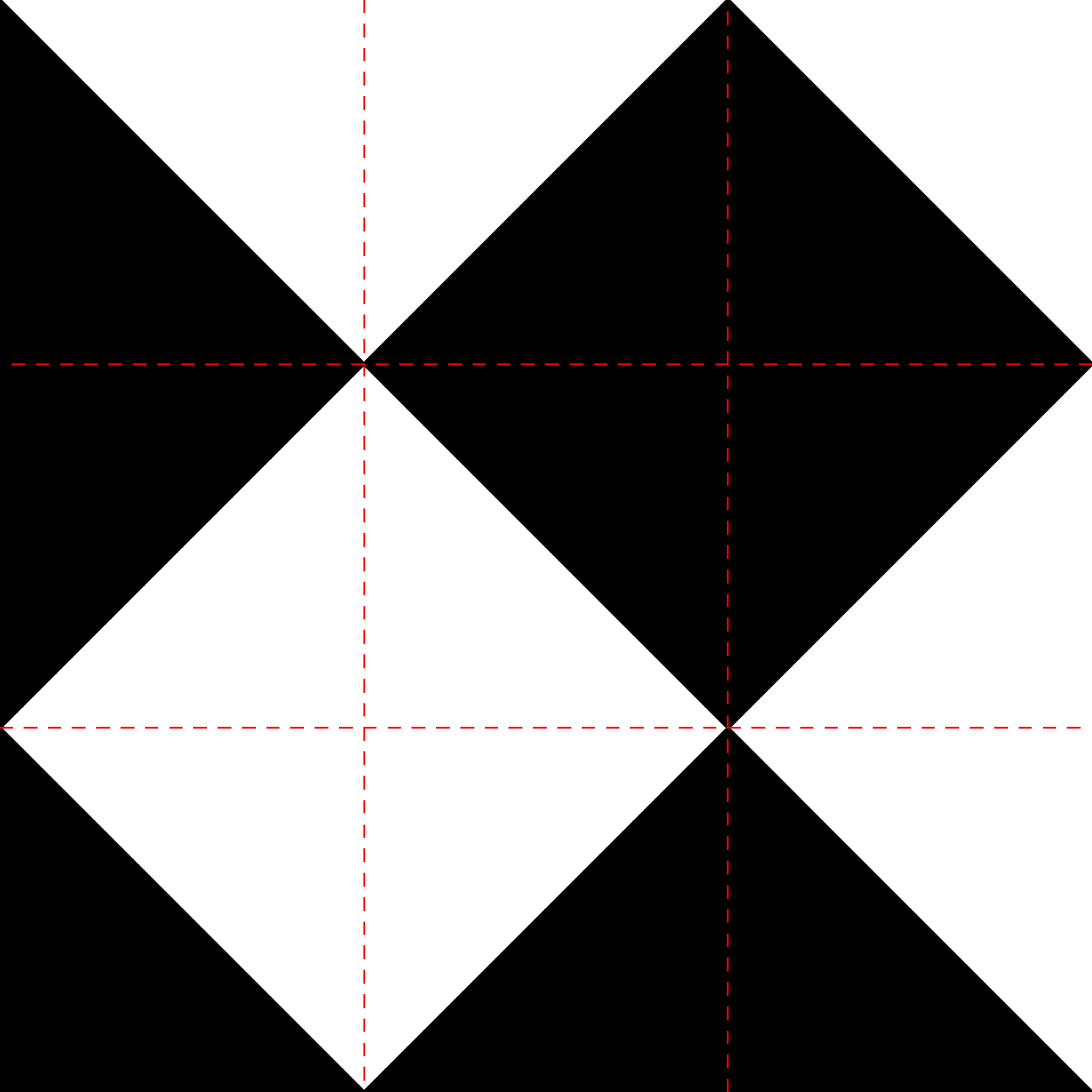}
\hskip 2cm
\includegraphics[width=3cm]{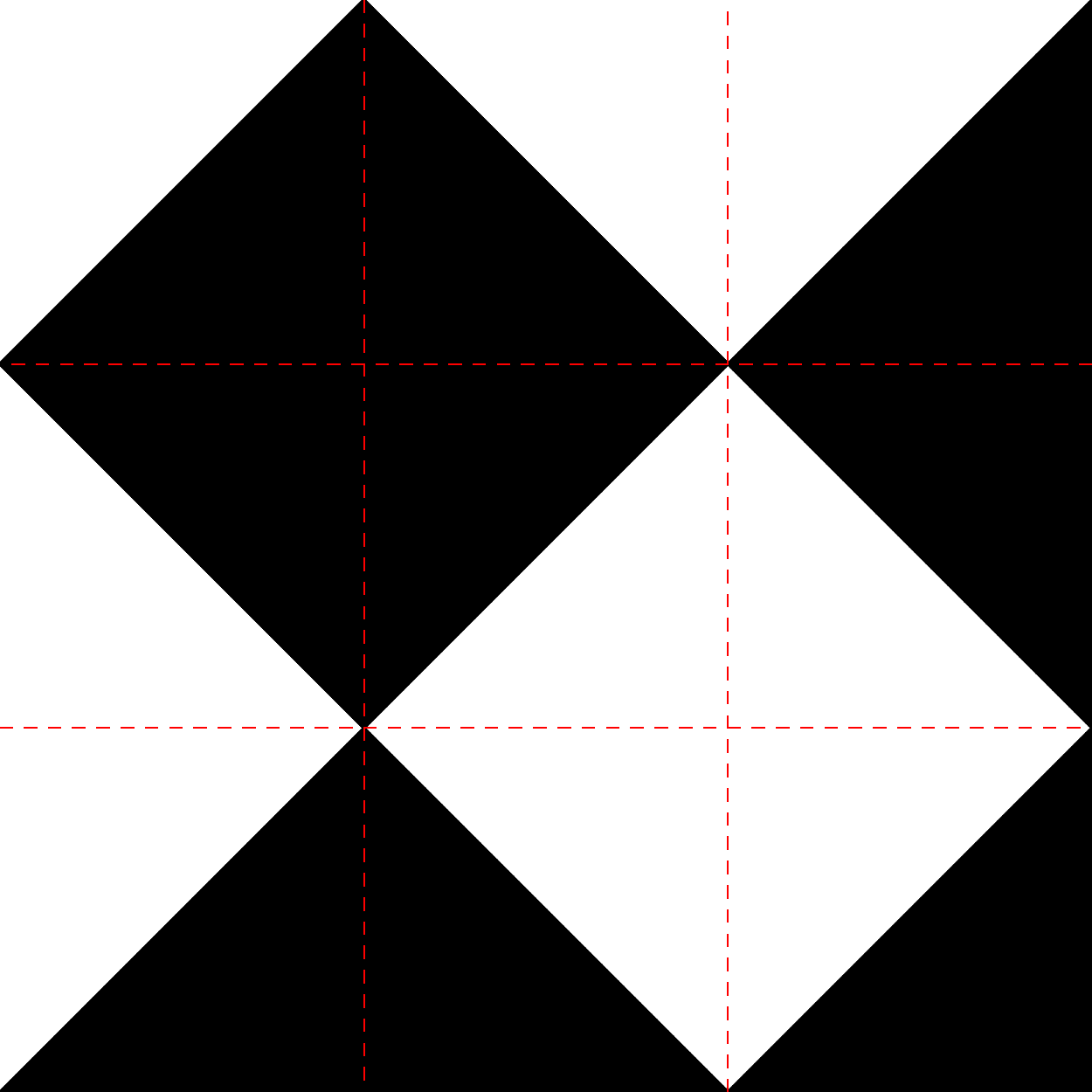}
\caption{Left $(p,q)=(3,0)$ with $\theta=\pi/4$, right $(p,q)=(3,0)$ with $\theta=3\pi/4$. Both these cases give the maximal cardinal of  $8$ nodal domains.}
\label{fig:3-0-special}
\end{figure}

\begin{figure}[htp]
\centering
\includegraphics[width=3cm]{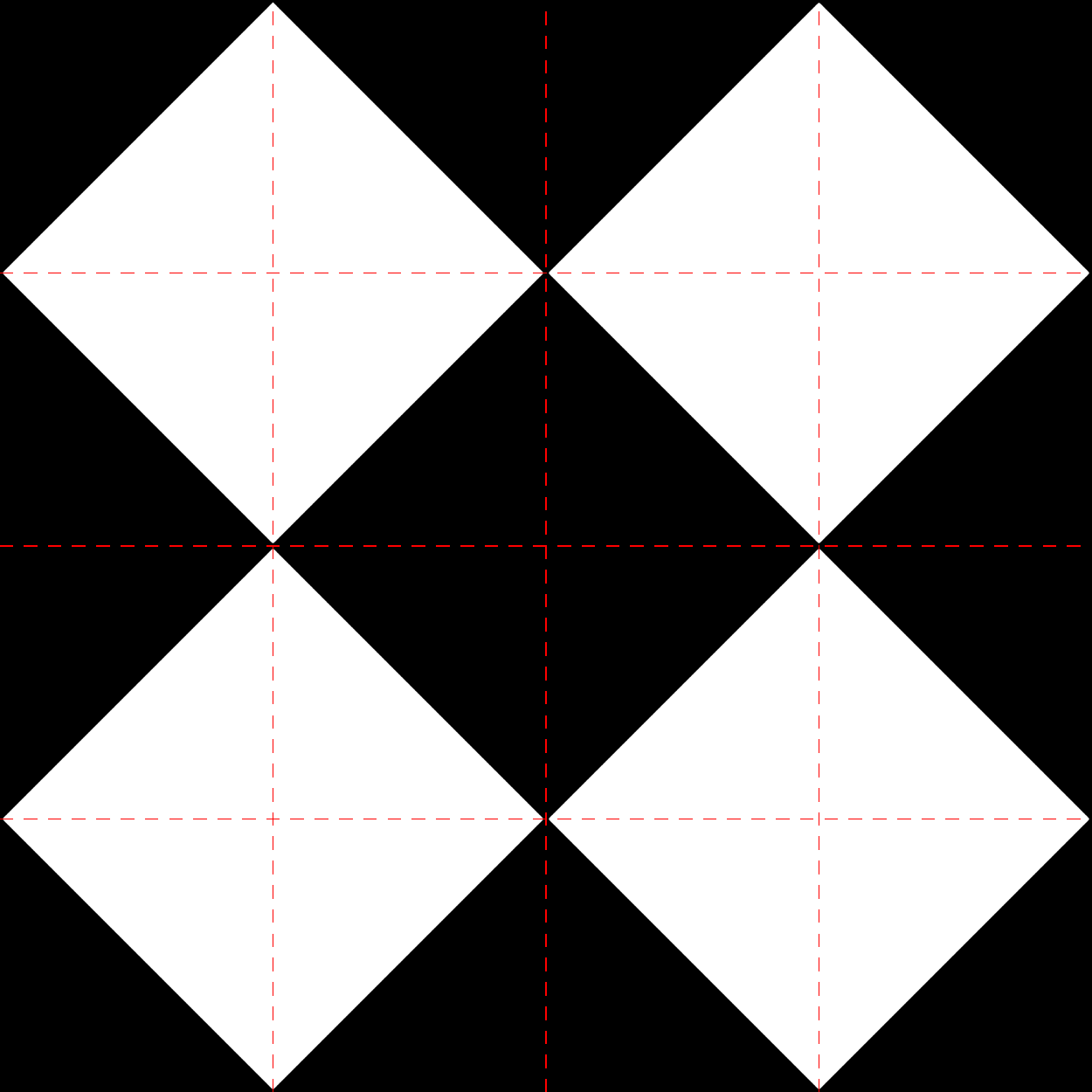}
\hskip 2cm
\includegraphics[width=3cm]{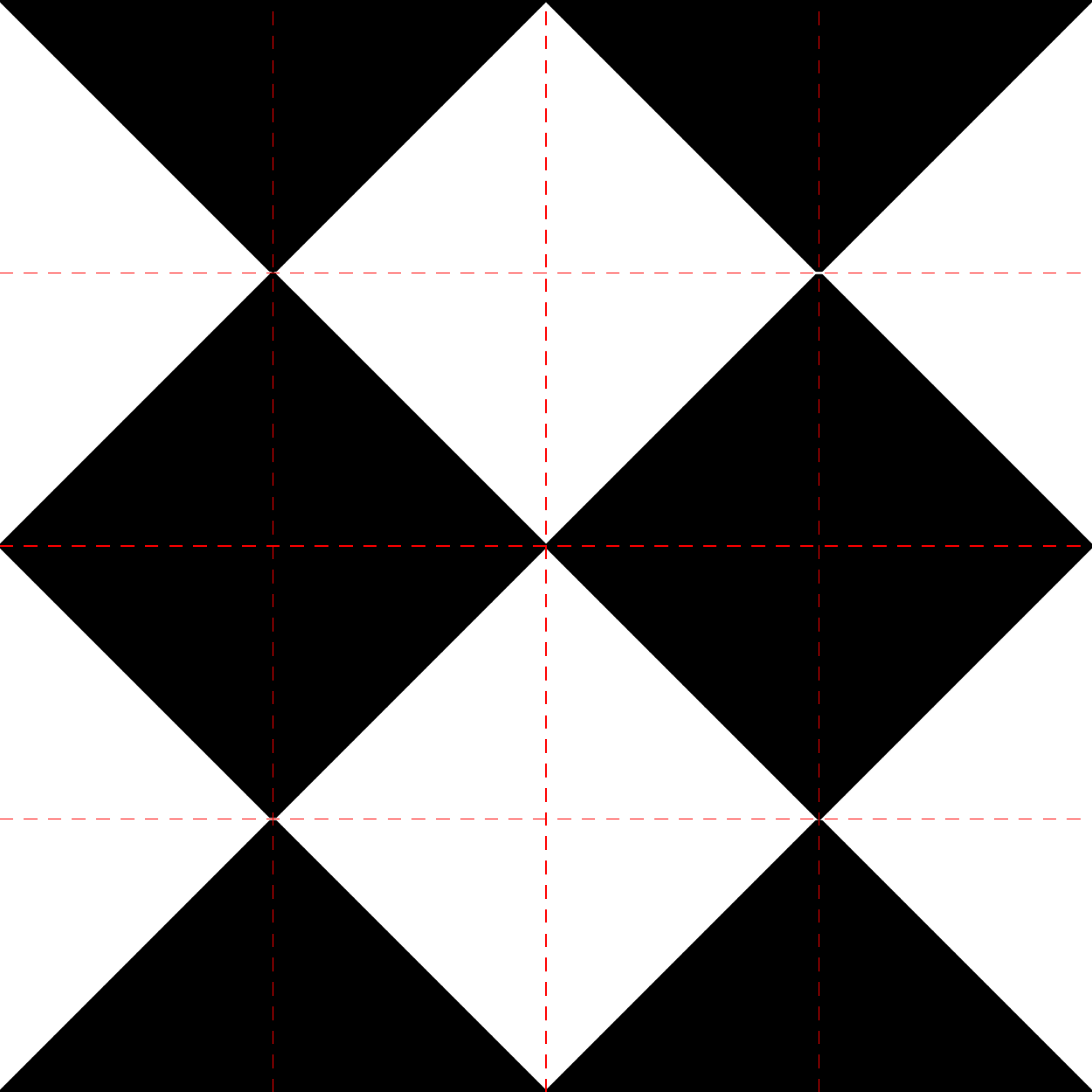}
\caption{Left $(p,q)=(4,0)$ with $\theta=\pi/4$, right $(p,q)=(4,0)$ with $\theta=3\pi/4$. Note that $\theta=\pi/4$ gives $13$ nodal domains, which is the maximal cardinal, while $\theta=3\pi/4$ gives $12$ nodal domains, only.}
\label{fig:4-0-special}
\end{figure}

\begin{cor}
\label{cor:p0}
The eigenvalues
$\lambda_{10}=\lambda_{11}$,
$\lambda_{16}=\lambda_{17}$,
$\lambda_{34}=\lambda_{35}$,
$\lambda_{44}=\lambda_{45}$,
$\lambda_{57}=\lambda_{58}$,
$\lambda_{72}=\lambda_{73}$,
$\lambda_{105}=\lambda_{106}$,
$\lambda_{122}=\lambda_{123}$, and
$\lambda_{167}=\lambda_{168}$
are not Courant sharp.
\end{cor}

\subsection{The case $(p,p)$}
\begin{lemma} 
\label{lem:pp}
 If the eigenspace corresponding to $\lambda=2p^2$ is one-dimensional
and $\Psi$ is an eigenfunction corresponding this eigenvalue, 
then $\mu(\Psi)=(p+1)^2$.
\end{lemma}
\begin{proof}
The eigenspace is spanned by $\cos px\cos py$, which
is a true product of a function depending on $x$ and one that depends on $y$.
Each of them has $p$ zeros, and thus the number of nodal domains equals 
$(p+1)^2$.
\end{proof}

\begin{cor}
\label{cor:pp}
The eigenvalues
$\lambda_{18}$,
$\lambda_{32}$,
$\lambda_{65}$,
$\lambda_{86}$,
$\lambda_{113}$,
$\lambda_{140}$, and
$\lambda_{206}$
are not Courant sharp.
\end{cor}

\section{Critical points, boundary points and the chessboard localization}
\label{Section5}
The reasoning below depends on the fact that the number of nodal domains of a continuous curve 
of eigenfunctions $\Psi_t$ is constant unless there are interior stationary points  appearing 
in the zero-set, i.e. $(x,y)\in\Omega$ such that
\begin{equation}
\label{eq:stationary}
\Psi(x,y)=0,\quad \partial_x \Psi(x,y)=0,\quad \partial_y \Psi(x,y)=0\,,
\end{equation}
or changes in the cardinal of the boundary points. We refer for this point to Lemma 4.4 in \cite{Ley0}. 
Hence the analysis of these situations is quite important.

\subsection{Critical points}
\begin{lemma}
\label{Lemma5.1}
If $p+q$ is odd, and $0<\theta\leq\pi/4$, then $\Phi_{p,q}^{\theta}$ 
satisfies~\eqref{eq:stationary} at the point $(x,y)$ in $\Omega$ if and only if
\begin{equation}
\label{eq:taneq}
p\tan px = q\tan qx\,,\quad \text{and}\quad p\tan py=q\tan qy\,.
\end{equation}
If these two equations are fulfilled, we recover the critical
value of $\theta$ via
\begin{equation}
\label{eq:thetaeq}
\tan\theta = -\frac{\cos qx\cos py}{\cos px\cos qy}\,.
\end{equation}
\end{lemma}

\begin{proof}

The eigenfunctions have the form
\[
\Phi_{p,q}^{\theta}(x,y)
=\cos \theta \cos q x \cos p y + \sin \theta \cos qy \cos p x\,.
\]
The zero critical points are determined   by
\[
\begin{aligned}
\cos \theta \cos q x \cos p y + \sin \theta \cos qy \cos p x &=0\,,\\
q  \cos \theta  \sin qx \cos p y + \sin \theta p \cos qy \sin p x &=0\,,\quad\text{and}\\
p \cos \theta \cos qx \sin p  y + q  \sin \theta \sin q y \cos p x &=0\,.
\end{aligned}
\]
Since $0<\theta\leq \pi/4$, the critical points $(x,y)$
should satisfy~\eqref{eq:taneq}.
(Note that one can reduce by dilation to the case when $p$ and $q$ are mutually prime.) 

Once a pair satisfying these two conditions is given, we recover the corresponding critical values of $\theta$  by \eqref{eq:thetaeq}.
\end{proof}

\subsection{Boundary points}
\label{ss5.2}
The intersection of the zero set of  $\Phi^\theta_{p,q}$ with the boundary is determined by the equations
\begin{equation}\label{bdp}
\begin{aligned}
\Phi^\theta_{p,q}(0,y)=0 &\iff \frac{\cos p y}{\cos q y}=-\cot\theta,\\
\Phi^\theta_{p,q}(\pi,y)=0 &\iff \frac{\cos py}{\cos q y}=\cot\theta,\\
\Phi^\theta_{p,q}(x,0)=0 &\iff \frac{\cos p x}{\cos q x}=-\tan\theta,\\
\Phi^\theta_{p,q}(x,\pi)=0 &\iff \frac{\cos px}{\cos q x}=\tan\theta,\\
\end{aligned}
\end{equation}

Outside the zeros of $x\mapsto \cos qx$, the function 
\begin{equation}\label{deffpq}
f_{p,q} (x)=\frac{\cos px}{\cos qx}\,
\end{equation}
has derivative 
\[
f_{p,q}'(x)= \frac{- p \sin px  \cos q x  +q  \cos p x \sin q x}{\cos ^2 q x}\,.
\]
Is it easy from the graph of $f_{p,q}$ to count for a given $\theta$  the number 
of points arriving at the boundary. Except at the corners 
(see Remark~\ref{remcor} below) this number is changing at the critical values 
of $f_{p,q}$ (see Remark~\ref{rem4.4}), which are the solutions of 
\begin{equation}
\label{pt=qt}
p \tan px = q \tan qx\,.
\end{equation}

\subsection{Guide for the case by case analysis}
In the case by case analysis we will always have in mind the following remarks.
\begin{remark} \label{rem4.4}
The critical points at the boundary (except for the corners) are given 
by $(0,x_i)$, $(x_i,0)$, $(\pi, x_i)$ and $(x_i,\pi)$ where $x_i$ is some non 
zero solution of \eqref{pt=qt}. The corresponding critical $\theta$ is 
obtained by equation \eqref{bdp}.  \end{remark}
\begin{remark}\label{remcor}
The values for 
which lines arrive to the corner correspond only to $\theta=\frac \pi 4$ and 
$\theta =\frac {3\pi}4$.
\end{remark}

\begin{remark}\label{rmcrtbnd}
The solutions of $p \tan px =q \tan qx$ can also be obtained by looking 
at the local extrema of  $f_{p,q}\,$.
\end{remark}

\begin{remark}\label{lmutuallyprime}
The analysis of the solutions of  $p \tan px =q \tan qx$ can (by a change of 
variable) be reduced to the case when $p$ and $q$ are mutually prime.
\end{remark}

From these remarks, we deduce that for a complete analysis of the nodal patterns 
corresponding to a pair $(p,q)$ such that the eigenvalue  $p^2 + q^2$ has multiplicity $2$, we 
should first analyze the graph of the function $f_{p,q}\,$. 
This will not only permit to count in function of $\theta$ the number of lines 
touching the boundary but will also permit to determine by the analysis of the 
local extrema to determine the critical value of $\theta$ for which we have 
critical points. 

\subsection{Chessboard argument and applications}\mbox{}
This idea was proposed by A. Stern~\cite{St} and used intensively 
and more rigorously in~\cite{BH,BHSphere,BH2}.
We consider a pair $(p,q)$ with $(p,q)$ mutually prime and $p+q$ odd and  
would like to localize the zeros of $\Phi_{p,q}^\theta$ for say 
$\theta \in (0,\frac \pi 2)$. It is based on a very elementary observation. 
We simply observe that if $\cos p x \cos p y \cos qx \cos qy >0$, then 
$\Phi^\theta \neq 0$. This determines the 
``white'' rectangles of a chessboard.  These rectangles are obtained by 
drawing the vertical lines $\{x= k \frac{\pi}{2p}\}$ ($k$ odd)
and  $\{x= k' \frac{\pi}{2q}\}$ ($k'$ odd), and similarly the 
horizontal lines $\{y= k \frac{\pi}{2p}\}$ ($k$ odd) and  
$\{y= k' \frac{\pi}{2q}\}$ ($k'$ odd), and hence the zero set should 
be contained in the closed ``black'' rectangles corresponding to the closure 
of the set  $\{\cos p x \cos p y \cos qx \cos qy <0\,\}$.  Note that 
these rectangles have different size. It is also important to determine 
which points at the boundary of a given rectangle belongs to the zero set. 
They are obtained  by the equations 
$(x,y)=(k \frac {\pi} {2p}, k'\frac {\pi} {2p})$ and for 
$(x,y)= (\ell \frac {\pi} {2q}, \ell' \frac{ \pi}{ 2q})$ for 
$k$, $\ell$, $k'$, $\ell'$ odd. So the nodal set should contain all these 
points and only these points.  We call these points \emph{admissible corners}. 
This can also be seen as a consequence of the fact that $\cos p x$ and $\cos qx $ 
have no common zero in $[0,\pi]$ when $p+q$ is odd. Hence we have proved.

\begin{lemma}
\label{lemma5.6}
If $p$ and $q$ are mutually prime and $p+q$ is odd, and 
$\theta \in (0,\frac \pi 2)$ then the only points of intersection of  the 
zero set of $\Phi_{p,q}^\theta$ with the boundary of a black rectangle are 
the admissible corners.

Moreover, these points are regular points of the zero set.
\end{lemma}
 
Another point is that the zero set cannot contain any closed curve inside a 
black rectangle (we also mean curves touching the boundary). The ground state 
energy inside the curve delimited by the curve (say in the case without double 
points) should indeed be strictly above the ground state energy of the 
rectangle (Dirichlet for a rectangle in the interior, Dirichlet--Neumann when 
the rectangle has at least one size common to the square $[0,\pi]^2$). But the 
minimal energy for these rectangles is $2 \max (p,q)^2 $ (a contradiction with 
the value $p^2+q^2$).

\begin{lemma}
\label{lemma5.7}
If $p$ and $q$ are mutually prime and $p+q$ is odd and 
$\theta \in (0,\frac \pi 2)$ then the zero set of $\Phi_{p,q}^\theta$ 
cannot contain any closed curve contained in a "black" rectangle.
\end{lemma}

\begin{remark}
As a consequence of these two lemmas let us observe that at an admissible corner 
only one curve belonging to the zero set can enter in a black rectangle  and 
that it should either go out  by an admissible corner, either  touch the 
boundary or meet another curve of the zero set at a critical point.
\end{remark}

\section{Special cases}
\label{Section6}
Most of the cases appearing in the table are treated via the general 
considerations of Sections~\ref{Section2} and~\ref{Section3}.  In this section, 
we consider a first list of special cases where a more careful analysis is 
needed, which  involves the analysis of boundary points or of critical points.

\subsection{The case $\lambda_7=\lambda_8=5$ ($(p,q)=(2,1)$)}
The eigenspace is two-dimensional,
\begin{equation}
\label{eq:psi21}
\Phi_{2,1}^{\theta}(x,y)=\cos\theta\cos2x\cos y 
+ \sin\theta\cos x\cos 2y\,,\quad 0\leq \theta<\pi\,.
\end{equation}
We know from Lemma~\ref{lem:antisymmetric} that this case is not Courant sharp, 
but that it has a maximum number of nodal domains being $6$.

\begin{lemma}
\label{lem:2-1}
$\mu(\Phi_{2,1}^0)=6$. If $0<\theta\leq \pi/4$ then $\mu(\Phi_{2,1}^\theta)=4\,$.
\end{lemma}

\begin{proof}
Observing that $\tan 2x = 2\tan x/(1-\tan^2 x)$  it is immediate that 
\[
2\tan 2x = \tan x
\]
has no zero in the open interval $(0,\pi)$. Hence by Lemma~\ref{Lemma5.1},  
there are no  critical points for $\Phi_{1,2}$ in $\Omega$. Having in 
mind~\eqref{bdp}, the analysis of the graph
\[
f_{2,1}(x) = \cos 2x/\cos x = (2 \cos^2 x -1)/\cos x
\]
(see Figure~\ref{fig:2-1-cos}) leads immediately (see Figure~\ref{fig:2-1} for 
$\theta=\frac \pi 8$) to the existence of $4$ 
nodal domains in this case (two boundary points for $y=0$ and $y=\pi$ and 
one boundary point for $x=0$ and $x=\pi$).

\begin{figure}[htbp]
\centering
\includegraphics[width=8cm]{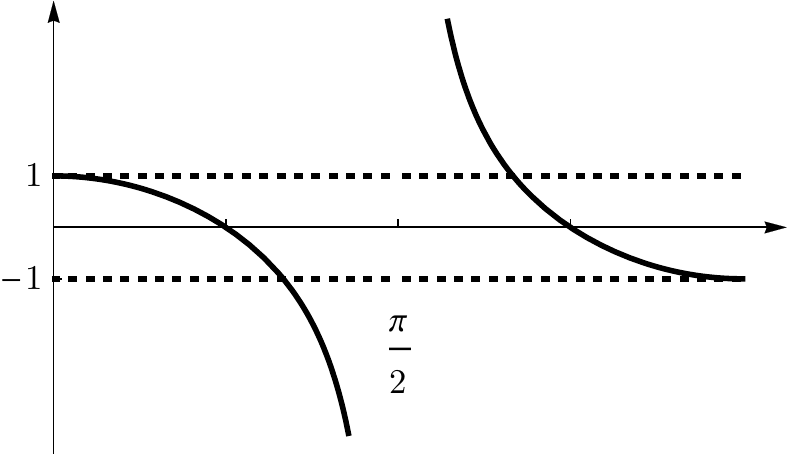}
\caption{The graph of $f_{2,1}(x)=\frac{\cos(2x)}{\cos(x)}$ in the interval 
$0<x<\pi$.}
\label{fig:2-1-cos}
\end{figure}

\begin{figure}[htp]
\centering
\includegraphics[width=3cm]{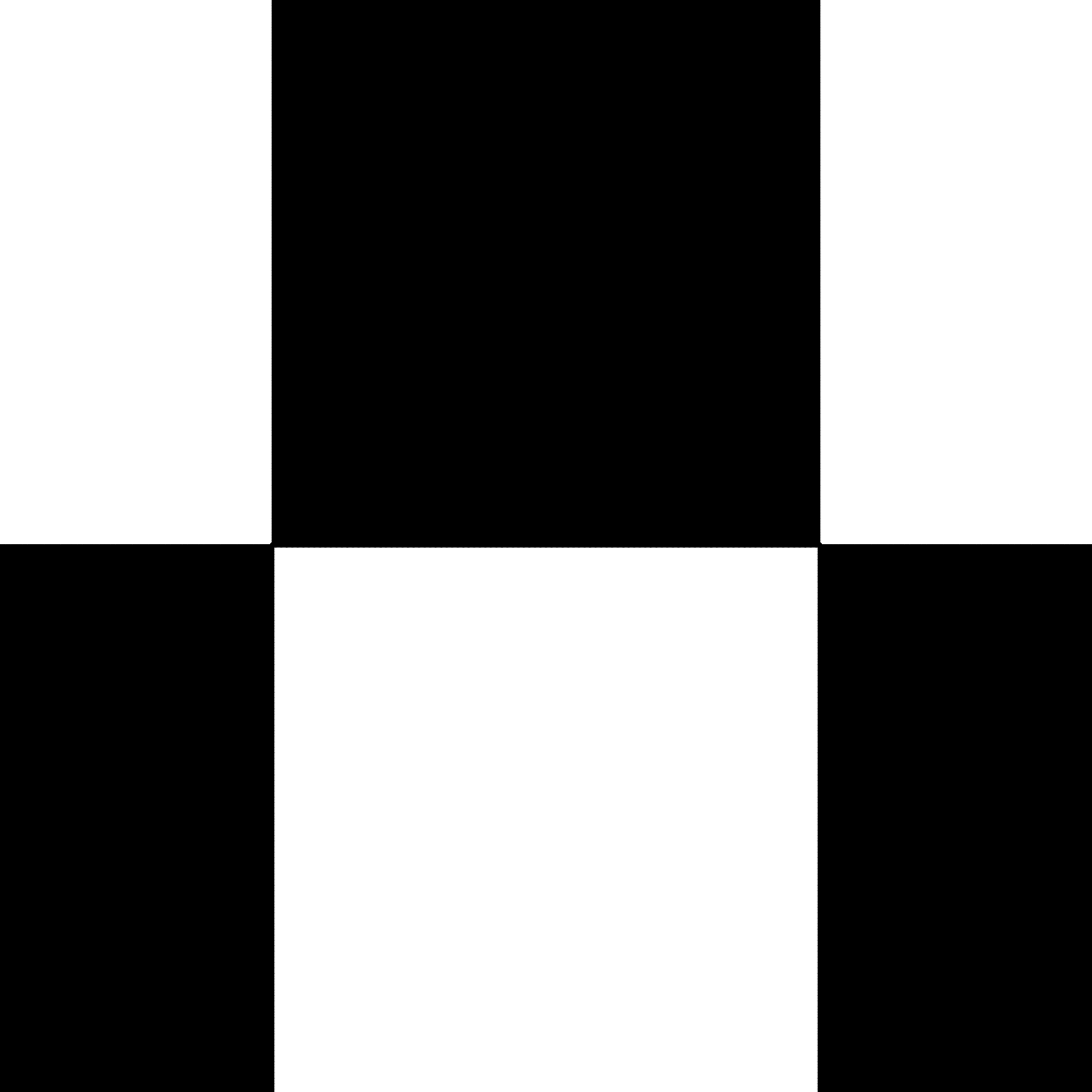}
\hskip 1cm
\includegraphics[width=3cm]{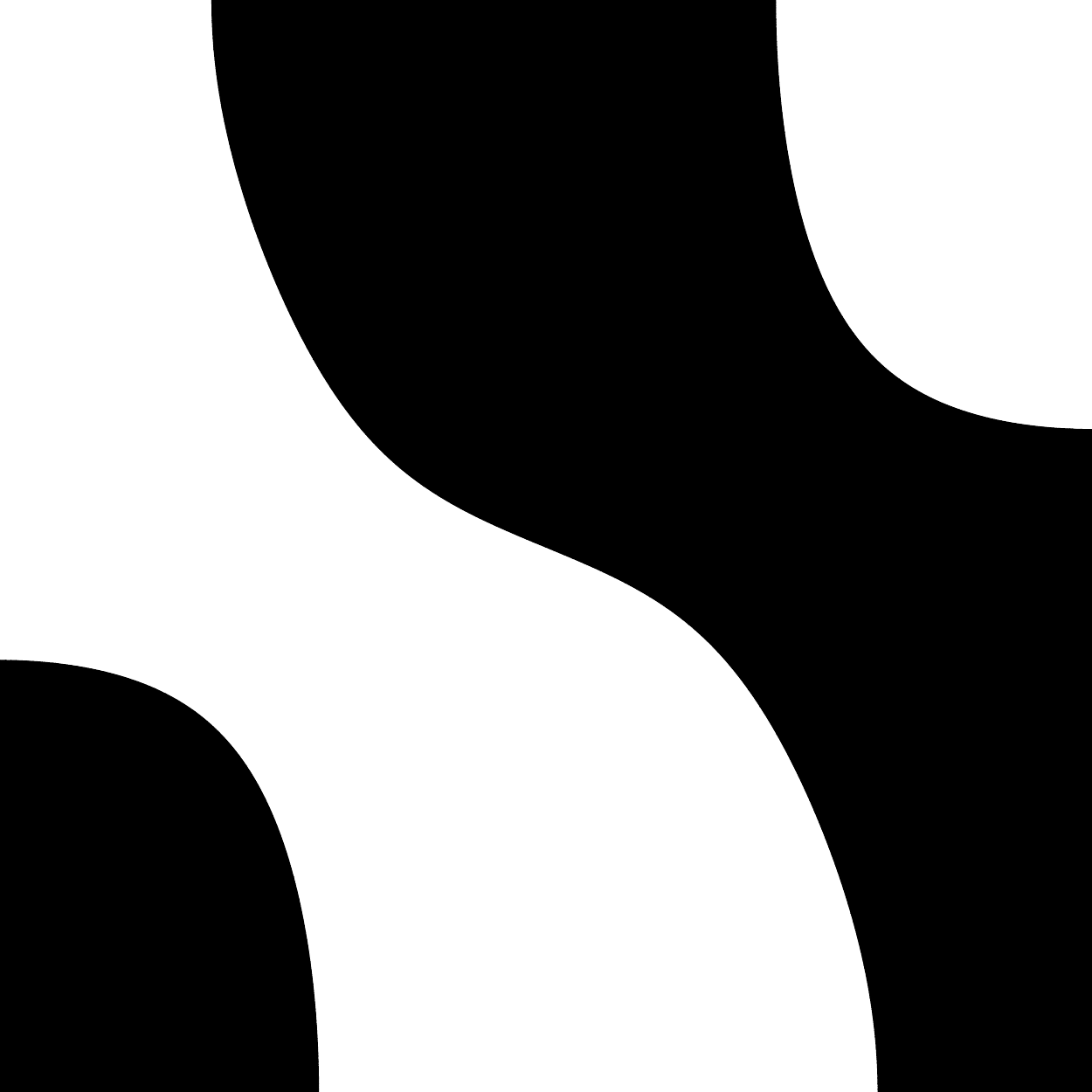}
\hskip 1cm
\includegraphics[width=3cm]{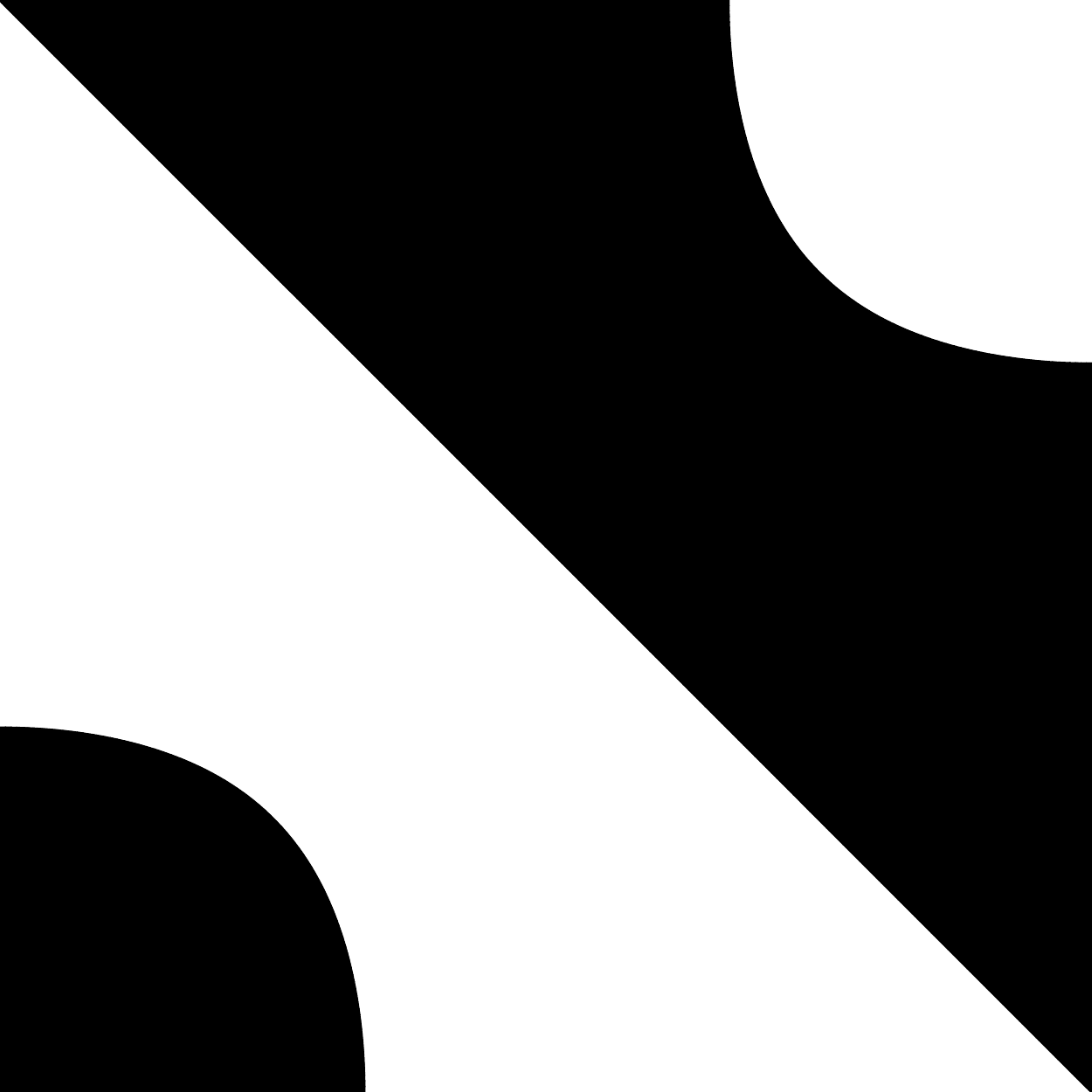}
\caption{Nodal domains for $\Phi_{2,1}^\theta$ when $\theta=0$, $\theta=\pi/8$,
 and $\theta=\pi/4$.}
\label{fig:2-1}
\end{figure}

It remains to consider two cases, $\theta=0$ and  $\theta=\frac{\pi}{4}$.
For $\theta=0$ we are in the product case and have $6$ nodal domains.
For $\theta=\pi/4$ we let $u=\cos x$ and $v=\cos y$, both living in $[-1,1]$. 
Then 
\[
\begin{aligned}
\Phi_{2,1}^{\pi/4}(x,y)=0 &\iff (2u^2-1)v+u(2v^2-1)=0\\
&\iff (u+v)(2uv-1)=0\,.
\end{aligned}
\]
Thus, we get the straight line $u=-v$ and the hyperbola $4uv=1$. We note that
they do not intersect. Thus, there are $4$ nodal domains in this case.
\end{proof}

\subsection{The case $\lambda_{21}=\lambda_{22}=20$ ($(p,q)=(4,2)$)}
\begin{lemma}
\label{lem:4-2}
Assume that $(\lambda_{21},\Psi_{21})$ is an eigenpair of $L$. Then
$\mu(\Psi_{21})\leq 15$. In particular we are not in the Courant sharp 
situation.
\end{lemma}

\begin{proof}
If $\theta\not\in\{0,\pi/2\}$, then Lemmas~\ref{lem:pandqeven} and~\ref{lem:2-1}
give that the number of nodal domains is less than or equal to 
$4\cdot 4-3=13$.
For $\theta\in\{0,\pi/2\}$ we are in the product case, and have $15$ nodal domains.
\end{proof}

\subsection{The case $\lambda_{70}=\lambda_{71}=80$ ($(p,q)=(8,4)$)}
\begin{lemma}
\label{lem:8-4}
Assume that $(\lambda_{70},\Psi_{70})$ is an eigenpair of $L$. Then
$\mu(\Psi_{70})\leq 57\,$.
In particular we are not in the Courant sharp situation.
\end{lemma}

\begin{proof}
If $\theta\not\in\{0,\pi/2\}$, then Lemmas~\ref{lem:pandqeven} and~\ref{lem:4-2}
imply that the number of nodal domains is less than or equal to 
$4\cdot 15-3=57\,$. For $\theta\in\{0,\pi/2\}$ we are in the product case, 
and have $45$ nodal domains.
\end{proof}

\subsection{The case $\lambda_{42}=\lambda_{43}=45$ ($(p,q)=(6,3)$)}
\begin{lemma}
\label{lem:6-3}
Assume that $(\lambda_{42},\Psi_{42})$ is an eigenpair of $L$. Then
$\mu(\Psi_{42})\leq 36\,$. In particular, the eigenpair $(\lambda_{42},\Psi_{42})$
is not Courant sharp.
\end{lemma}

\begin{proof}
Fix $\theta\in[0,\pi)$.
It holds that $\Phi_{6,3}^\theta(x,y)=\Phi_{2,1}^\theta(3x,3y)$ in the square
$\{(x,y)~|~0<x<\pi/3,\ 0<y<\pi/3\}$.  
Moreover, the values of $\Phi_{6,3}^\theta$ in the rest of $\Omega$ are 
recovered by ``folding evenly''. Thus, the number of nodal domains are less 
then or equal to (in fact, this could be made sharper) nine times the number of 
nodal domains of $\Phi_{2,1}^\theta\,$.

From Lemma~\ref{lem:2-1} it follows that if $\theta\not\in\{0,\pi/2\}$ then
\[
\mu(\Phi_{6,3}^\theta)\leq 9\times 4=36\,.
\]
If $\theta\in\{0,\pi/2\}$ then we 
are in the product case and have $28$ nodal domains.
\end{proof}

\subsection{The case $\lambda_{14}=\lambda_{15}=13$ ($(p,q)=(3,2)$)}
Here the situation is similar to $(p,q)=(2,1)$, but with one additional
nodal domain touching each part of the boundary. The general eigenfunction
is
\begin{equation}
\label{eq:psi32}
\Phi_{3,2}^\theta (x,y)=\cos\theta\cos3x\cos 2y
+ \sin\theta\cos 2x\cos 3y\,,\quad 0\leq \theta<\pi\,.
\end{equation}

\begin{lemma}
\label{lem:3-2}
If $\theta=0$ then $\mu(\Phi_{3,2}^0)=12$. If $0<\theta\leq\pi/4\,$, 
$\mu(\Phi_{3,2}^{\theta})=6\,$. In particular, if $(\lambda_{14},\Psi_{14})$ is an
eigenpair of $L$, it is not Courant sharp.
\end{lemma}

\begin{proof}
By Lemma~\ref{Lemma3.8} it is sufficient to
consider $0\leq \theta\leq \pi/4\,$.

If $\theta=0$ we are in the product case and there are exactly
$12$ nodal domains.

Assume that $0<\theta\leq \pi/4$. Let us start to count the nodal
lines that touch the boundary. For this, we use \eqref{bdp} with $(p,q)=(3,2)$.
The function $f_{3,2}(x)=\frac{\cos3x}{\cos 2x}$, 
$x\in[0,\pi]\setminus\{\pi/4,3\pi/4\}$ has derivative 
\[
f_{3,2}'(x)=-\frac{(3+\cos 2x+\cos 4x)\sin x}{\cos^2 2x}. 
\]

\begin{figure}[htbp]
\centering
\includegraphics[width=8cm]{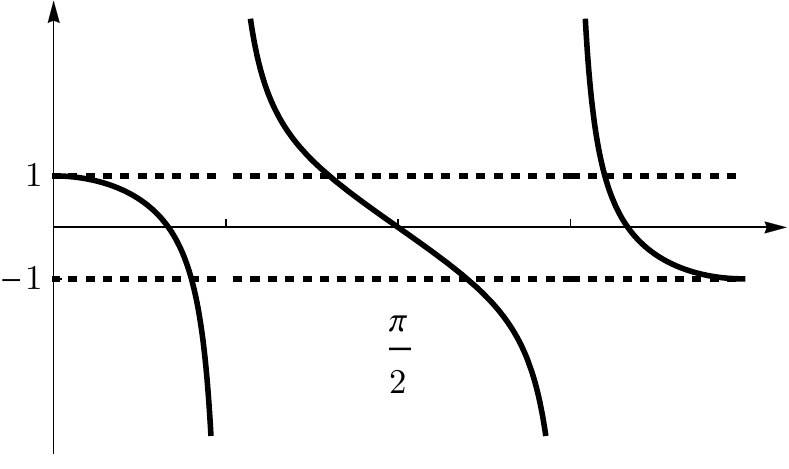}
\caption{The graph of $f_{3,2}(x)=\frac{\cos(3x)}{\cos(2x)}$ in the interval 
$0<x<\pi$.}
\label{fig:3-2-cos}
\end{figure}

In particular, $f_{3,2}'$ is negative where it is defined. 
We find immediately that $f_{3,2}$ attains 
all values in $[-1,1]$ three times and  all values in 
$\mathbb{R}\setminus\{[-1,1]\}$ twice.

Using Lemma~\ref{Lemma5.1} and  Remark~\ref{rmcrtbnd} and the fact that 
$f_{3,2}$ has no critical points in its domain of definition, we get that 
we have no interior critical points of $\Phi_{3,2}^{\theta}$.
We conclude (see for example Figure \ref{fig:3-2} for $\theta =\frac \pi 8$) 
that if $0<\theta<\pi/4$ then we have exactly
three nodal lines touching each part of the boundary of $\Omega$ where $y=0$ 
and $y=\pi$, and exactly two nodal lines on each of the parts of the boundary 
of $\Omega$ where $x=0$ and $x=\pi$. 

Thus, we only have to consider the case $\theta=\pi/4$.
If $\theta=\pi/4$ then we have one nodal line $y=\pi-x$ and two additional
nodal lines touching each part of the boundary.

Next, we eliminate the case of loops in the zero set by 
using Lemma~\ref{lemma5.7}.  In this simple case, we can do a more 
algebraic computation. We make the substitution
$u=\cos x$ and $v=\cos y$ and find that $\Phi_{3,2}^{\pi/4}(x,y)=0$ if
and only if
\[
(u+v)\bigl(8u^2v^2-4u^2-4v^2 - 2uv+3 \bigr)=0\,.
\]
One solution is $u=-v$. Let $F(u,v)=8u^2v^2-4u^2-4v^2-2uv+3\,$. Then
$F'_v=16u^2v-8v-2u$ and hence
\[
F'_v = 0 \iff v=\frac{u}{4(2u^2-1)}\,,\quad u\neq\pm1/\sqrt{2}\,.
\]
It is easily seen that $F'_v(\pm1/\sqrt{2},v)=\mp \sqrt{2}\neq 0\,$. Next,
\[
F\Bigl(u,\frac{u}{4(2u^2-1)}\Bigr)= \frac{32u^4-39u^2+12}{4(1-2u^2)}\,,
\]
and since $1521=39^2<4\cdot32\cdot12=1536$ we find that $F$ and $F'_v$ have
no common zeros. Thus, the nodal lines are never vertical. By symmetry in
$u$ and $v$ it follows that they are never horizontal either.

All in all, this means that for $\theta=\pi/4$ (and thus for 
$0<\theta\leq\pi/4$) we have five non-intersecting nodal lines, and 
so six nodal domains.
\end{proof}

\begin{figure}[htp]
\centering
\includegraphics[width=3cm]{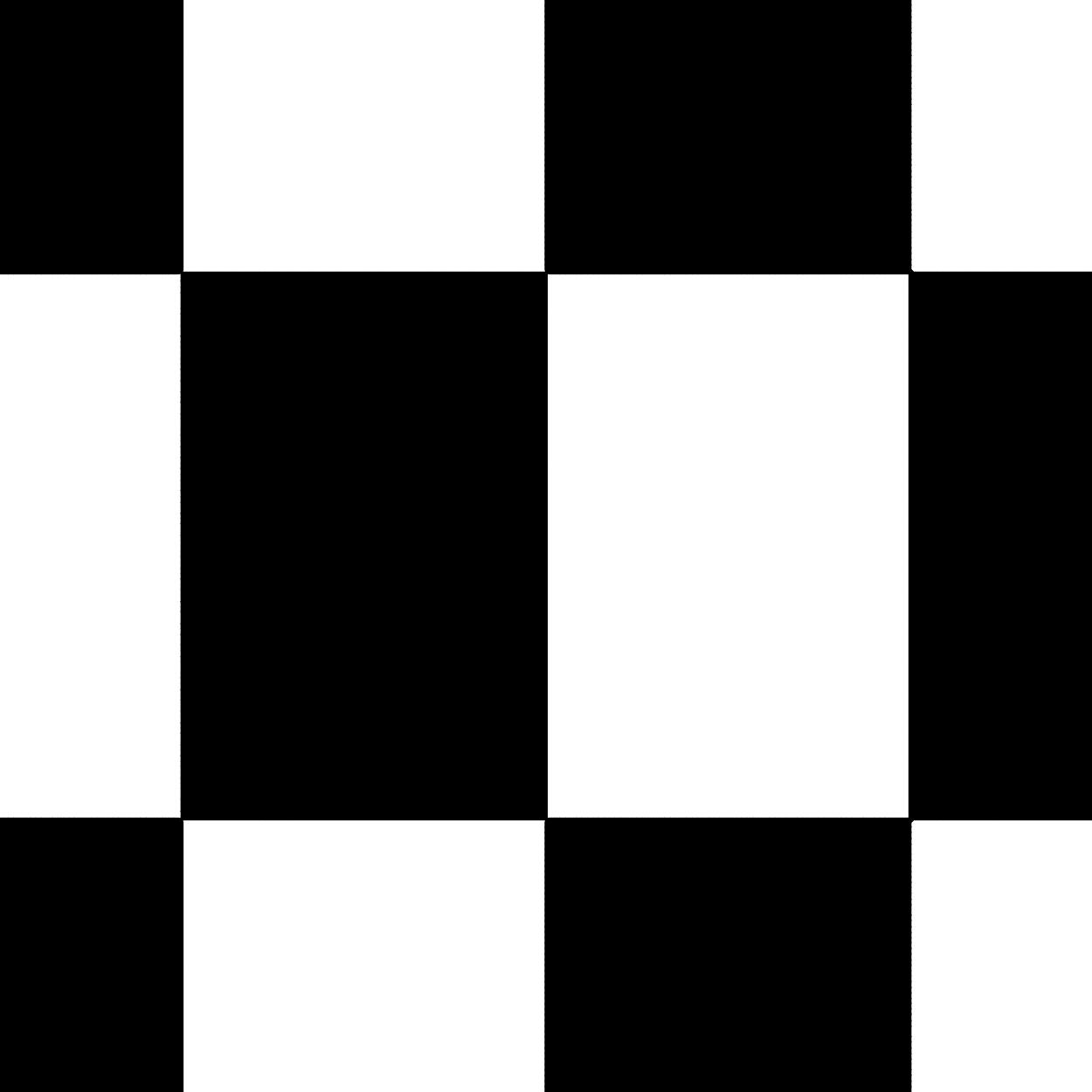}
\hskip 0.5cm
\includegraphics[width=3cm]{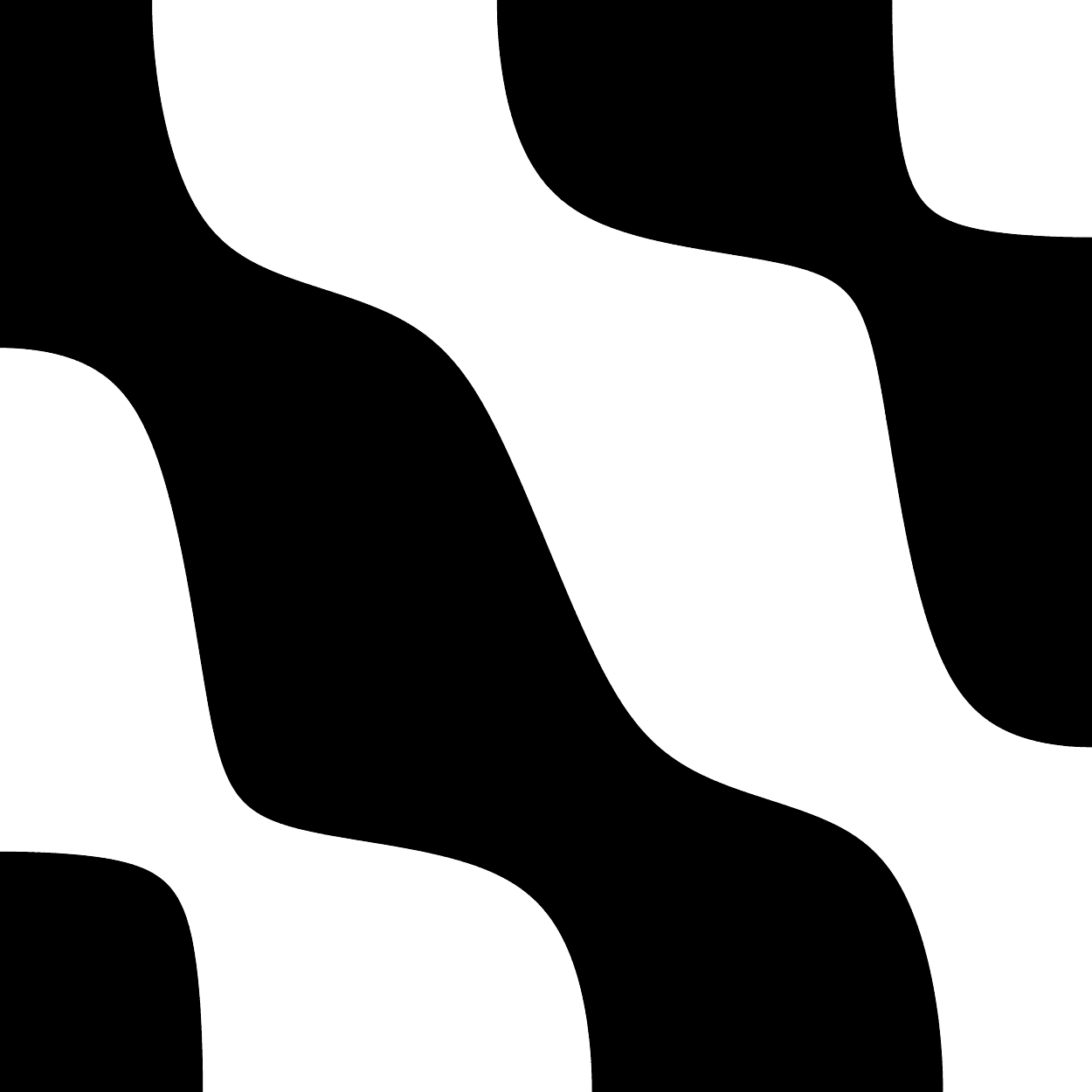}
\hskip 0.5cm
\includegraphics[width=3cm]{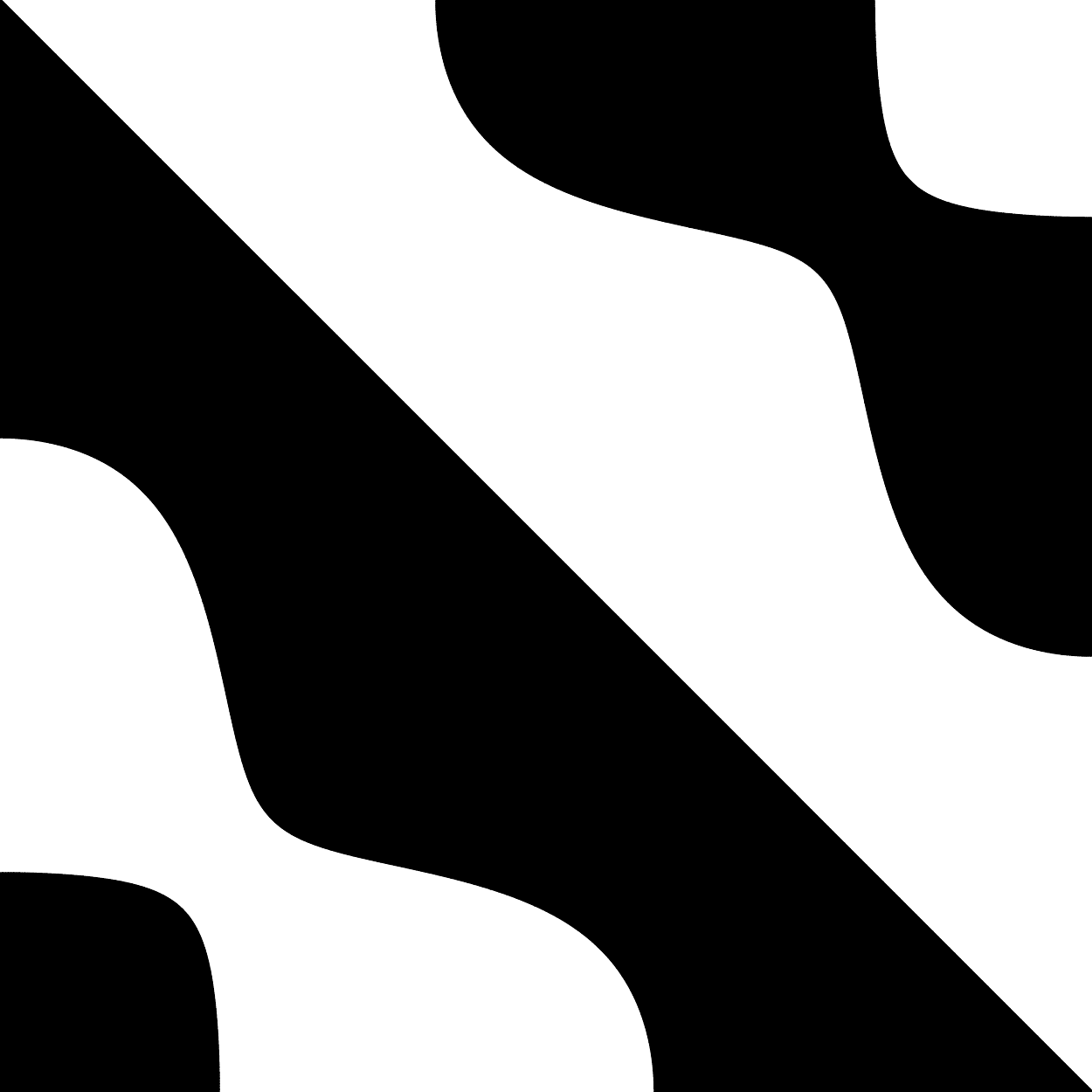}
\caption{Nodal domains for $\Phi_{3,2}^\theta $ when $\theta=0$, $\theta=\pi/8$, and $\theta=\pi/4\,$.}
\label{fig:3-2}
\end{figure}

\subsection{The case $\lambda_{49}=\lambda_{50}=52$ ($(p,q)=(6,4)$)}

\begin{lemma}
\label{lem:6-4}
If $\theta\in\{0,\pi/2\}$ then $\mu(\Phi_{6,4}^\theta)=35\,$. For all other values 
of $\theta\,$, $\mu(\Phi_{6,4}^\theta)\leq 21\,$. In particular, if 
$(\lambda_{49},\Psi_{49})$ is an eigenpair of $L\,$, then it is not Courant sharp.
\end{lemma}

\begin{proof}
If $\theta\in\{0,\pi/2\}$ then we are in the product case, and have 
$(6+1)\times(4+1)=35$ nodal domains.

If $\theta\not\in\{0,\pi/2\}$ then, combining Lemmas~\ref{lem:pandqeven} 
and~\ref{lem:3-2}, we find that the number of nodal domains are at most 
$4\cdot 6-3=21$.
\end{proof}

\section{Four remaining cases}
\label{Section7}

It remains to analyze the four cases $(4,1)$, $(8,3)$, $(9,4)$, and $(10,4)$. 
We will start by a detailed analysis of the case $(4,1)$. For the last three
cases we will use a chessboard localization to improve the estimate
$|\Omega^{\text{inn}}|\leq|\Omega|$ used in Section~\ref{Section2} 
(see Lemma \ref{lemma2.2}).

\subsection{The case $\lambda_{18}=\lambda_{19}=17$ ($(p,q)=(4,1)$)}

\begin{lemma}
\label{lem:4-1}
For any $\theta$, $\mu(\Phi_{4,1}^\theta)\leq 10$. The possible values 
of $\mu(\Phi_{4,1}^\theta)$ are $6$, $8$ and $10$. In particular, if 
$(\lambda_{18},\Psi_{18})$ is an eigenpair of $L$, then it is not Courant sharp.
\end{lemma}

\begin{proof}
The case $\theta=0$ being very simple to analyze (product situation, $10$ 
nodal domains, hence not Courant sharp). By Lemma \ref{Lemma3.8},  it 
is sufficient to do the analysis for $\theta \in (0,\frac \pi 4]$. 

\paragraph{\bfseries Step 1: analysis of the graph of $f_{4,1}$}
As explained in Section~\ref{Section5}, everything can be read on the graph 
of $f_{4,1}$.
\begin{figure}[ht]
\centering
\includegraphics[width=8cm]{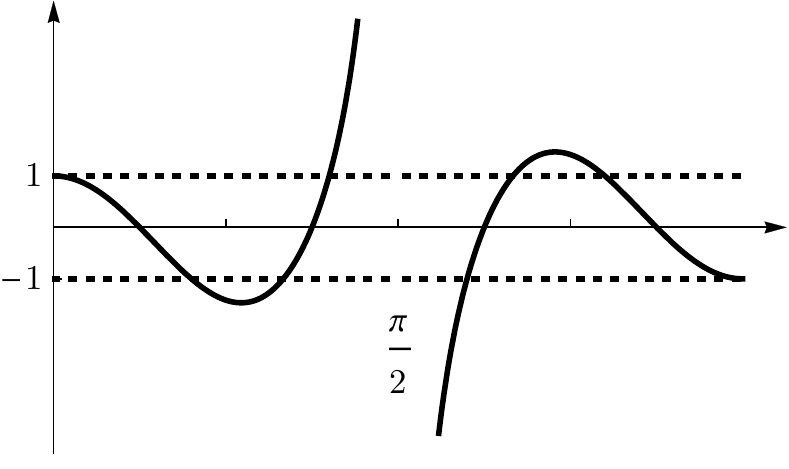}
\caption{The graph of $f_{4,1}(x)=\frac{\cos(4x)}{\cos(x)}$ in the interval 
$0<x<\pi$.}
\label{fig:4-1-cos}
\end{figure}
From the graph, we find that $f_{4,1}$ attains all values in $]-1,1[$ 
four times,  and  all values in $\mathbb{R}\setminus\{[-1,1]\}$ three, two 
or one times. This transition has to be analyzed further. Looking at the 
local extrema of $f_{4,1}$, there are two points $x_1$ and $x_2$ which 
are  solutions for \eqref{pt=qt}.
\[
x_1=\arctan\sqrt{2\sqrt{10}-5}\approx 0.86,\quad \text{and}\quad x_2 
= \pi-\arctan\sqrt{2\sqrt{10}-5}\approx 2.29.
\]
We can now follow the scheme of analysis presented in Section~\ref{Section5}.

\paragraph{\bfseries Step 2 : Interior critical points.}
Using the symmetry of $x_1$ and $x_2$ with respect to $\frac \pi 2$, we 
get that, for $\theta \in (0,\frac \pi 4]$, the only  stationary points 
of $\Phi_{4,1}^\theta$ are $(x_1,x_2)$ and $(x_2,x_1)$  and appear 
when $\theta=\pi/4$.

So we need a special analysis for $\theta =\frac \pi 4$. From 
Figure~\ref{fig:4-1} we see that the number of nodal domains is $10$.
We immediately see that the anti-diagonal belongs to the zero set.

\paragraph{\bfseries Step 3: Analysis of the boundary points}
Here we have to use~\eqref{bdp} for $(p,q)=(4,1)$ and our preliminary 
analysis of the graph of $f_{4,1}$.
We conclude that if $0<\theta<\pi/4$ then we have exactly
four nodal lines touching each part of the boundary of $\Omega$ where $y=0$ 
and $y=\pi$,  BUT  the number of  nodal lines on each of the parts of 
the boundary of $\Omega$ where $x=0$ and $x=\pi$ can be $1$, $2$ or $3$. 
There is another critical value $\theta_1 \in (0,\frac \pi 4)$ corresponding 
to a local maximum of $f_{4,1}$:
\[
\cot \theta_1 = f_{4,1} (x_2)\,.
\]
One touching occurs at $(0,x_1)$ and simultaneously at $(\pi, x_2)$.
It remains to count the number of nodal domains in this situation and to 
count the number of nodal domains for one value of 
$\theta \in (\theta_1,\frac \pi 4)$.

In each of these intervals  two strategies are possible:
\begin{itemize}
\item analyze perturbatively the situation for $\theta$ close to one of the 
ends of the interval;
\item choose one specific value of $\theta$ in the interval.
\end{itemize}
In our case, we have finally three critical values of $\theta$ 
($0$, $\theta_1$ and $\frac \pi 4$) and two values to choose in $(0,\theta_1)$ 
and $(\theta_1,\frac \pi 4)$. Because $\theta_1 \in (\frac \pi 8,\frac \pi 4)$, 
the picture for $\theta=\pi/8$ in Figure~\ref{fig:4-1} is the answer in the 
interval $(0,\theta_1)$.
Hence, we see six nodal domains  when $ \theta \in (0,\theta_1)$.

It remains to analyze the situation for $\theta =\theta_1$  and to analyze 
another case in $(\theta_1,\frac \pi 4)$. We  observe $8$ nodal domains for 
$\theta \in [\theta_1,\frac \pi 4)$ and come back to $10$ for 
$\theta=\frac \pi 4$. But this is the $18$-th eigenvalue.

\begin{figure}[htp]
\centering
\makebox[\textwidth]{%
\includegraphics[width=2.25cm]{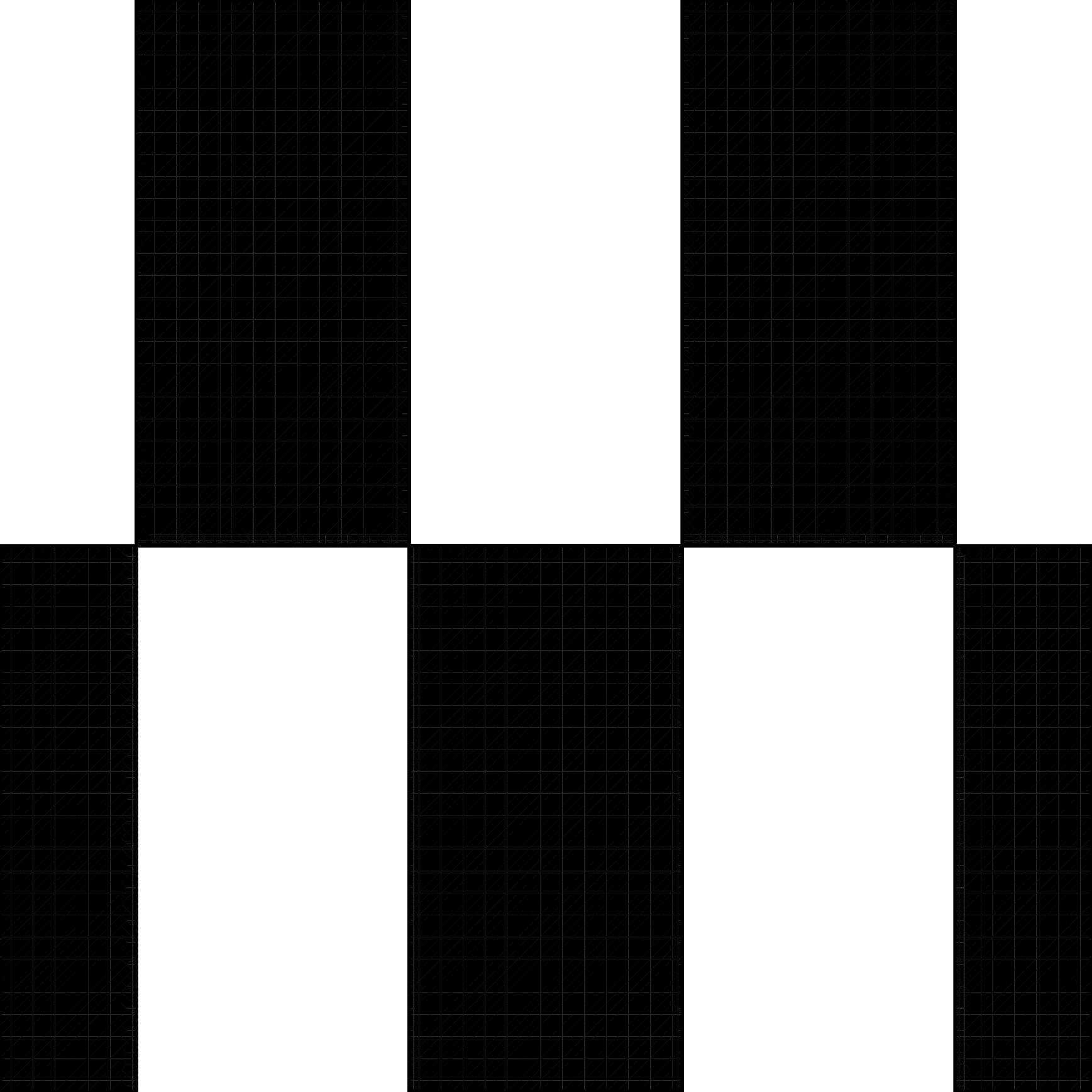}
\hskip .5cm
\includegraphics[width=2.25cm]{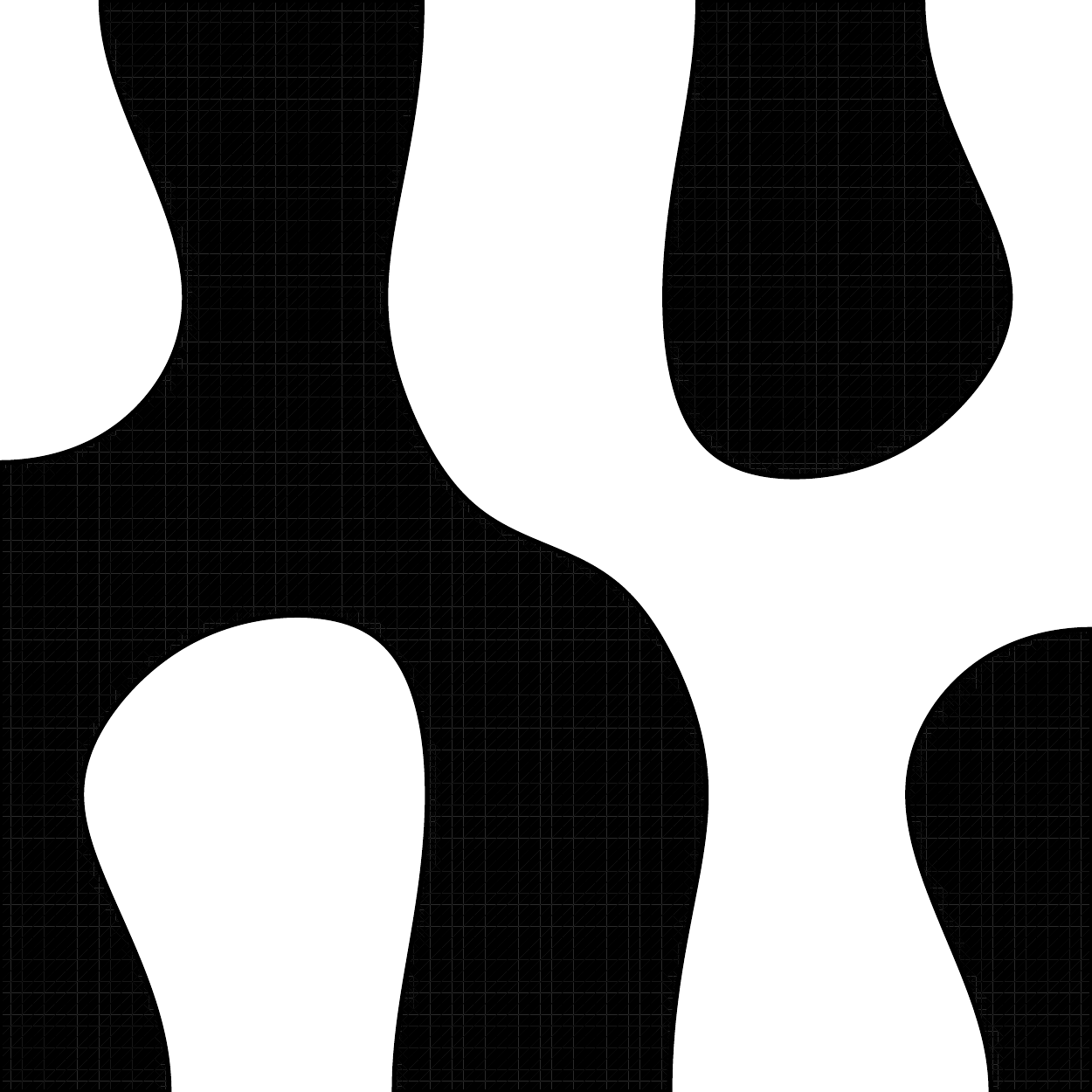}
\hskip .5cm
\includegraphics[width=2.25cm]{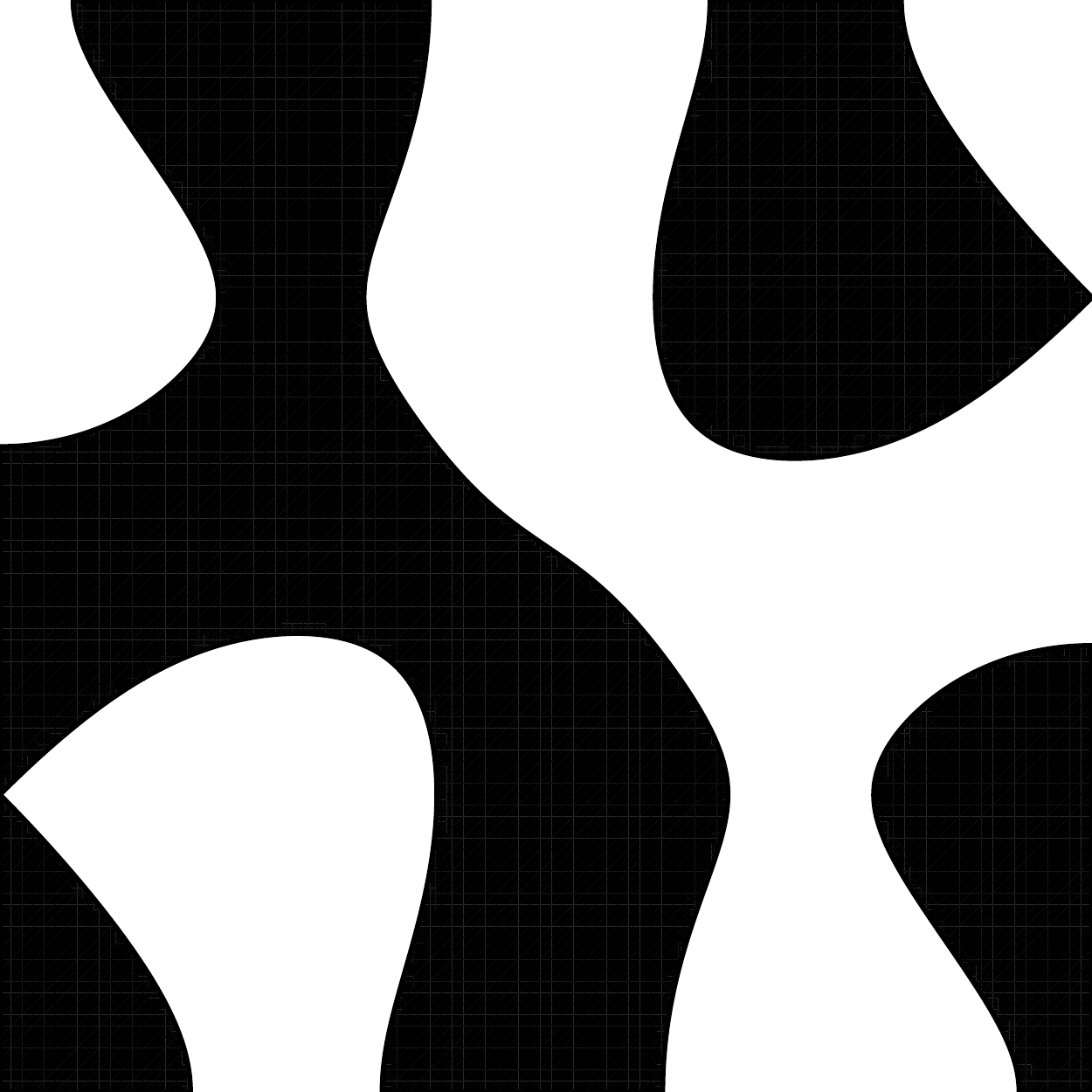}
\hskip .5cm
\includegraphics[width=2.25cm]{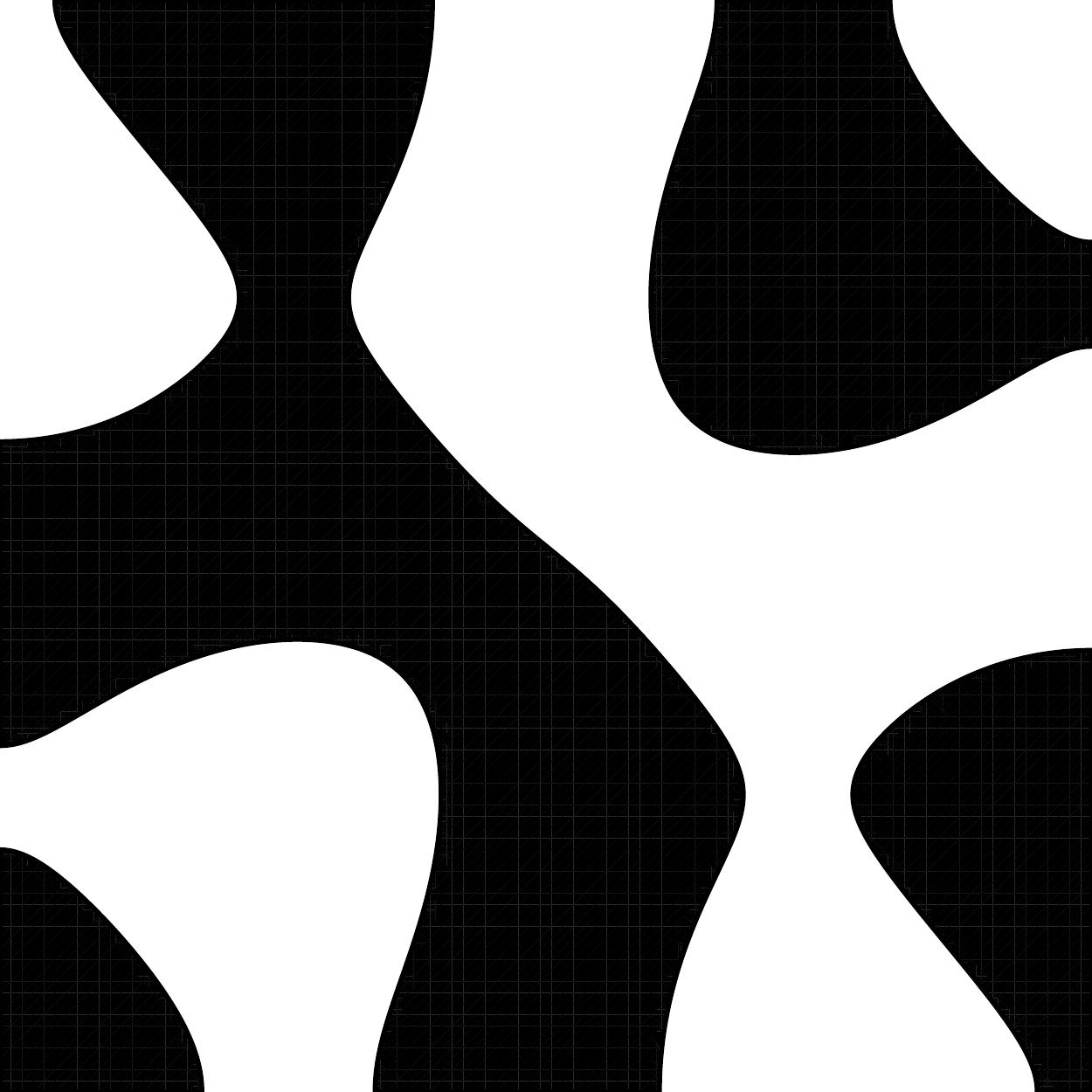}
\hskip .5cm
\includegraphics[width=2.25cm]{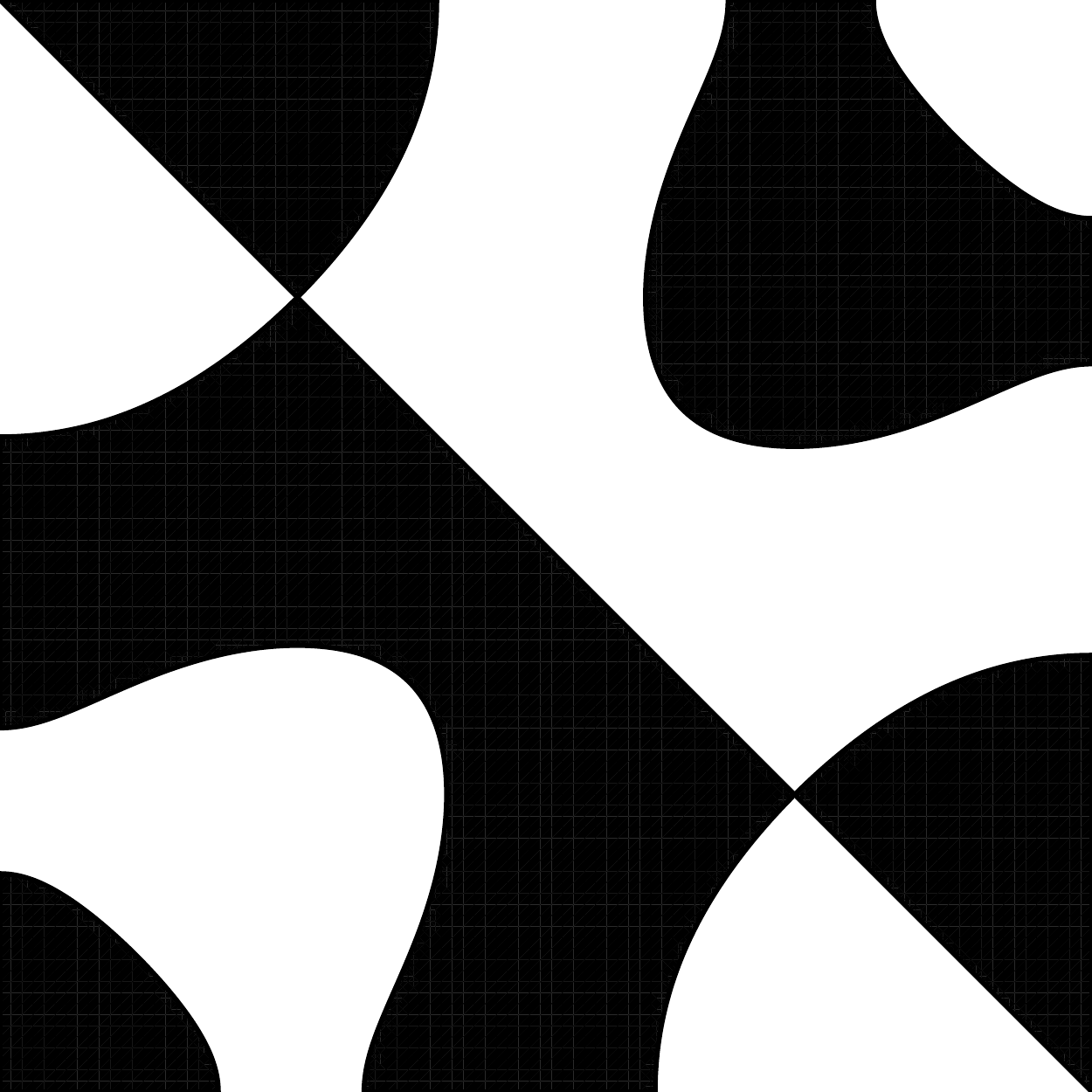}
}
\caption{The graphs show nodal domains in the case $(p,q)=(4,1)$. From left
to right, $\theta=\pi/8$, $\theta=\theta_1\approx 0.60$, $\theta=\pi/4-0.1$ and
$\theta=\pi/4$.}
\label{fig:4-1}
\end{figure}

\paragraph{\bfseries Step 4: chessboard localization}
We refer to Section~\ref{Section5} and particularly to the Lemmas~\ref{lemma5.6} 
and~\ref{lemma5.7}.
We now consider $\theta$ in the interval $(0,\theta_1)$. We know that there are 
no critical points inside. Hence one line entering in a rectangle by one of the 
admissible corner belonging to the nodal set should exit the black rectangle by 
another corner or by the boundary. Conversely, a line starting from the 
boundary should leave the black rectangle through a corner in the zero set. 
Note that contrary to the case considered in~\cite{BH} it is not true that 
all the corners belong to the zero set.

We now look at the points on the boundary. For $x=0$, we have shown that there 
is only one point $(0,y_1)$. Moreover $y\in ( \frac \pi 2, \frac{5\pi}{8})$. 
Similar considerations can be done to localize the four points on $y=0$ and on 
$y=\pi$. These localization are independent of $\theta \in (0,\theta_1)$. 
Finally, we notice that $x=\pi/ 2$ meets the nodal set at a unique point 
$(\pi / 2,\pi / 2)$  and the same for $y=\frac \pi 2$.
It is then easy to verify that one can reconstruct uniquely  the nodal picture 
using these rules.

For $\theta =\theta_1$ and  $\theta \in (\theta_1,\frac \pi 4)$, similar 
arguments lead to a unique  topological  type.  For $\theta=\frac \pi 4$, one 
should first analyze the neighborhood of the anti-diagonal where critical 
points are present.
\begin{figure}[htbp]
\centering
\includegraphics[width=6cm]{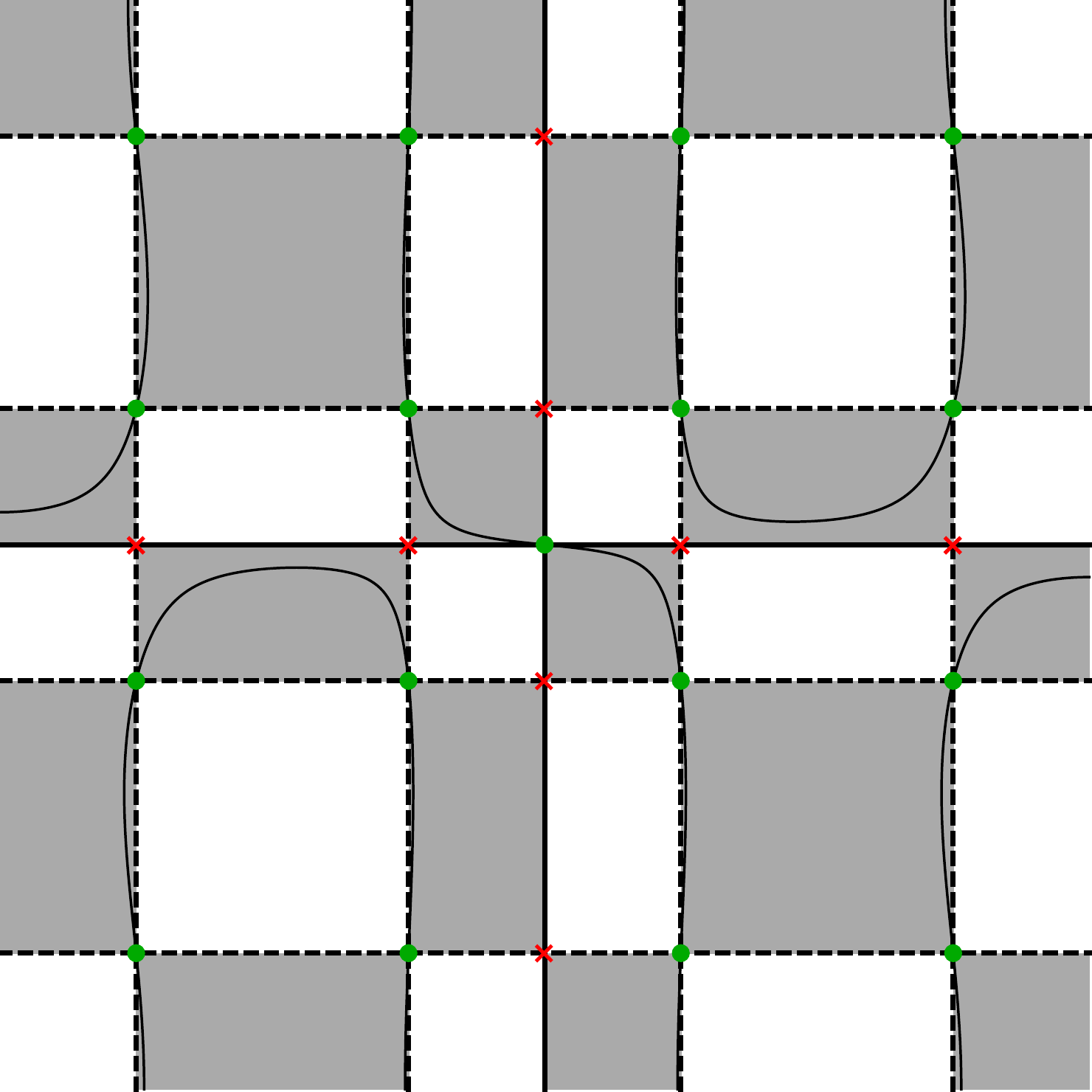}
\caption{$(p,q)= (4,1)$, $\theta=0.1\,$.}
\label{fig:4-1-chess-01}
\end{figure}
\begin{figure}[htbp]
\centering
\includegraphics[width=6cm]{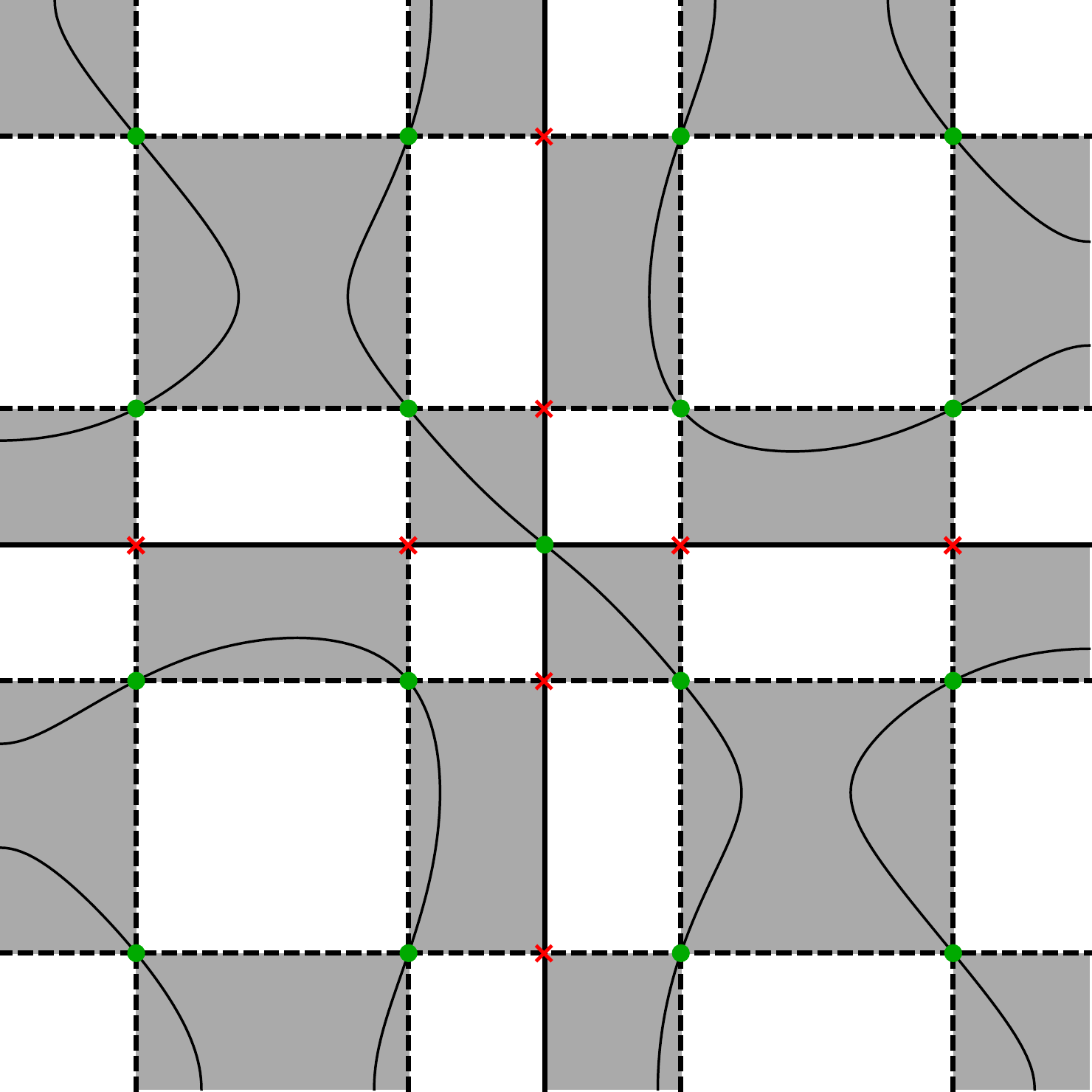}
\caption{$(p,q)= (4,1)$, $\theta=\pi/4-0.1$.}
\label{fig:4-1-chess-pi4-01}
\end{figure}
\end{proof}
\newpage

\subsection{The case $\lambda_{66}=\lambda_{67}= 73$ ($(p,q)= (8,3)$)}

\begin{lemma}
\label{lem:8-3}
Let $(\lambda_{66},\Psi_{66})$ be an eigenpair of $L$. Then 
$\mu(\Psi_{66})\leq 56$. In particular, $(\lambda_{66},\Psi_{66})$ is not
Courant sharp.
\end{lemma}

\begin{proof}
By Lemma~\ref{Lemma3.8} it is sufficient to estimate the nodal domains
of $\Phi_{8,3}^\theta$ for $0\leq\theta\leq \pi/4$. First, we note that for
$\theta=0$ we are in the product situation and have $36$ nodal domains.

Next, we use the chessboard argument. For all $0<\theta\leq \pi/4$ it holds
that the function $\cos 8x\cos 3y\cos 3x\cos8y>0$ in the white rectangles. Thus no nodal lines can cross white 
rectangles. Moreover, since both $\cos\theta\neq 0$ and $\sin\theta\neq 0$, we
find that the nodal lines must pass corners where both $\cos8x\cos3y$ and
$\cos3x\cos8y$ are zero, and 
that they cannot pass corners where only one of them are zero (marked with a red
cross).  In Figure~\ref{fig:8-3-chess-pi4-marked}, we paint white 
rectangles blue in the following way: First we let each white rectangle touching 
the boundary become blue. Then we paint each white rectangle having
a forbidden corner (marked with a red cross) in common with a blue rectangle. 
This latter procedure is repeated until it does not apply anymore. The so 
recolored blue rectangles are then necessarily all subsets of 
$\Omega^{\text{out}}$ of nodal domains touching the boundary. Note that this construction is independent of $\theta$. Thus,
\[
|\Omega^{\text{out}}|\geq \pi^2\Bigl[\tfrac{1}{2}-8\times \bigl(\tfrac{1}{8}\bigr)^2\Bigr]
=\tfrac{3}{8}\pi^2,
\]
and hence
\[
|\Omega^{\text{inn}}|\leq \tfrac{5}{8}\pi^2.
\]
Then Lemma \ref{lemma2.2} gives,
\[
\mu(\Psi_{66})
\leq 
\frac{|\Omega^{\text{inn}}|}{\pi j_{0,1}^2}\lambda_{66}+4\lfloor\sqrt{\lambda_{66}}\rfloor
\leq \frac{365\,\pi}{8\, j_{0,1}^2}+32\approx 56.8.\qedhere
\]

\begin{figure}[htbp]
\centering
\includegraphics[width=8cm]{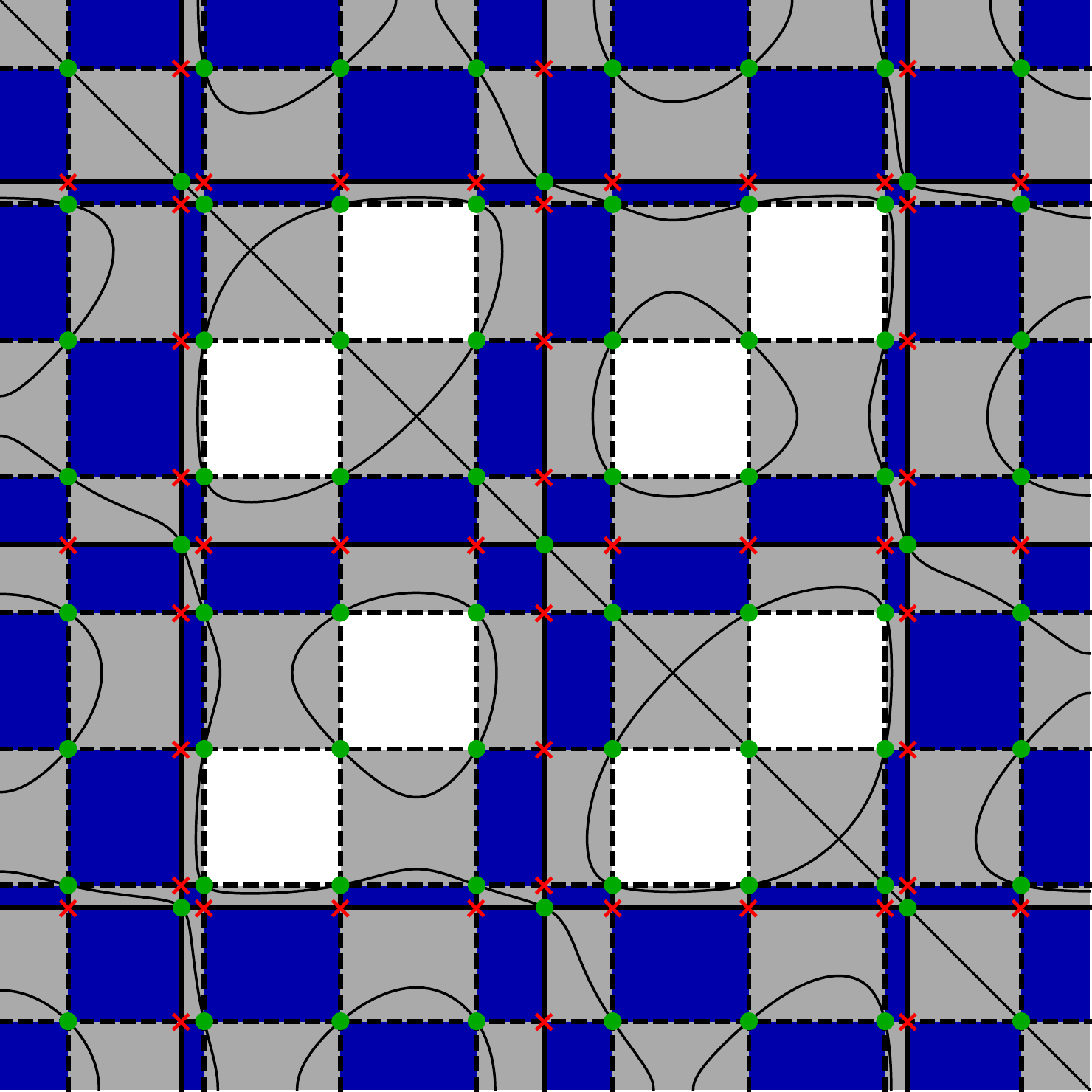}
\caption{The case $(p,q)= (8,3)$, $\theta = \frac \pi 4 $.}
\label{fig:8-3-chess-pi4-marked}
\end{figure}
\end{proof}

\newpage

\subsection{The case $\lambda_{84}=\lambda_{85}= 97$   ($(p,q)= (9,4)$)}
\begin{lemma}
\label{lem:9-4}
Let $(\lambda_{84},\Psi_{84})$ be an eigenpair of $L$. Then 
$\mu(\Psi_{84})\leq 71$. In particular, $(\lambda_{84},\Psi_{84})$ is not
Courant sharp.
\end{lemma}

\begin{proof}
This proof  goes line by line as the proof of Lemma~\ref{lem:8-3}, with a 
change of numbers only. For $\theta=0$ we are in the product situation, with 
$50$ nodal domains. Here, the blue area of Figure~\ref{fig:9-4-chess-pi4-marked}
equals $\tfrac{211}{648}\, \pi^2$, and so we know that $\Omega^{\text{out}}$ has this
lower bound. This means that 
\[
|\Omega^{\text{inn}}|\leq \tfrac{437}{648}\pi^2\,.
\]

We get from Lemma~\ref{lemma2.2}
\[
\mu(\Psi_{84})\leq 
\frac{|\Omega^{\text{inn}}|}{\pi\, j_{0,1}^2}\lambda_{84}+4\lfloor\sqrt{\lambda_{84}}\rfloor
\leq \frac{42389\, \pi}{648\, j_{0,1}^2}+36\approx 71.5.\qedhere
\]

\begin{figure}[htbp]
\centering
\includegraphics[width=8cm]{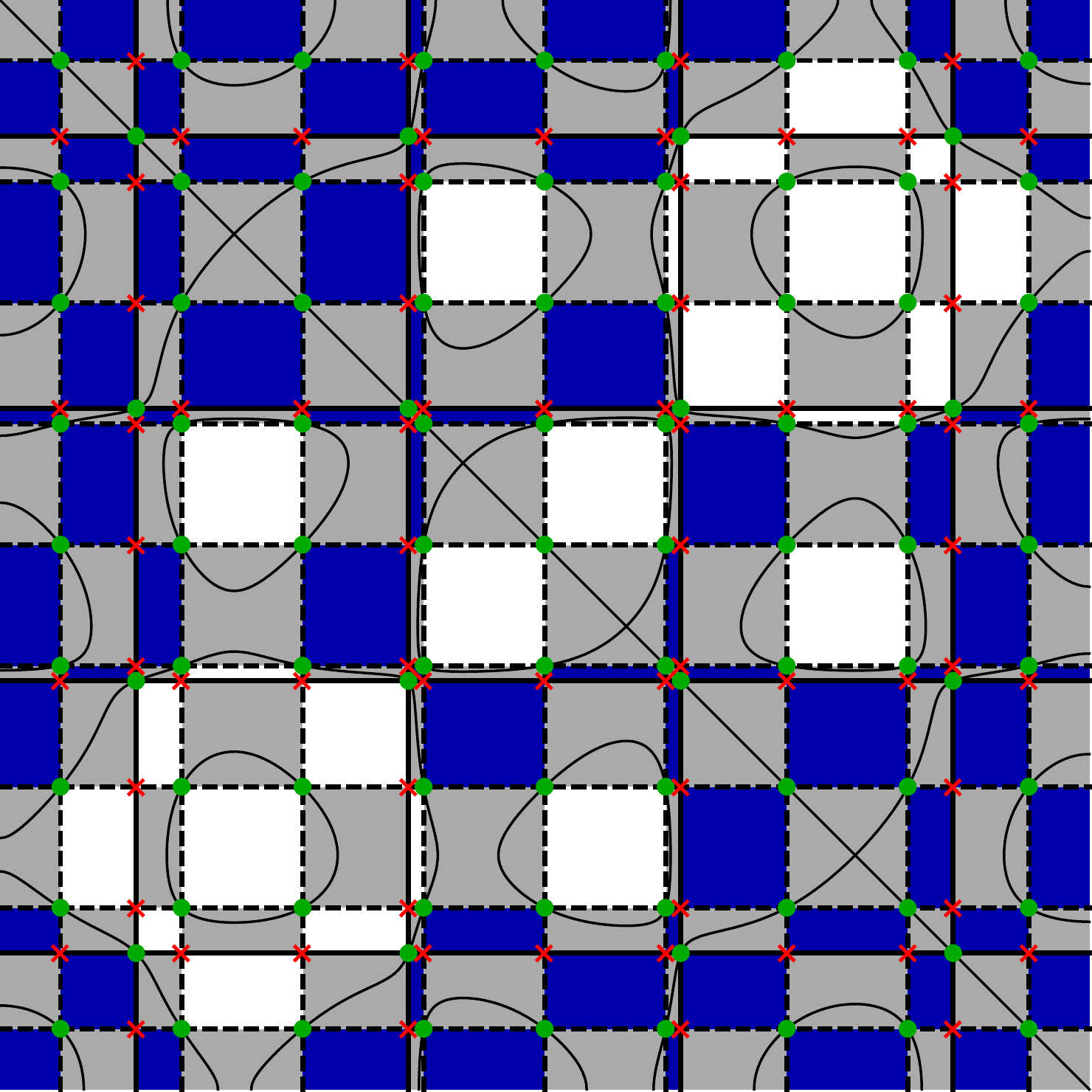}
\caption{The case $(p,q)= (9,4)$, $\theta = \frac \pi 4$.}
\label{fig:9-4-chess-pi4-marked}
\end{figure}
\end{proof}

\subsection{The case $\lambda_{101}=\lambda_{102}=116$ ($(p,q)=(10,4)$)}

\begin{lemma}
\label{lem:10-4}
Let $(\lambda_{101},\Psi_{101})$ be an eigenpair of $L$. Then 
$\mu(\Psi_{101})\leq 89$. In particular, $(\lambda_{101},\Psi_{101})$ is not
Courant sharp.
\end{lemma}

\begin{proof}
We first note that we are in the product situation if $\theta=0$ or $\theta=\pi/2$,
with $55$ nodal domains.
Since $\cos(10x)$ and $\cos(4x)$ have no common zeros, we can again apply the 
chessboard argument for $0<\theta<\pi/2$.  For $\pi/2<\theta<\pi$
the argument is exactly the same, but with the roles of white and black 
rectangles interchanged.

Thus, as in the previous proofs, we count the area of the blue region, see
Figure~\ref{fig:10-4-chess-pi4-marked}. We find that it equals 
$\tfrac{21}{100}\pi^2$. 

Hence $|\Omega^{\text{inn}}|\leq \tfrac{79}{100}\pi^2$, and by 
Lemma~\ref{lemma2.2}
\[
\mu(\Psi_{101})\leq 
\frac{|\Omega^{\text{inn}}|}{\pi j_{0,1}^2}\lambda_{101}
+4\lfloor\sqrt{\lambda_{101}}\rfloor
\leq \frac{2291\pi}{25j_{0,1}^2}+40\approx 89.8.\qedhere
\]
\begin{figure}[htbp]
\centering
\includegraphics[width=8cm]{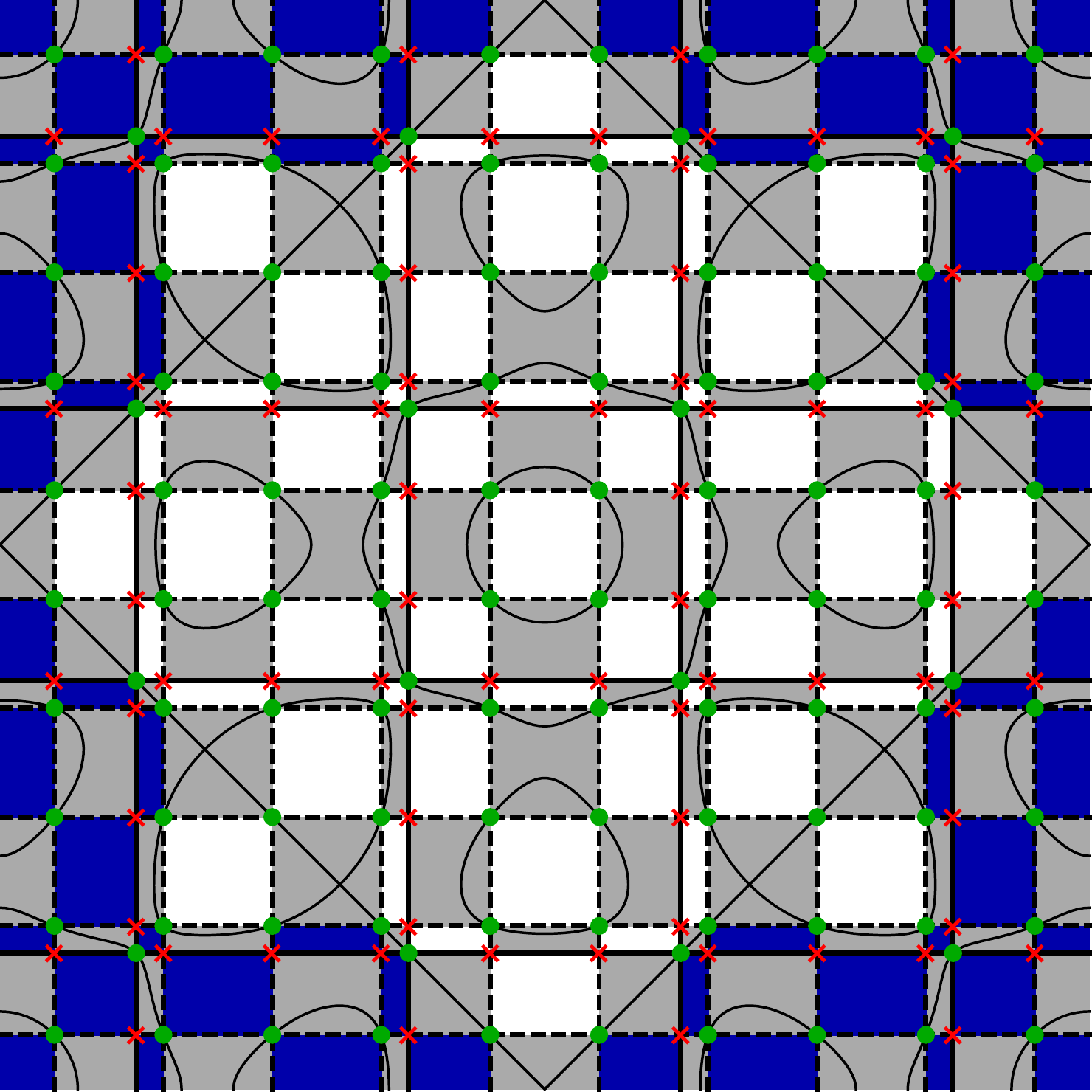}
\caption{The case $(p,q)= (10,4)$, $\theta = \frac \pi 4$.}
\label{fig:10-4-chess-pi4-marked}
\end{figure}
\end{proof}

With the proof of this last statement and the analysis presented in the table we have achieved the proof of Theorem \ref{thm:main}.  In the next section we will by curiosity analyze the spectral pattern of some of the families.

\section{Optimal calculations}\label{Section8}

Although not used in the proof of the main results in this paper, we present 
in the spirit of the analysis of the case $(4,1)$ a complete analysis of 
the nodal pattern in the cases $(5,2)$ and $(8,3)$.

\subsection{The case $\lambda_{29}=\lambda_{30}= 29$ ($(p,q)= (5,2)$)}
We already know that we are 
not Courant sharp by Lemma \ref{lem:antisymmetric}.  It does not cost too much (an 
application could be for the analysis of the case $(10,4)$ but we only need a 
weaker upper bound) to establish the
\begin{prop}
\label{lem:5-2}
For any $\theta$, we have the optimal bound  $\mu(\Phi_{5,2}^\theta) \leq 18\,$. 
\end{prop}

The analysis of the equation $5 \tan 5x = 2 \tan 2x$ leads to the existence of 
two positive solutions $(x_1,x_2)$ of this equation with 
\[
0 <\frac \pi 4 <  x_1 <\frac \pi 2 < x_2 = \pi -x_1\,.
\]
These values appear also as the critical values of $f_{5,2}(x)$ as can be seen 
in Figure~\ref{fig:5-2-cos}.
\begin{figure}[htbp]
\centering
\includegraphics[width=8cm]{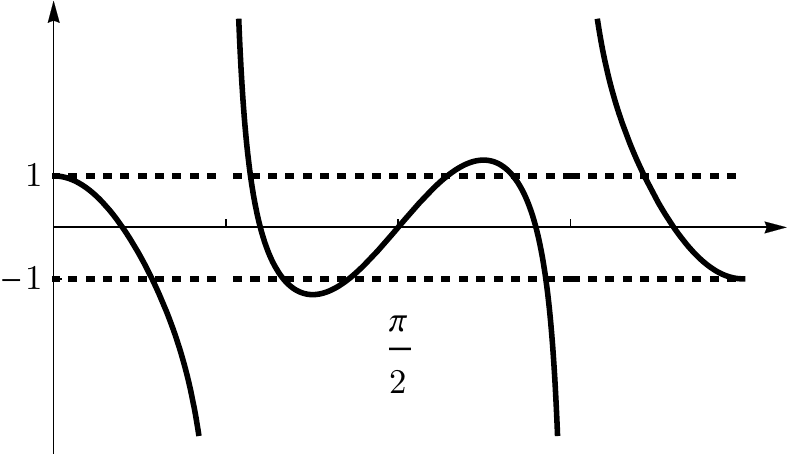}
\caption{The graph of $f_{5,2}(x)=\frac{\cos(5x)}{\cos(2x)}$ in the interval 
$0<x<\pi$.}
\label{fig:5-2-cos}
\end{figure}

It is sufficient to analyze the situation for $\theta \in (0,\frac \pi 4]$.  
The discussion is rather close to the case $(4,1)$.

For $\theta =0$, we start from $3\times 6$ nodal domains. We have indeed $10$ 
critical points and $14$ boundary points (avoiding the corners).
In the interval $(0,\frac \pi 4)$, there are no critical point inside the 
square.  But there are transition at the boundary for 
$\theta_1\in (0,\frac \pi 4)$ such that $\tan \theta_1= 1/ f_{5,2}(x_2)$. Hence 
the number of nodal domains is fixed  in $(0,\theta_1)$
and because when starting from $0$, we have only opening at the former crosses, 
the number of nodal domains can only decrease. More precisely, the number of 
nodal domain is eight.

\begin{figure}[htp]
\centering
\makebox[\textwidth]{
\includegraphics[width=2.25cm]{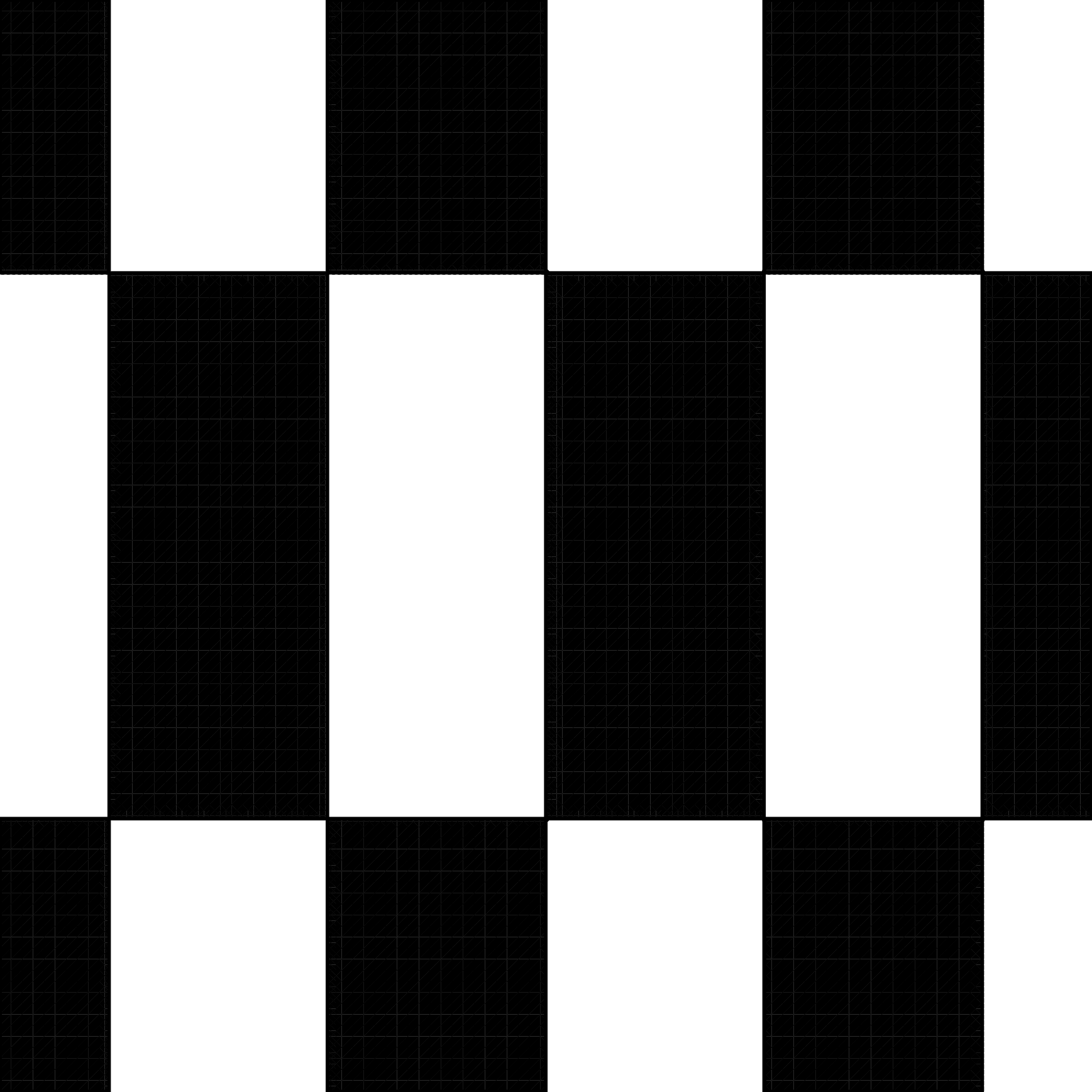}
\hskip .5cm
\includegraphics[width=2.25cm]{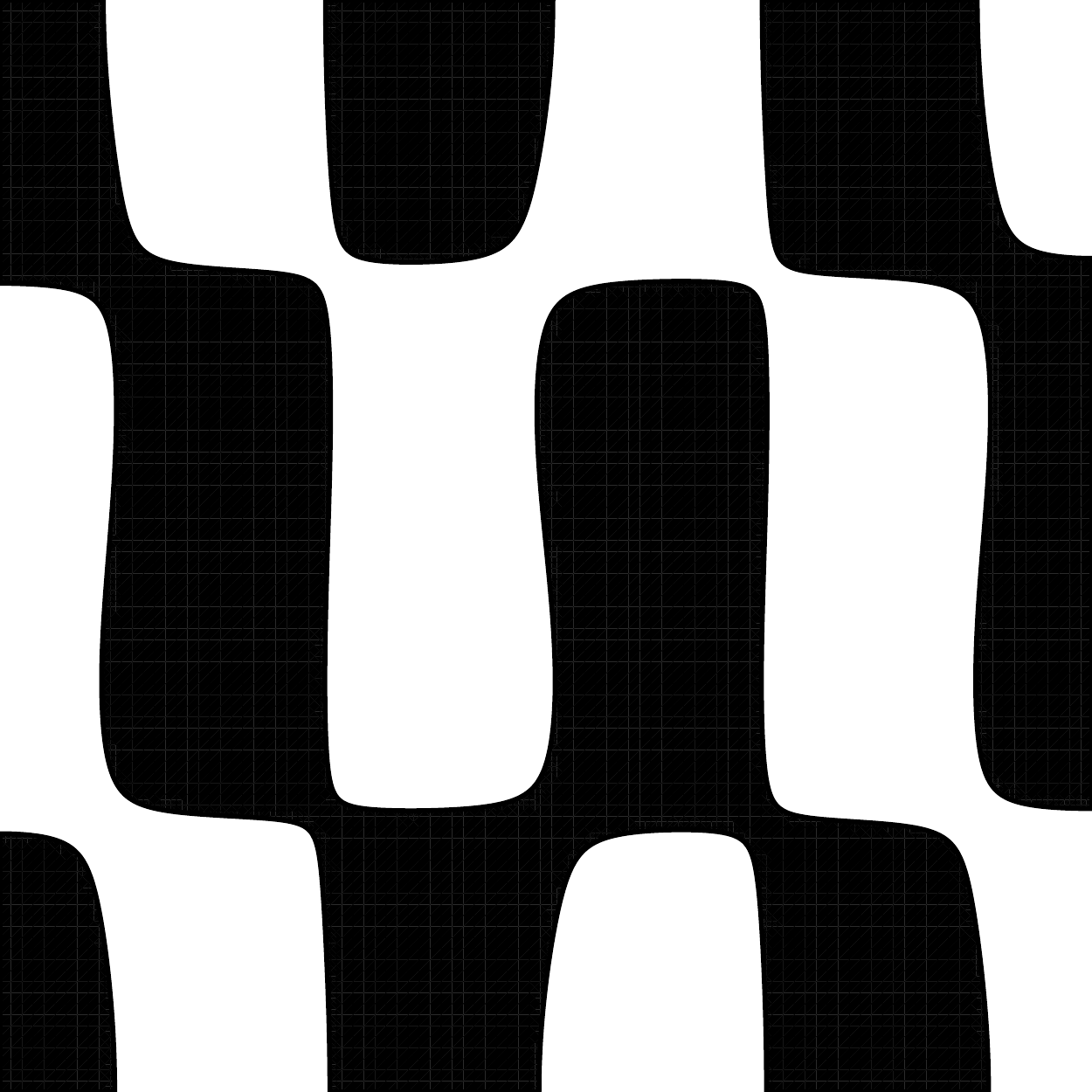}
\hskip .5cm
\includegraphics[width=2.25cm]{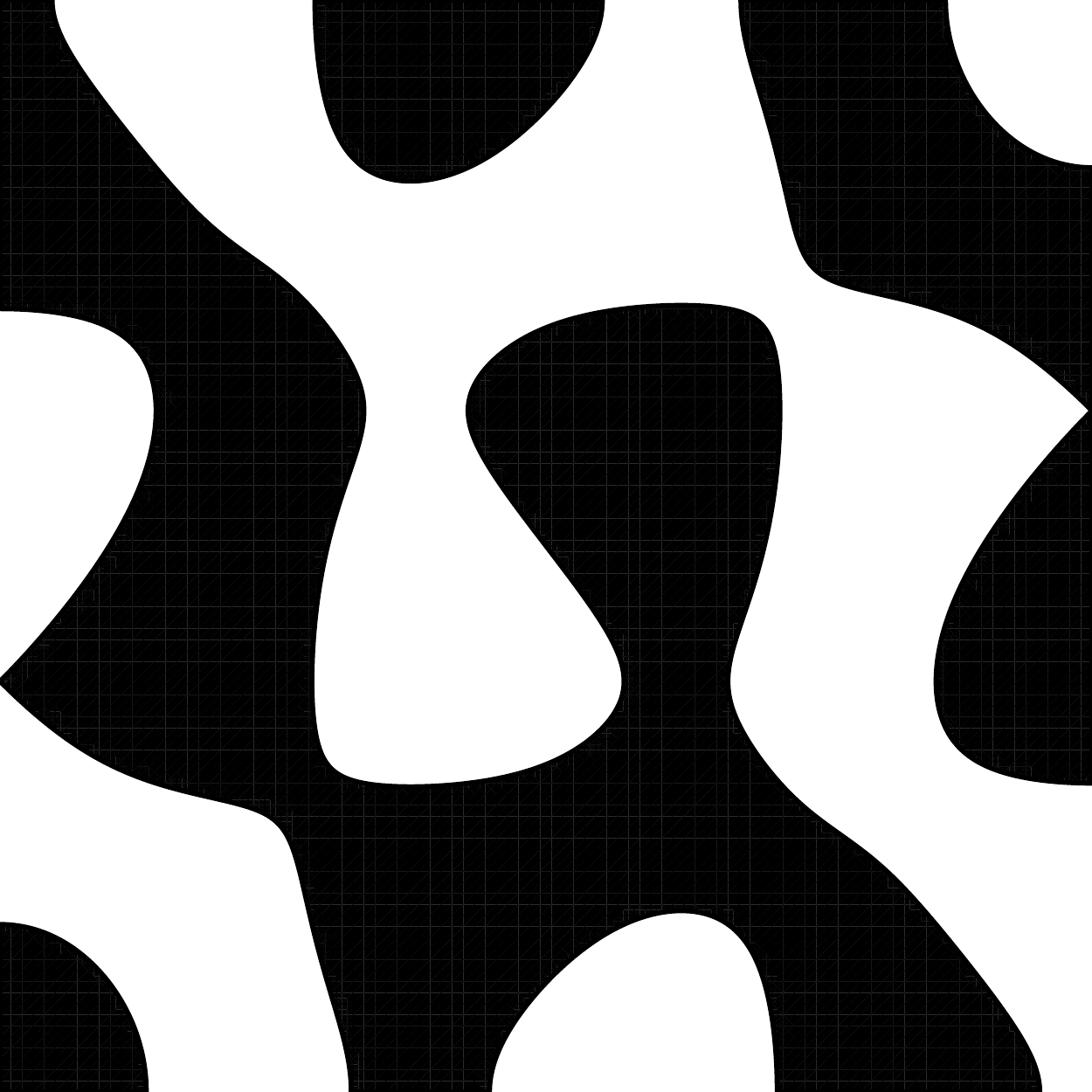}
\hskip .5cm
\includegraphics[width=2.25cm]{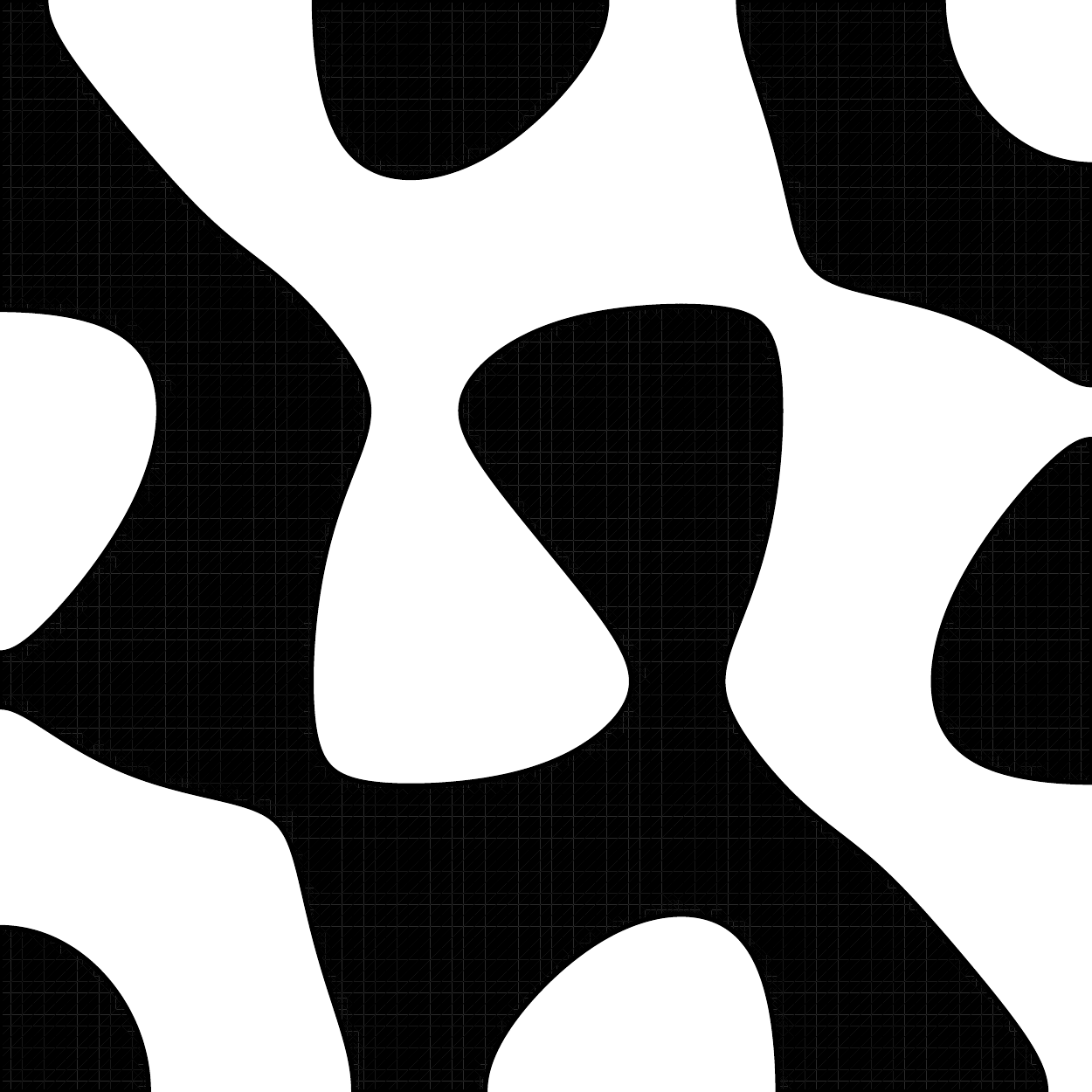}
\hskip .5cm
\includegraphics[width=2.25cm]{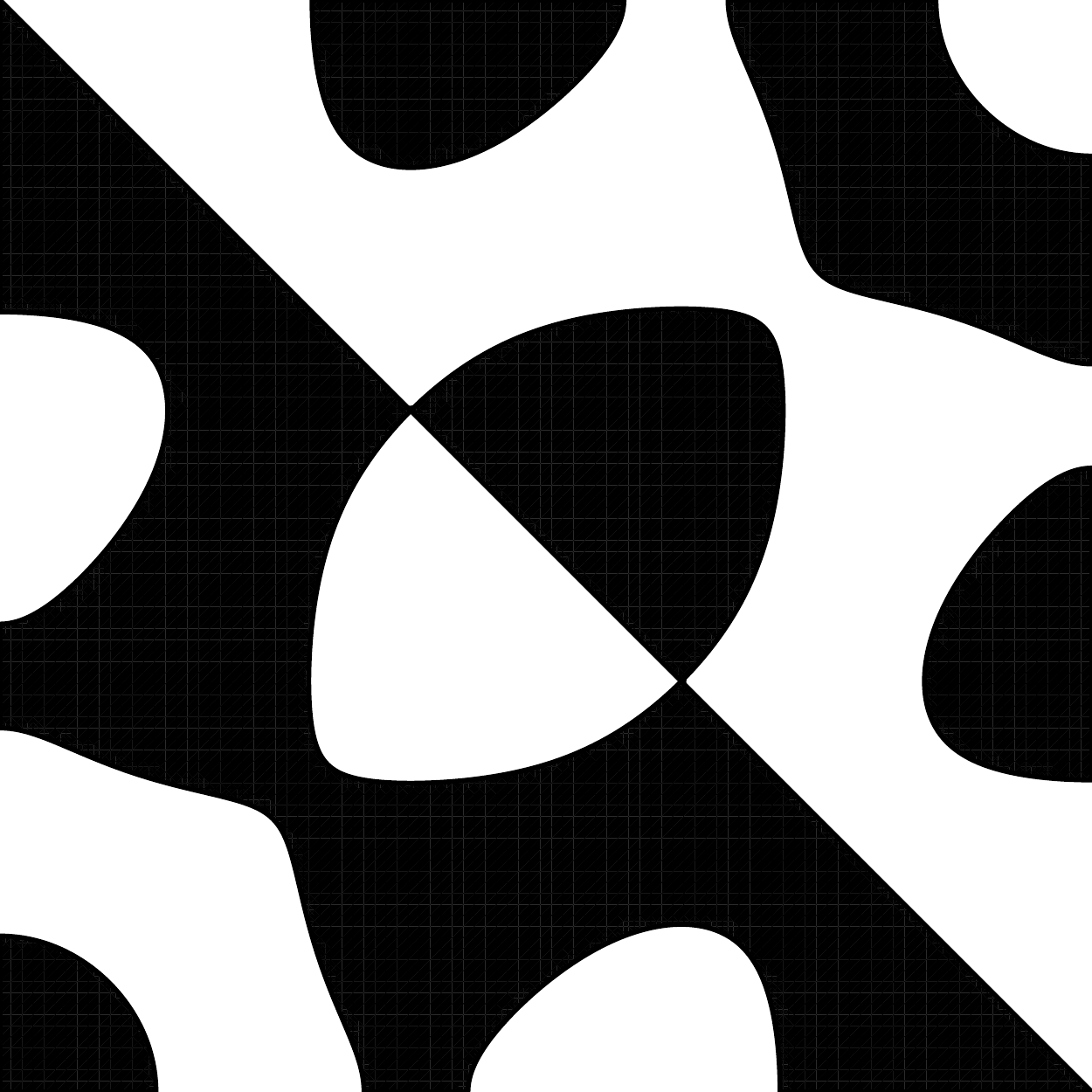}
}
\caption{The graphs show nodal domains in the case $(p,q)=(5,2)$. From left
to right, $\theta=0$, $\theta=0.1$, $\theta=\theta_1\approx 0.65$, $\theta=\pi/4-0.1$ and
$\theta=\pi/4$.}
\label{fig:5-2}
\end{figure}

This results  from the numerics or the perturbative analysis starting 
from $\theta=0$.  An analysis for $\theta=\theta_1$ should be done. Then again 
the number of nodal domains is fixed in $(\theta_1,\frac \pi 4)$ and equal to $10$.

At $\theta =\frac \pi 4$, the nodal set contains the anti-diagonal 
$y = \pi -x$. 
The number of boundary points  outside the corners  become $4$ on each side.
We have two critical points $(x_1,x_2)$ and $(x_2,x_1)$ on the anti-diagonal.
Numerics  permits  us  to determine the number of nodal domains, which 
increases by $2$ and is equal to $10$ for $\theta=\frac \pi 4$.

We  keep in mind the results established in Section \ref{Section5} on the 
chessboard localization.  We now consider $\theta$ in the interval 
$(0,\theta_1)$. We know that there are 
no critical points inside. Hence one line entering in a rectangle by one of the 
corners belonging to the nodal set should exit the black rectangle by another 
corner or by the boundary. Conversely, a line starting from the boundary should 
leave the black rectangle through a corner in the zero set. Note that contrary 
to the case considered in~\cite{BH} it is not true that all the corners belong 
to the zero set.

We now look at the points on the boundary. For $x=0$, we have shown that there 
are exactly two points $(0,\eta_1)$ and $(0,\eta_2)$. Moreover 
$\eta_1\in ( \frac{ \pi}{10}, \frac{\pi}{4})$ and 
$\eta_2 \in (\frac{7\pi}{10}, \frac{3\pi}{4})$. Similar considerations can be 
done to localize the five points on $y=0$, $\xi_1,\xi_2,\xi_3,\xi_4,\xi_5$,  and on 
$y=\pi$, $\xi'_1,\dots,\xi'_5$, and two points on $x=\pi$, $ \eta'_1, \eta'_2$. These 
localizations are independent of $\theta \in (0,\theta_1)$. 

Let us see if one can reconstruct uniquely  the  (topology of the) nodal picture using these 
rules.\\
The nodal line starting from the boundary at  $(0,\eta_1)$ has no choice  
(that is the ordered sequence of admissible corners which are visited 
 is uniquely determined) and should arrive to $(\xi_1,0)$. 
The line starting of $(0,\eta_2)$ has no choice and should arrive to
 $(\xi_2,0)$. The line starting from $(\xi'_2,\pi)$ should come back 
to $(x'_3,\pi)$. The line starting from $(x'_4,\pi)$ is obliged to go to
$(\pi,\eta'_1)$ and the line starting from $(\xi'_5,\pi)$ has to go to $(\pi,\eta'_2)$. 
All these lines are unique. It remains one line which has to go from 
$(\xi'_1,\pi)$ to $(\xi_5,0)$ with the obligation to visit all the elements of the 
two lattices which have not been visited before. A small analysis shows that 
it remains two possible paths around the center (see Figure \ref{fig:5-2alt}).

\begin{figure}[htbp]
\centering
\includegraphics[width=3cm]{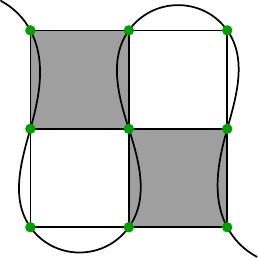}  
\hskip 2cm 
\includegraphics[width=3cm]{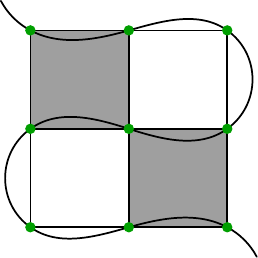}
\caption{The two alternatives in the central square (consisting of four smaller 
squares). The left alternative is the right one in  the next figure.}
\label{fig:5-2alt}
\end{figure}
Hence we need an additional argument to fix the sequence of visited 
admissible corners.
For example it is enough to show that on the line $x= \frac{2\pi}{5}$ there are 
no zero  $(\frac{2\pi}{5},y)$ with $y \in (\frac \pi 2, \frac{7 \pi}{10})$. 
This is at least clear for $\theta$ small and because no critical point can 
appear before $\theta =\frac \pi 4$. We are done with this case.

\begin{figure}[ht]
\centering
\includegraphics[width=8cm]{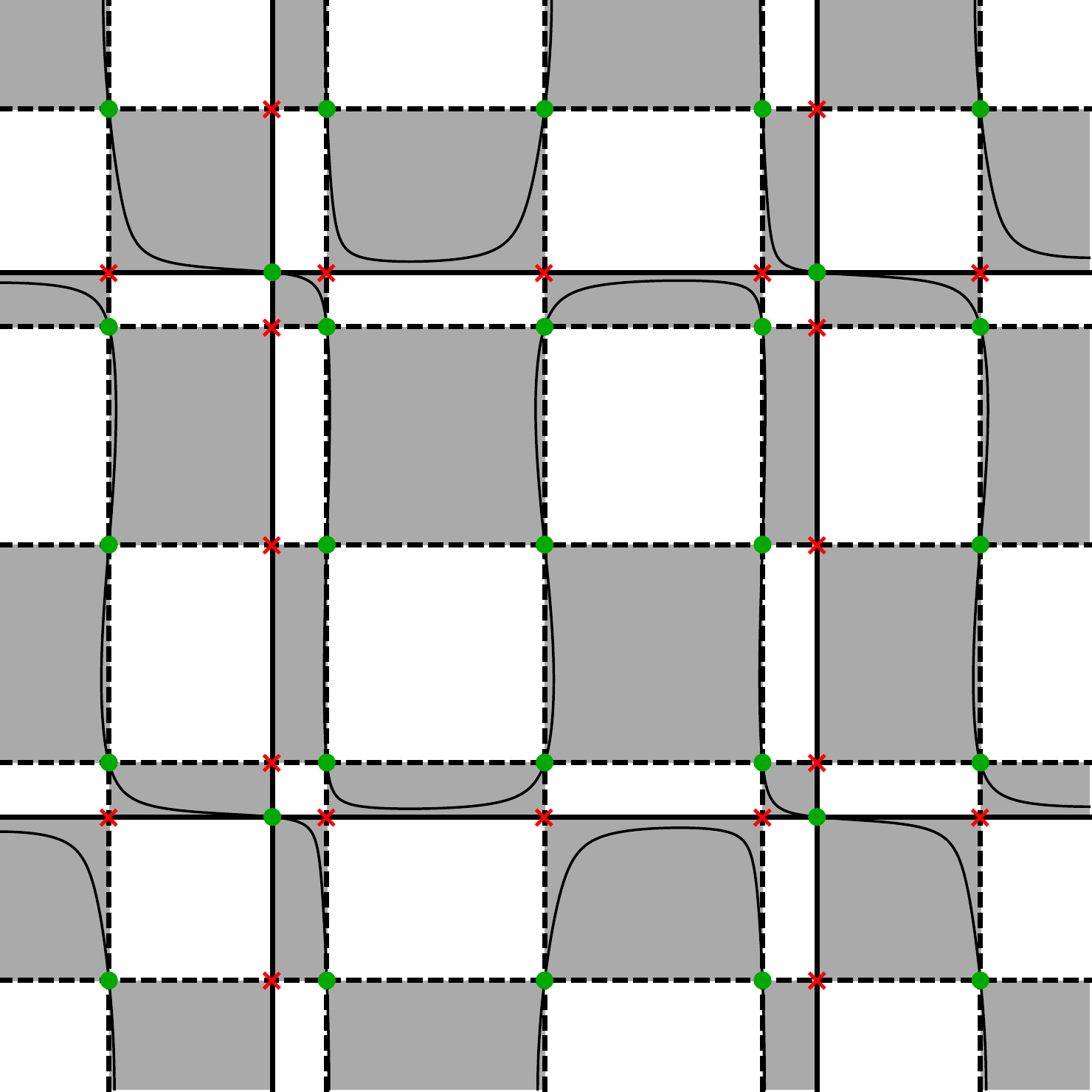}
\caption{$(p,q)= (5,2)$, $\theta=0.1$.}
\label{fig:5-2-chess-01}
\end{figure}

\begin{figure}[ht]
\centering
\includegraphics[width=8cm]{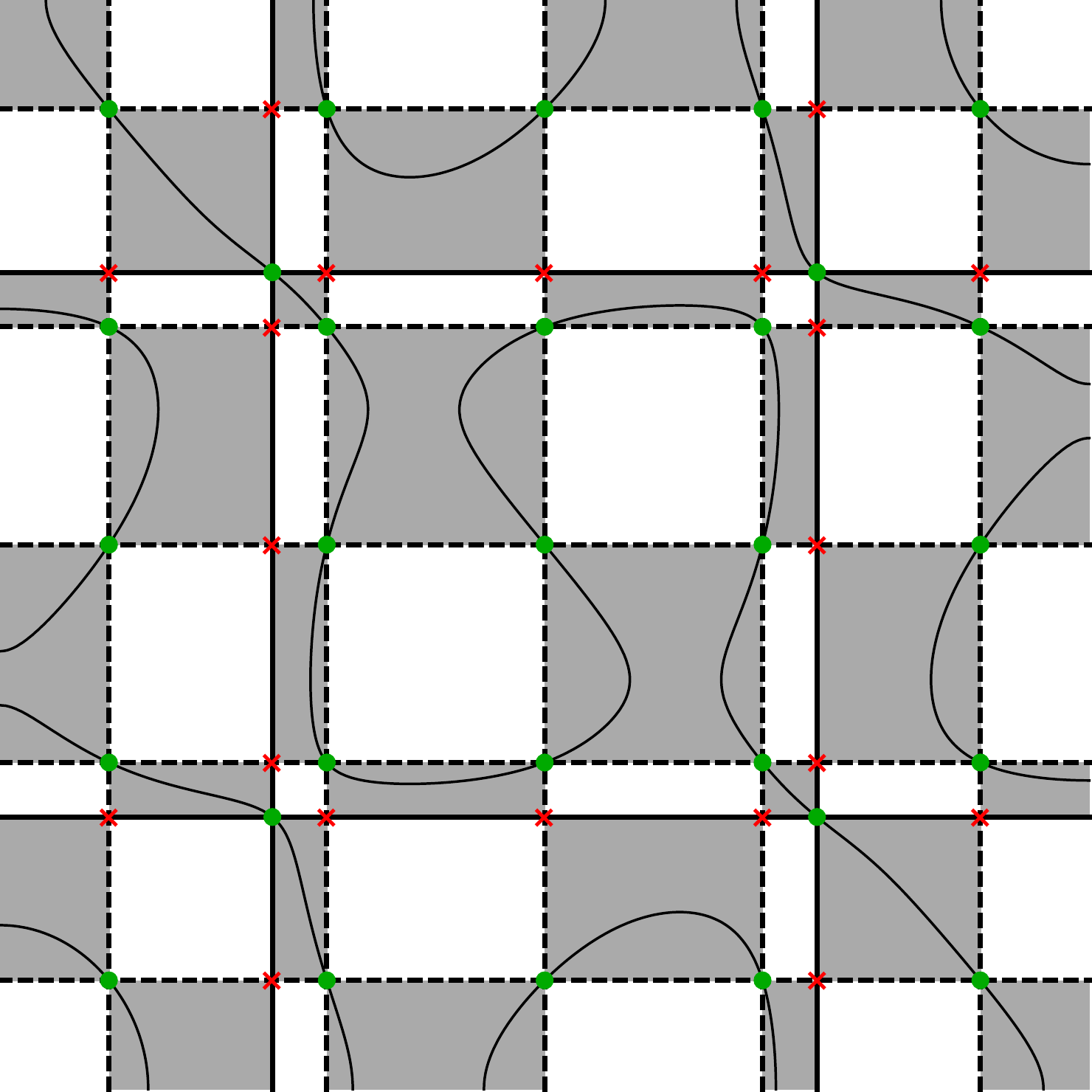}
\caption{$(p,q)= (5,2)$, $\theta=\pi/4-0.1$.}
\label{fig:5-2-chess-pi4-01}
\end{figure}

For $\theta \in (\theta_1,\frac \pi 4)$, similar arguments lead to a unique 
topological  type.  We have now four points $(0, \eta_j)$ ($j=1,\dots,4$) and four points $(\pi,\eta'_j)$ ($j=1,\dots,4$) at the vertical boundaries
 but except a change in the black rectangle containing  $(0,\eta_2)$ and $(0,\eta_3)$ and the black rectangle containing $(\pi,\eta'_2)$ and $(\pi,\eta'_3)$, nothing changes outside.  At a first sight, there are still two possibilities but the 
transition to the second possibility can only occur through a critical point. 
This is impossible before $\theta=\frac \pi 4$.

\begin{remark}
As a corollary, we recover that
the eigenvalue $\lambda_{101}$ is not Courant sharp.
This is an immediate consequence of 
Lemma~\ref{lem:pandqeven} and Lemma~\ref{lem:5-2}. This gives:
\[
\mu(\Phi_{10,4}) \leq 69\,.
\]
This can be improved by using the detailed case by case analysis of 
$\Phi^\theta_{5,2}$. We will then get the optimal upper-bound:
\[
\mu(\Phi_{10,4}) \leq 55 \,,
\]
to compare with Lemma \ref{lem:10-4}.
\end{remark}

\subsection{The case $\lambda_{66}=\lambda_{67}= 73$ ($(p,q)= (8,3)$)}

\begin{prop}
\label{lem:8-3-opt}
For any $\theta$, we have the optimal bound  $\mu(\Phi_{8,3}^\theta) \leq 36\,$. 
\end{prop}

\begin{figure}[ht]
\centering
\includegraphics[width=8cm]{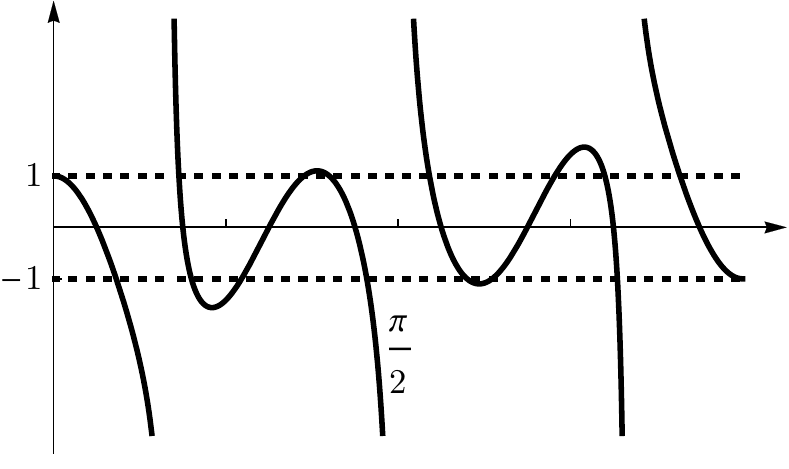}
\caption{The graph of $f_{8,3}(x)=\frac{\cos(8x)}{\cos(3x)}$ in the interval 
$0<x<\pi$.}
\label{fig:8-3-cos}
\end{figure}

The analysis of the Figure~\ref{fig:8-3-cos} shows the existence of six 
critical points:
\[
x_0=0<\frac \pi 6 < x_1 < \frac{5\pi}{16} < x_2
< \frac \pi 2 < x_3 < x_4< \frac{5\pi}{6} < \pi=x_5\,,
\]
with 
\[
x_4=\pi -x_1\,,\quad x_3=\pi -x_2\,.
\]
We associate with  these critical values the two positive numbers:
\[
z_1= f_{8,3} (x_2)\approx 1.10\,,\quad z_2 = - f_{8,3} (x_1)\approx 1.56\,,
\]
and observe that
\[
1 < z_1 < z_2\,.
\]
Associated with $( z_1,z_2)$ we introduce the two critical angles   
in $(0,\frac \pi 4)$:
\[
 \tan \theta_1= 1/z_2\,, \,\tan \theta_2 = 1/z_1\,.
\]
For these two values some transition should appear at the boundary.

We now look at the interior critical points corresponding to pairs 
$(x_i,x_j)$ ($i=1,\dots, 4)$ and $j=(1,\dots,4)$.  The corresponding critical
 $\theta_{ij}$ are determined by 
\[
 \tan \theta_{ij} = - f_{8,3} (x_j)/ f_{8,3} (x_i)
\]
 with $\theta_{ij} \in (0,\pi)$.
 
Using the symmetries, it is enough to look at the ones which belong to 
$(0,\frac \pi 4]$. We recover $\frac \pi 4$ with any pair $(x_i, \pi -x_i)$
but we have also to consider $ \theta_{13}$ which is determined by $z_1/z_2$. 
We observe (numerically) that
\[
   \frac {1}{z_2} < \frac{z_1}{z_2} < \frac{1}{z_1}\,.
\]
Hence we have at the end to look at the values 
$0$, $\theta_1$, $\theta_{13}$,  $\theta_2$ and $\frac \pi 4$ and to four 
values of $\theta$ corresponding to each of the intervals $(0,\theta_{1})$, 
$(\theta_{1}, \theta_{13})$, $(\theta_{13},\theta_2)$ and 
$(\theta_2,\frac \pi 4)$.

\begin{figure}[ht]
\centering
\makebox[13cm]{%
\includegraphics[width=6cm]{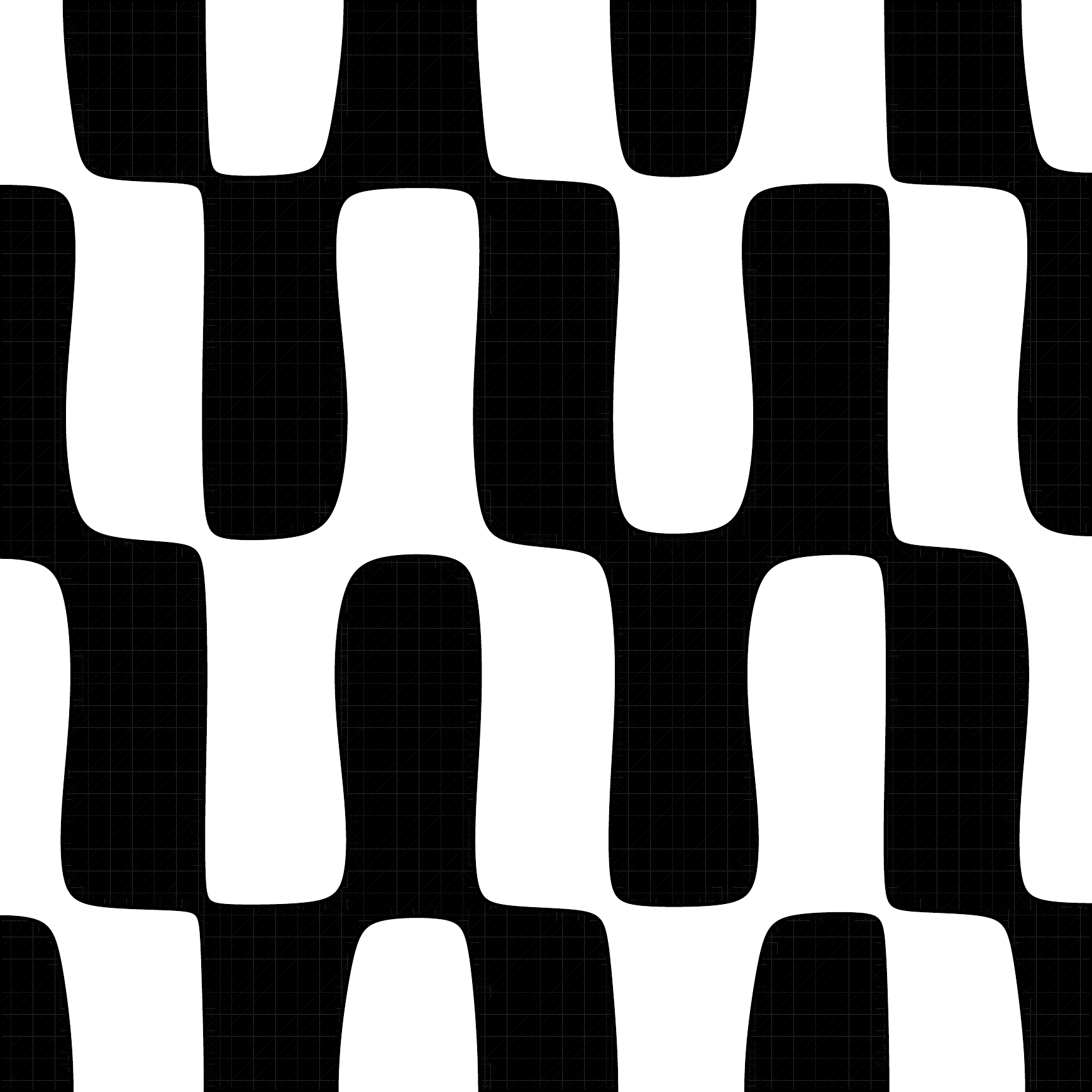}
\hskip 1cm
\includegraphics[width=6cm]{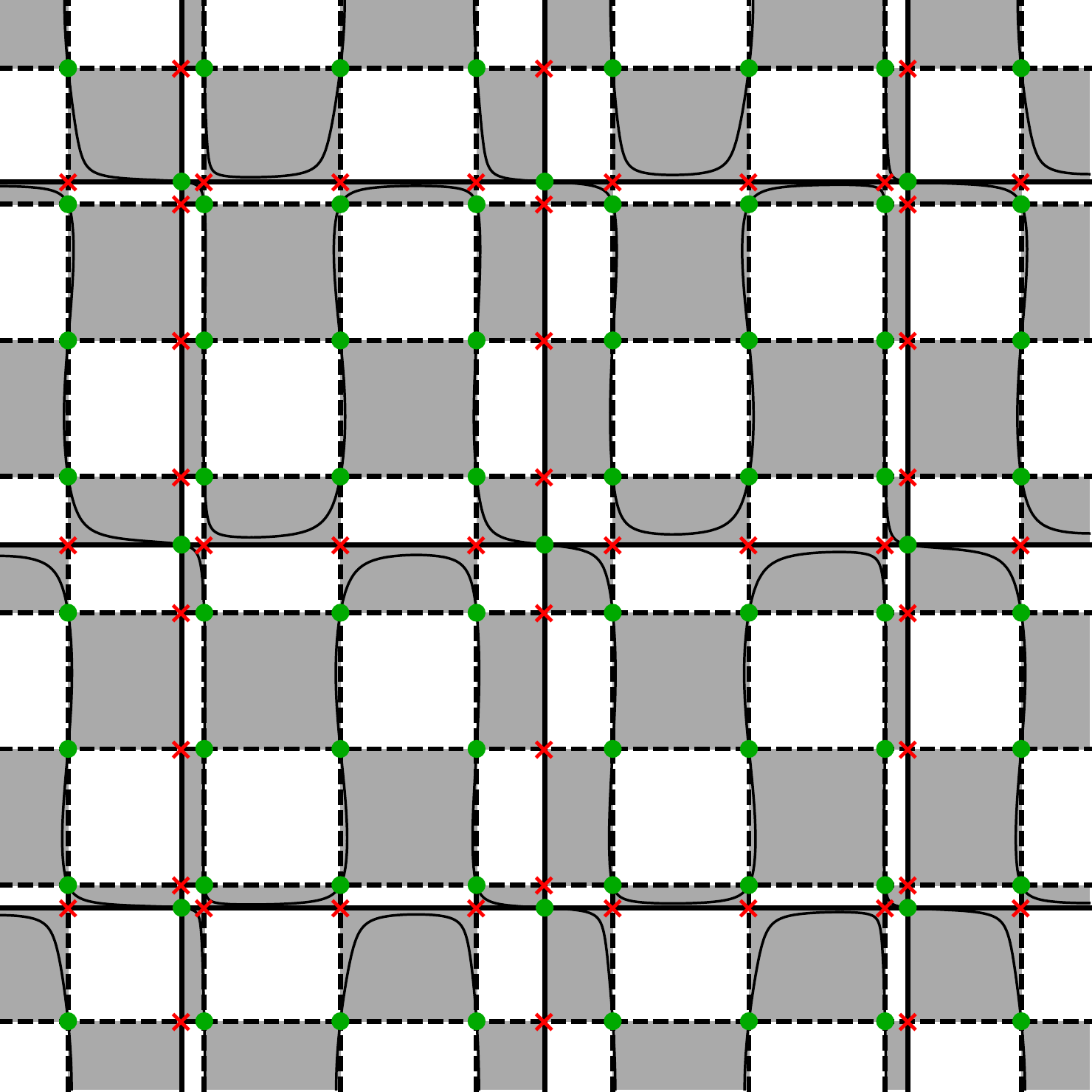}
}
\caption{Nodal domains for $(p,q)= (8,3)$ and $\theta=0.1$ (12 nodal domains).}
\label{fig:8-3-chess-01}
\end{figure}

\begin{figure}[ht]
\centering
\makebox[13cm]{%
\includegraphics[width=6cm]{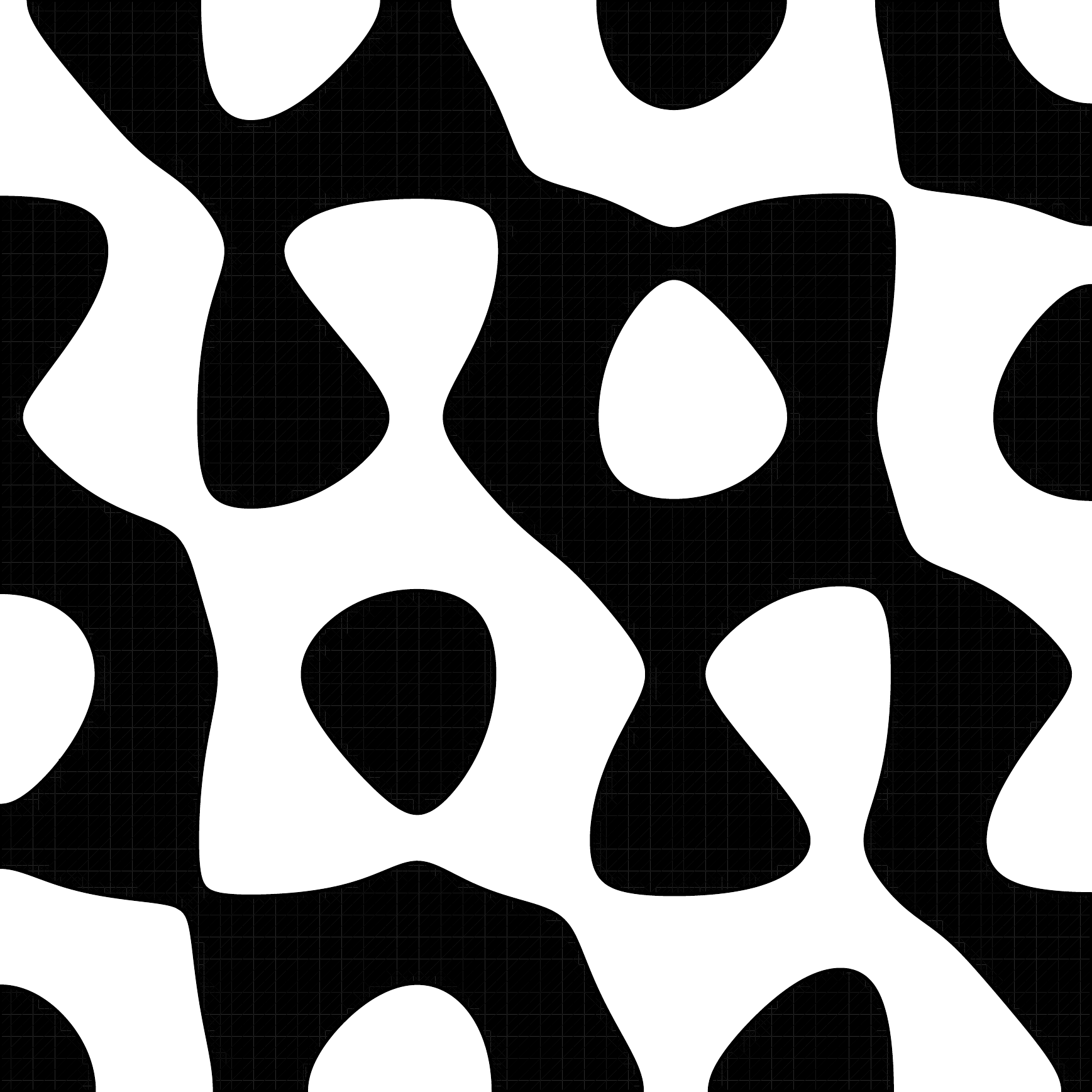}
\hskip 1cm
\includegraphics[width=6cm]{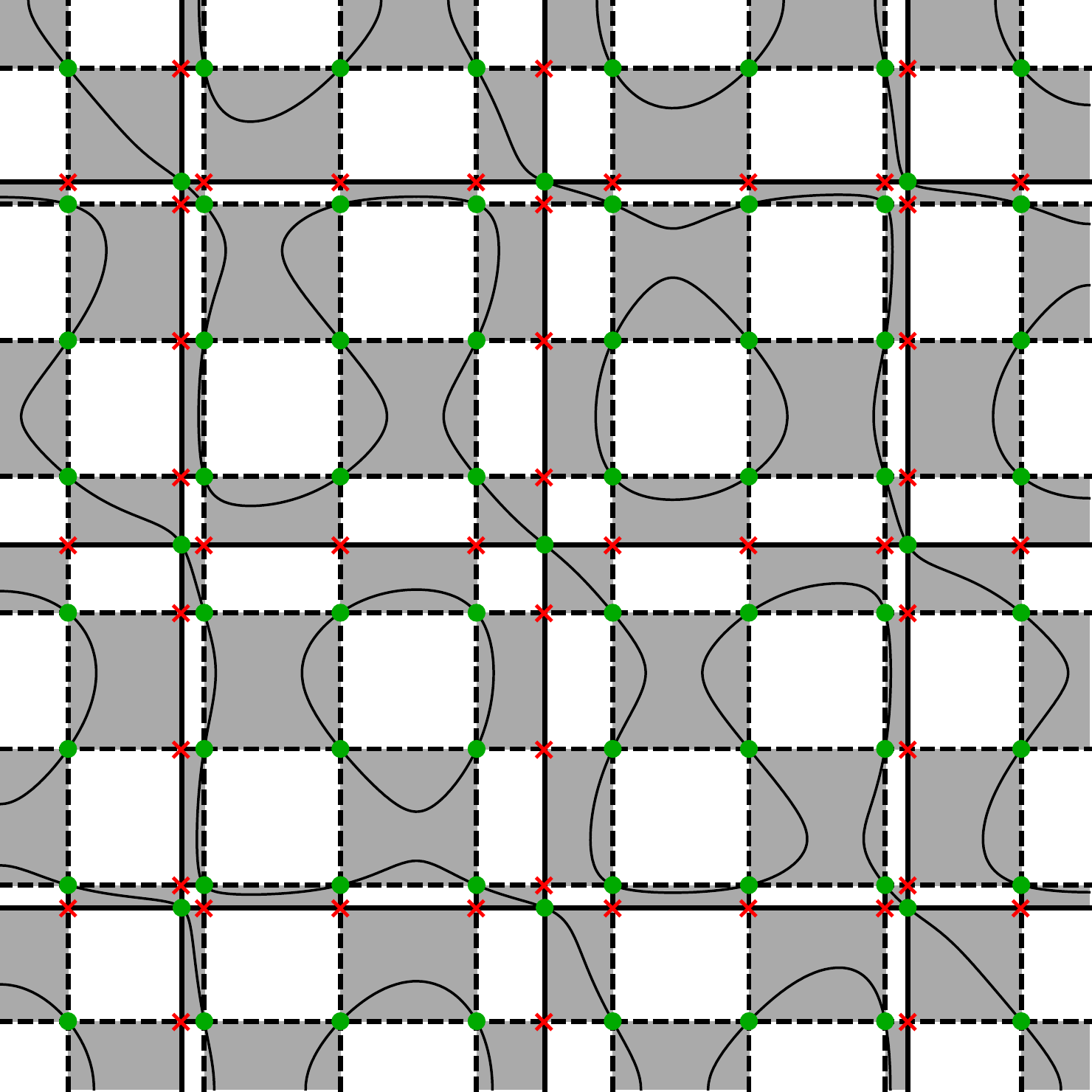}
}
\caption{Nodal domains for $(p,q)=(8,3)$ and  $\theta=\pi/4-0.1$ (16 nodal domains).}
\label{fig:8-3-thetapi4small}
\end{figure}

\begin{figure}[htbp]
\centering
\makebox[13cm]{%
\includegraphics[width=6cm]{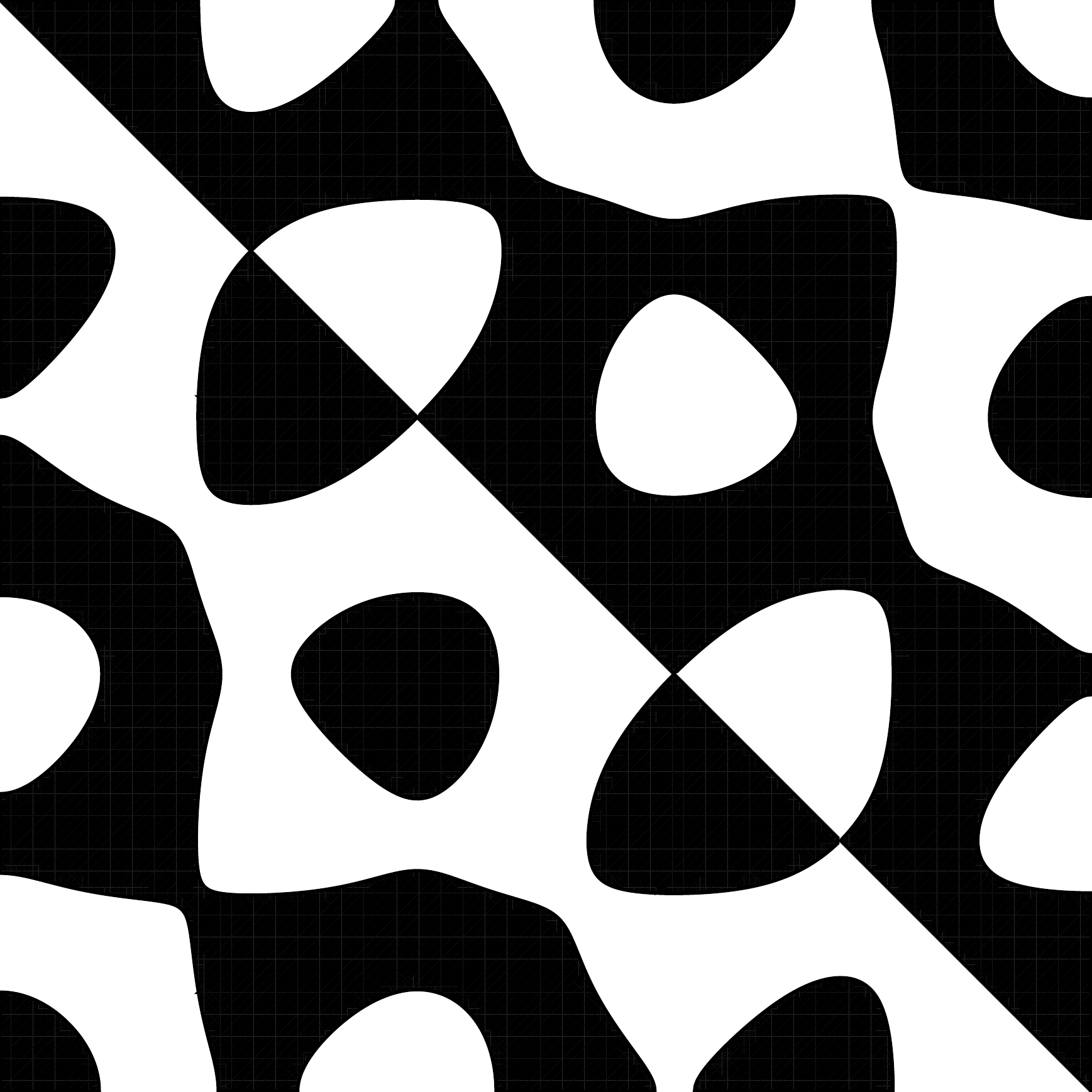}
\hskip 1cm
\includegraphics[width=6cm]{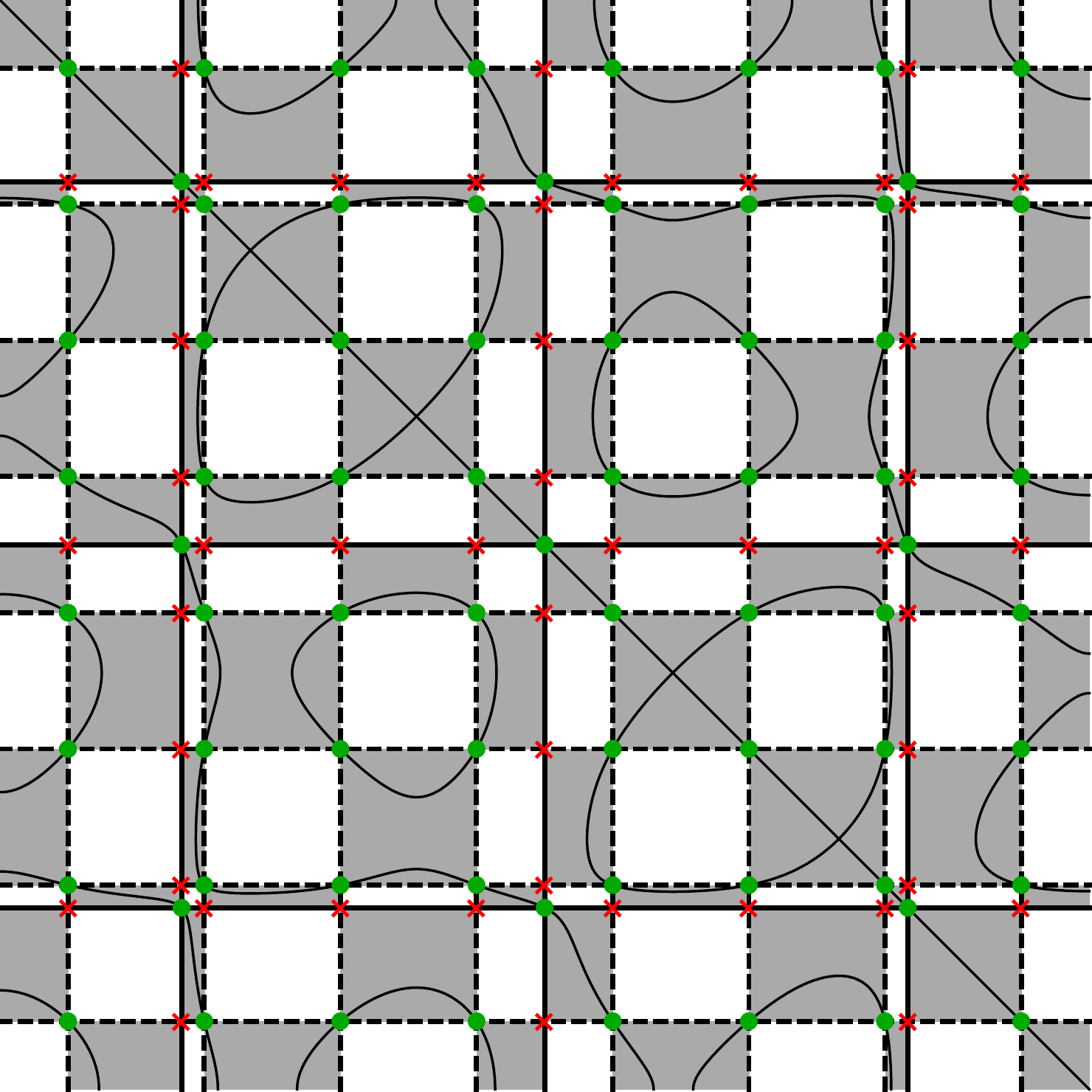}
}
\caption{Nodal domains for $(p,q)= (8,3)$ and $\theta=\pi/4$ (22 nodal domains).}
\label{fig:8-3-thetapi4}
\end{figure}

 We keep in mind the results obtained in Section~\ref{Section5} concerning 
the chessboard localization and its consequences.\\
We now consider $\theta$ in the interval $(0,\theta_1)$. We know that there 
are no critical points inside. Hence one line entering in a rectangle by one 
of the corner belonging to the nodal set should exit the black rectangle by 
another corner or by the boundary. Conversely, a line starting from the 
boundary should leave the black rectangle through a corner in the zero set.

We now look at the points on the boundary. For $x=0$, we have shown that there 
are exactly three points $(0,\eta_1)$, $(0,\eta_2)$, and $(0,\eta_3)$. Moreover 
$\eta_1\in ( \frac{ \pi}{6}, \frac{3\pi}{16})$, 
$\eta_2 \in (\frac{\pi}{2}, \frac{9\pi}{16})$, and 
$\eta_3\in (\frac{5\pi}{6}, \frac{15 \pi}{16})$. Similar considerations can be 
done to localize the  nine  points on $y=0$, $\xi_1, \dots,\xi_8$,  and on 
$y=\pi$, $\xi'_1,\dots,\xi'_8$, and  three points on $x=\pi$, 
$ \eta'_1, \eta'_2,\eta'_3$. These localizations are independent of 
$\theta \in (0,\theta_1)$.

Let us see if one can reconstruct uniquely  the nodal picture using these 
rules. The nodal line starting from $(0,\eta_1)$  has no other choice than going 
through one admissible corner to the point $(\xi_1,0)$. The nodal line starting 
from $(0,\eta_2)$ has no other choice than going to $(\xi_2,0)$ after passing 
through five admissible corners.  The curve starting from $(\xi_3,0)$ has no 
other choice than coming back to the same boundary at $(\xi_4,0)$ after  
passing through two admissible corners. Similarly, the line starting from 
$(\xi_6,0)$ has no other choice than coming back to the same boundary at 
$(\xi_7,0)$ after  passing through two admissible corners.  For the nodal 
line starting from $(0,\eta_3)$, the first five admissible corners to visit 
are uniquely determined by the given rules. Then the line enters in a rectangle 
with four admissible corners. There are two choices for leaving this rectangle. 
The determination of  the right admissible corner can be done by using 
perturbation theory or a barrier argument. This leads to go down to the left 
down corner. After visiting this one the two next admissible corners are 
uniquely determined. The nodal line enters in a rectangle with four admissible 
corners. Again, we have to use a perturbation argument to decide that we have 
to leave at the admissible left up corner. Then everything is uniquely 
determined till the nodal line touches the boundary at $(\xi_5,0)$. We now use 
the symmetry with respect to the diagonal to draw three new nodal lines.

The  last line joining  $(\xi'_1,\pi)$ to $(\xi_8,0)$ is then uniquely 
determined. In this way we get twelve nodal domains.

The case $\theta=\theta_1$ corresponds to a change at the boundary. Instead 
of three lines touching at the boundary at $x=0$ and $x=\pi$, 
 two new lines touch the boundary at the same point at $x=0$ between the 
former $(0,\eta_1)$ and $(0,\eta_2)$ (resp at $x=\pi$ between $(\pi,\eta'_2)$ 
and $(\pi,\eta'_3)$). The number of nodal domains becomes equal to $14$.
 
 For $\theta\in (\theta_1,\theta_{13})$, nothing has changed except that we 
have now exactly five points at $x=0$ and five points at $x=\pi$. The number of 
nodal domains is constant and equal to $14$.
 
 For $\theta = \theta_{13}$,  two critical points appear inside the square 
leading to the creation of two new nodal domains. We have now sixteen nodal 
domains. Nothing has changed at the boundary.
  
For $\theta \in ( \theta_{13}, \theta_2)$, the two critical points disappear.
Nothing changes at the boundary and we keep $16$ nodal domains.
 
 For $\theta = \theta_2$, two new lines touch the boundary at the same point 
at $x=0$ and similarly at $x=\pi$. This creates two new nodal domains. We now 
get $18$ nodal domains.
 
  For $\theta \in (\theta_{13},\frac \pi 4)$, we have now seven touching 
points at $x=0$ and seven at $x=\pi$. The number of nodal domains remains 
equal to $18$.
  
  Finally, for $\theta=\frac \pi 4$, four critical points appear on the 
anti-diagonal. Four nodal domains are created. We have now $22$ nodal domains.

This ends the (sketch of) the proof of Proposition \ref{lem:8-3-opt}.

\begin{figure}[htp]
\centering
\makebox[\textwidth]{
\includegraphics[width=3cm]{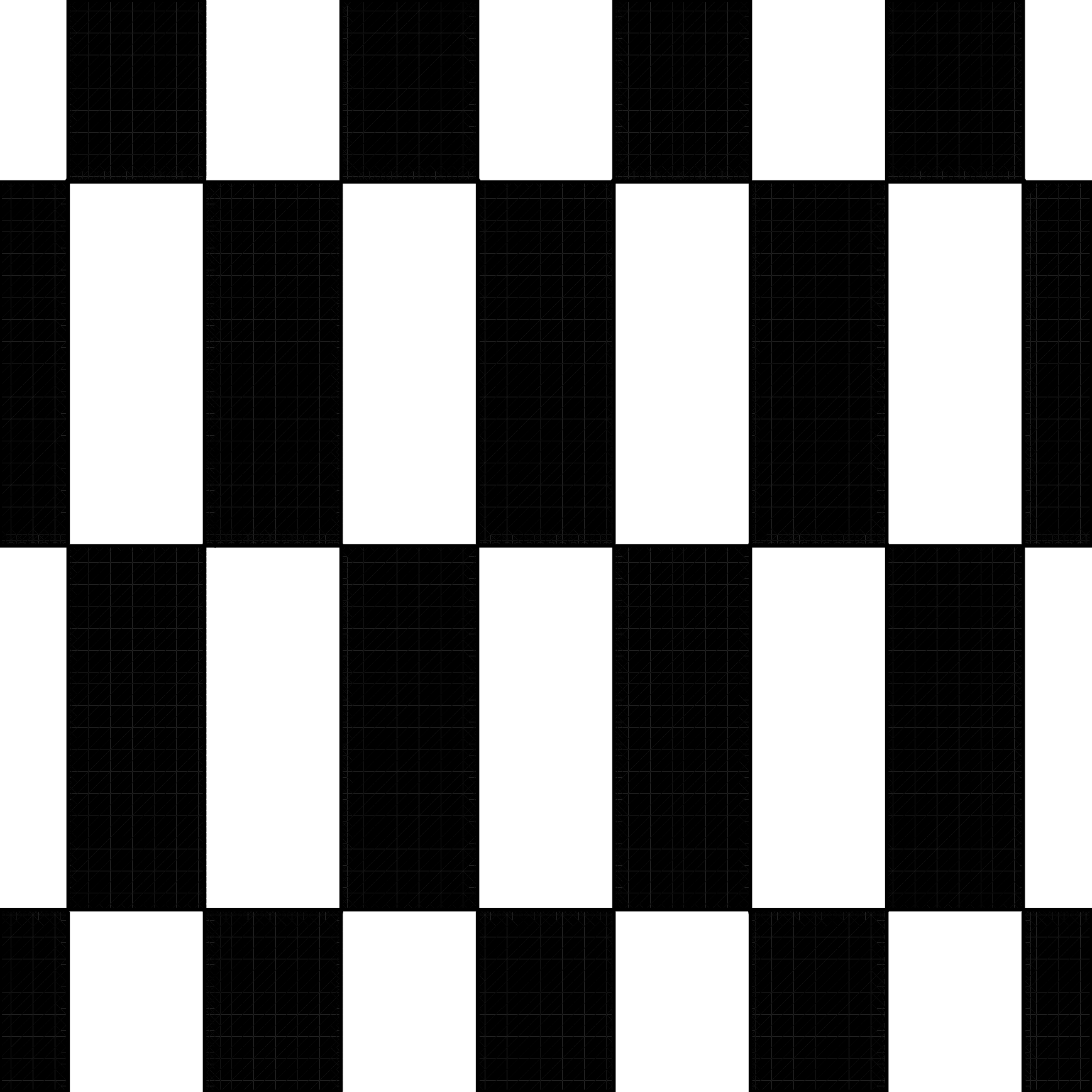}
\hskip .5cm
\includegraphics[width=3cm]{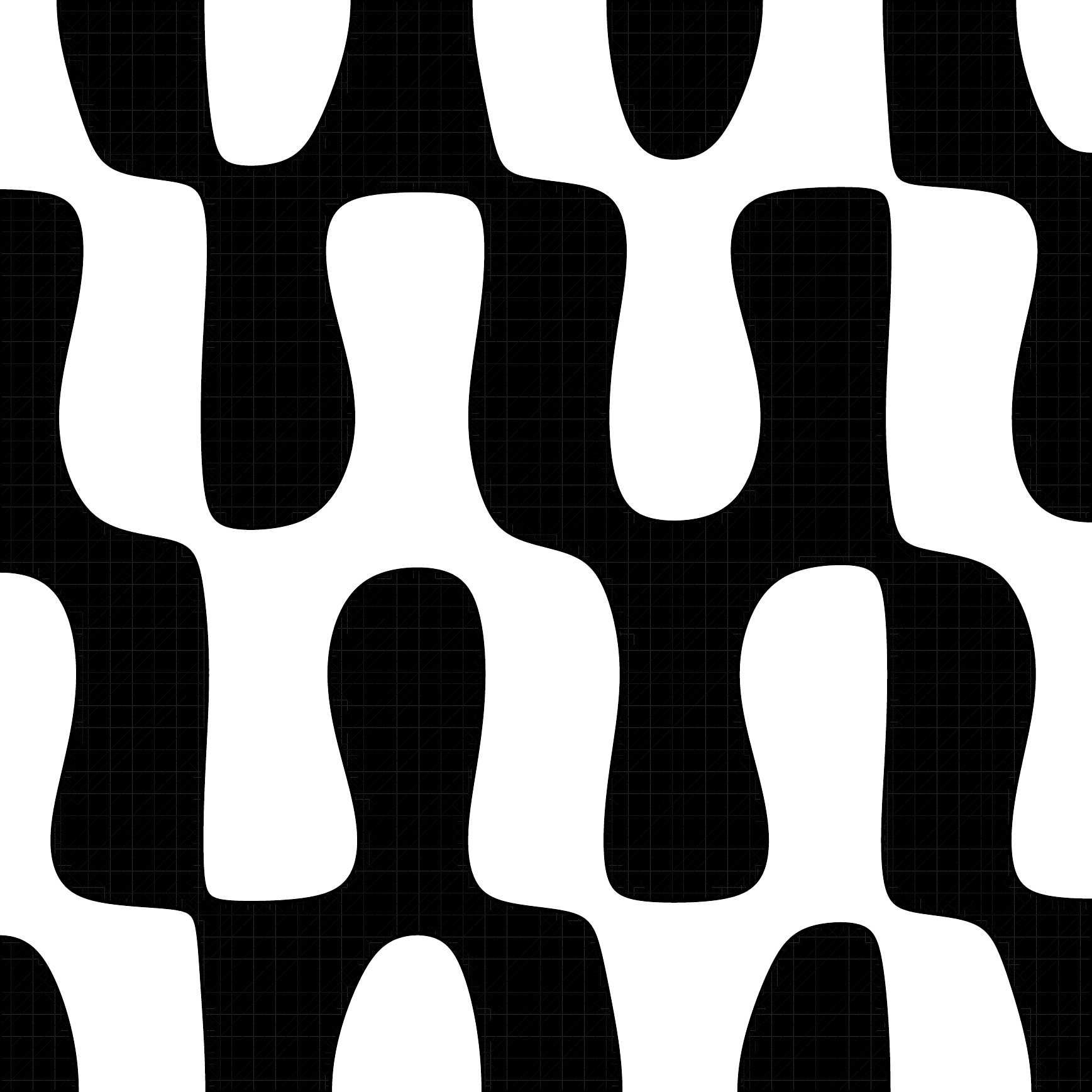}
\hskip .5cm
\includegraphics[width=3cm]{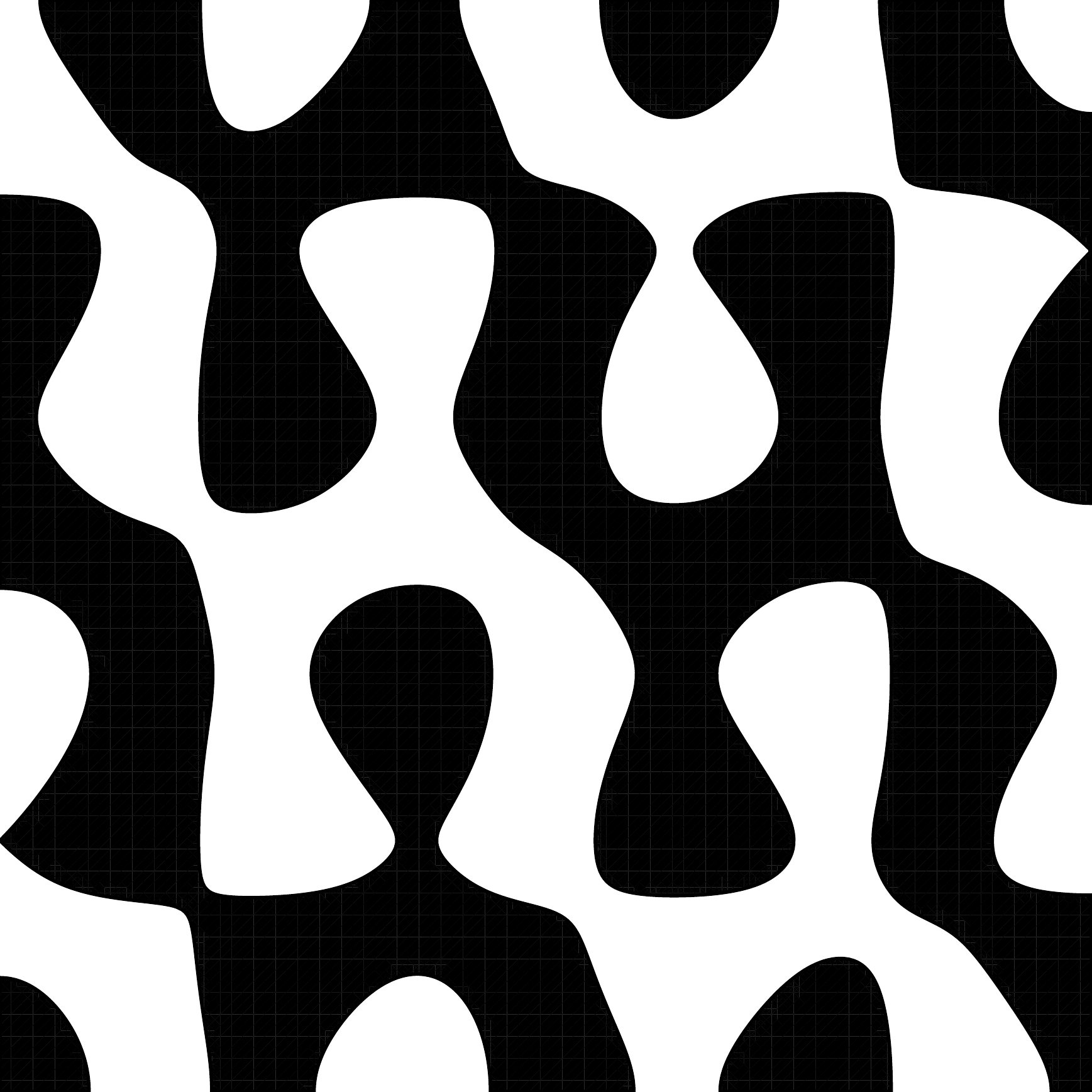}
}
\vskip 0.5cm
\makebox[\textwidth]{
\includegraphics[width=3cm]{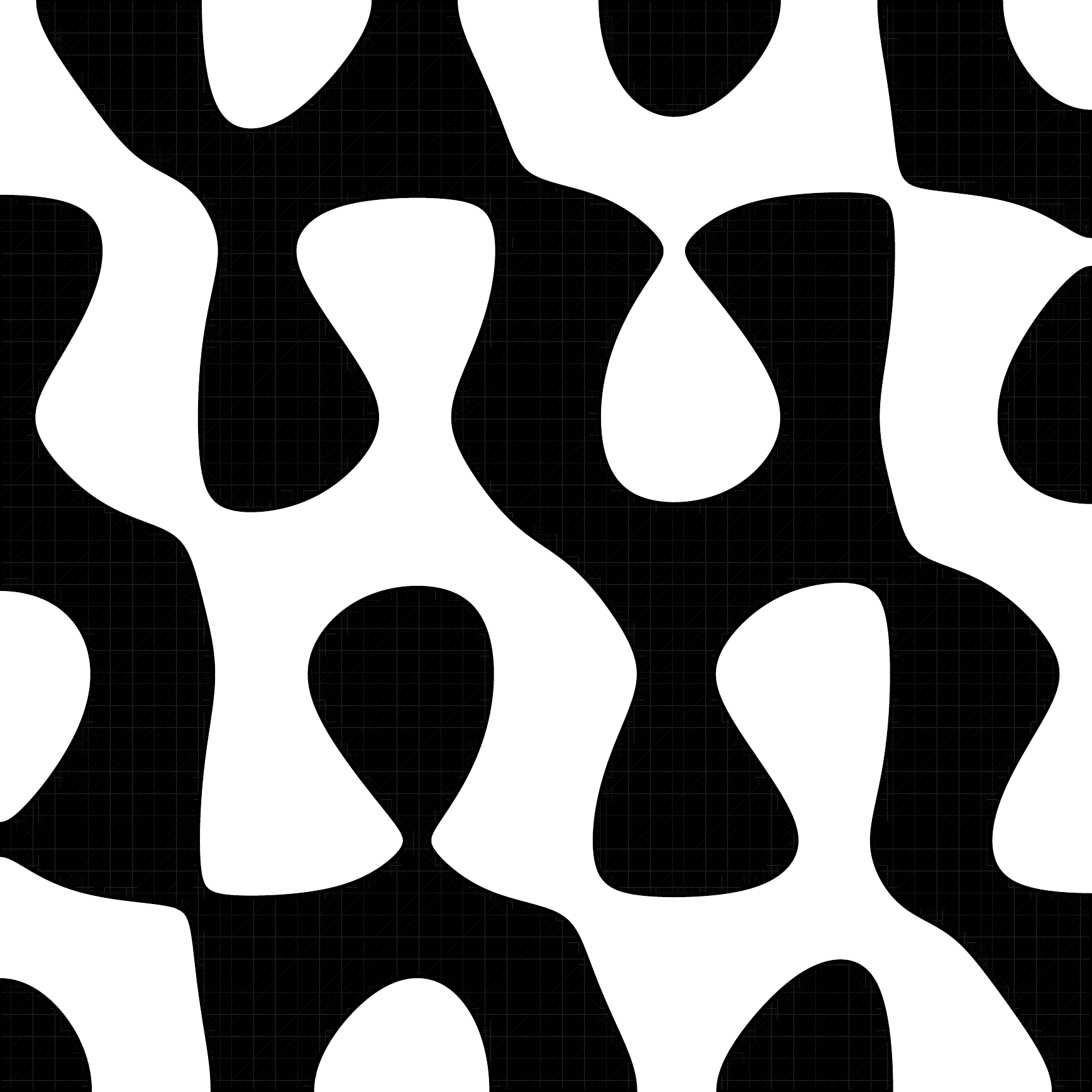}
\hskip .5cm
\includegraphics[width=3cm]{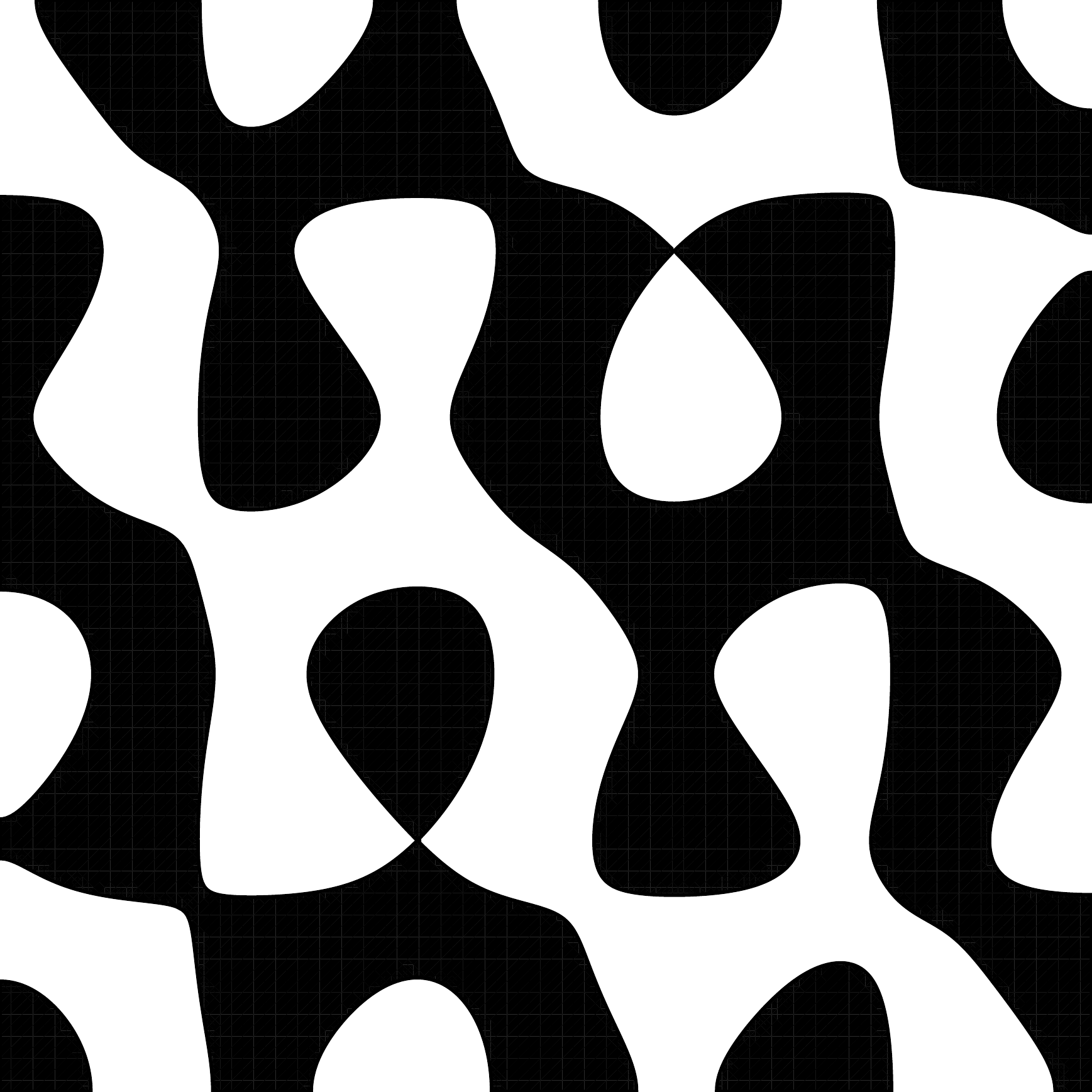}
\hskip .5cm
\includegraphics[width=3cm]{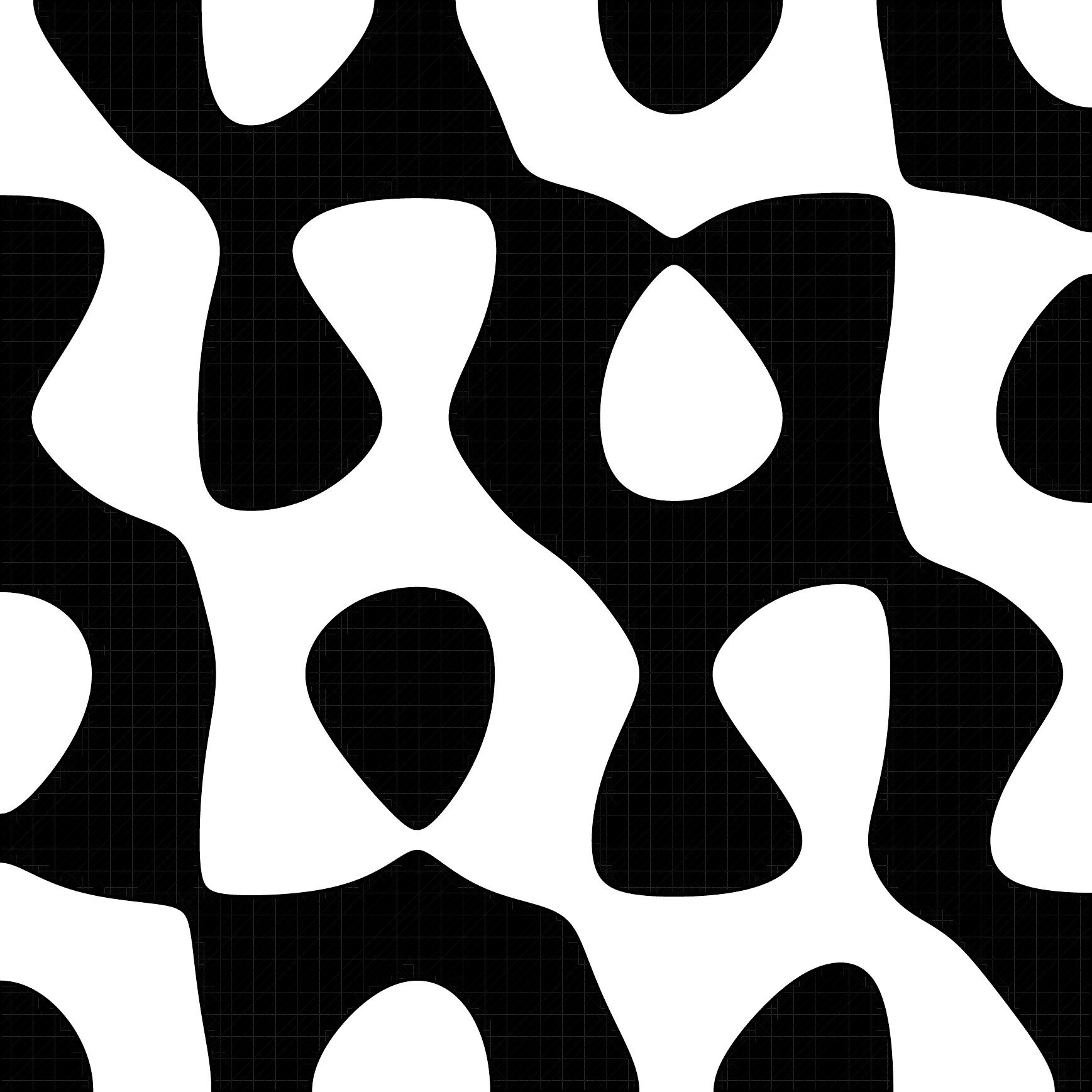}
}
\vskip 0.5cm
\makebox[\textwidth]
{
\includegraphics[width=3cm]{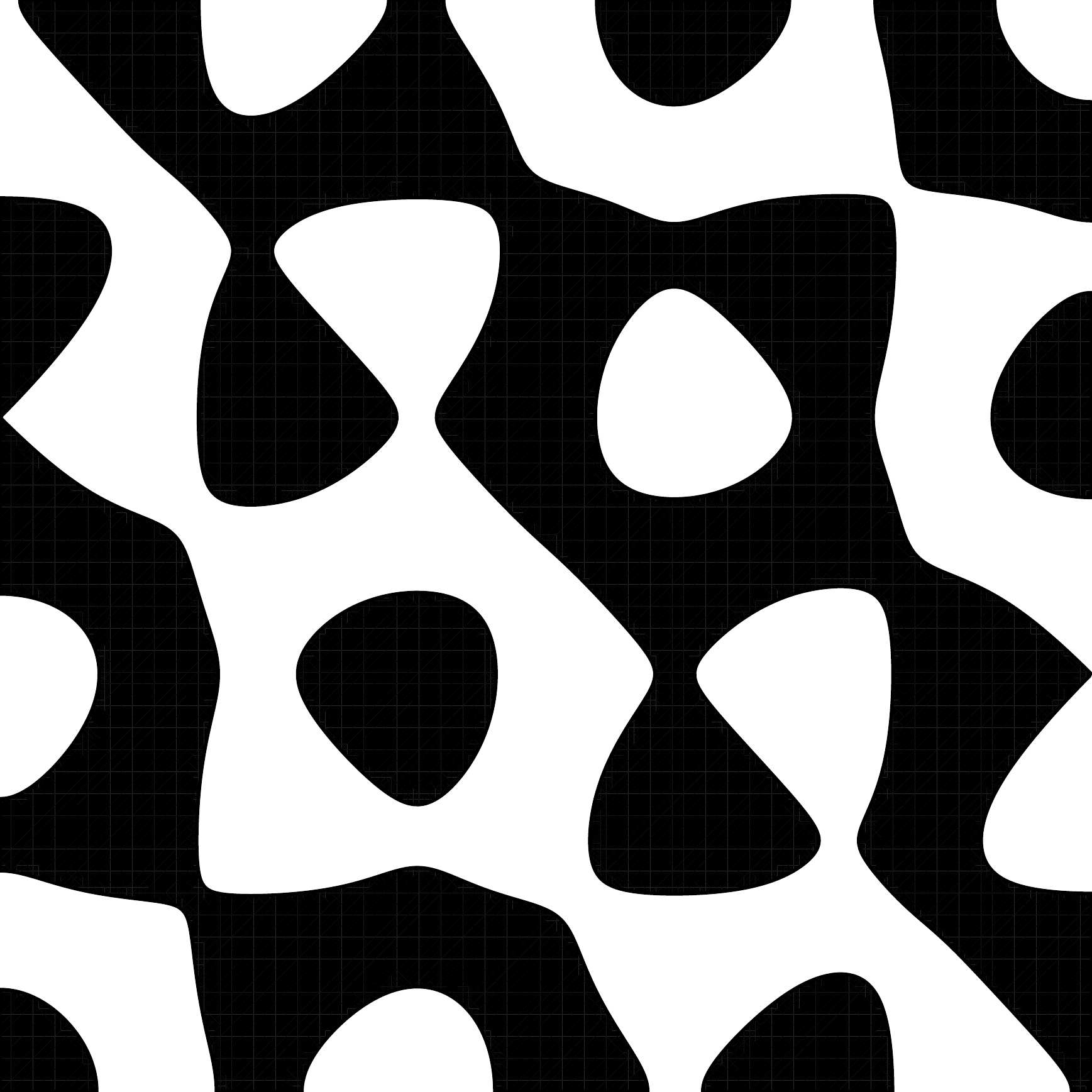}
\hskip .5cm
\includegraphics[width=3cm]{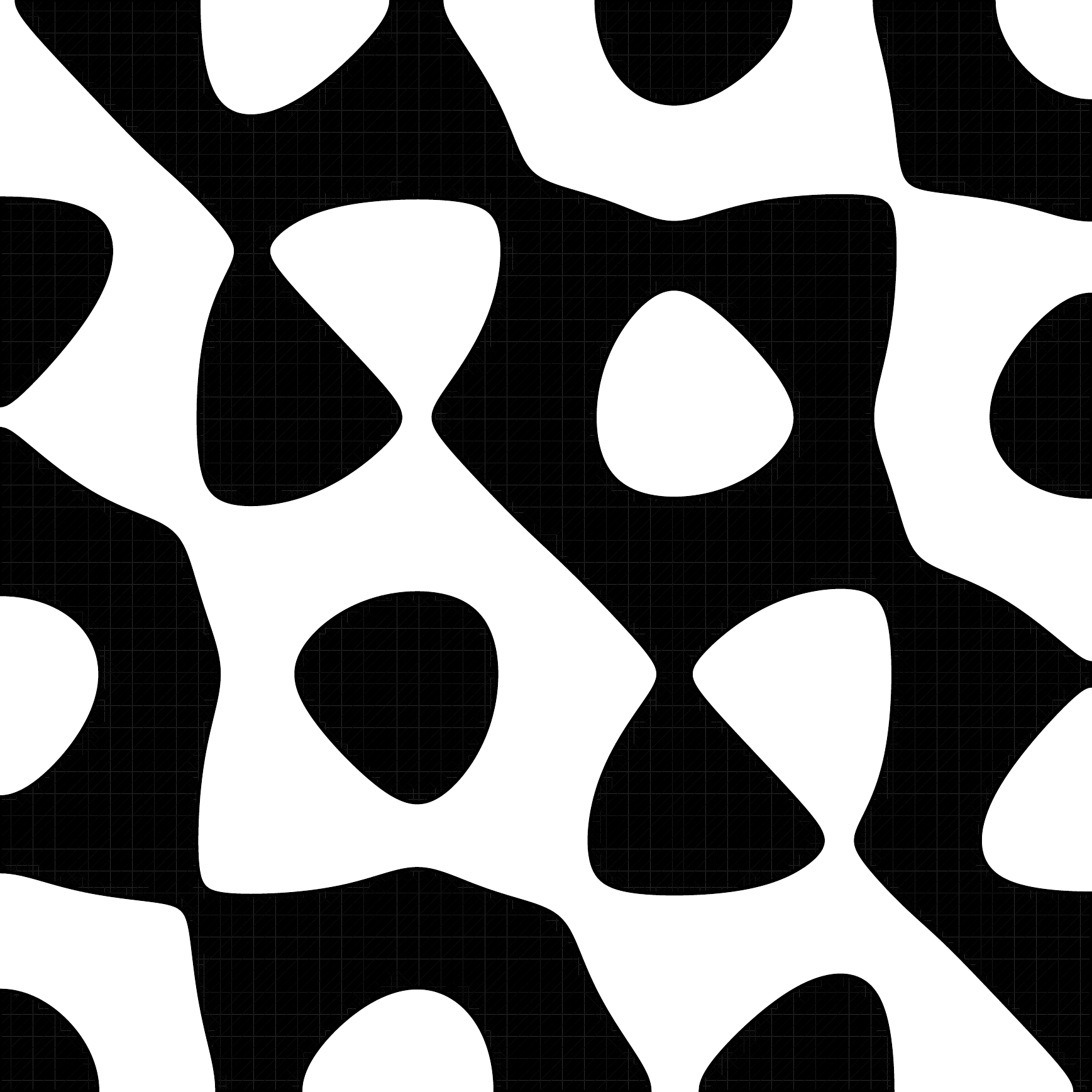}
\hskip .5cm
\includegraphics[width=3cm]{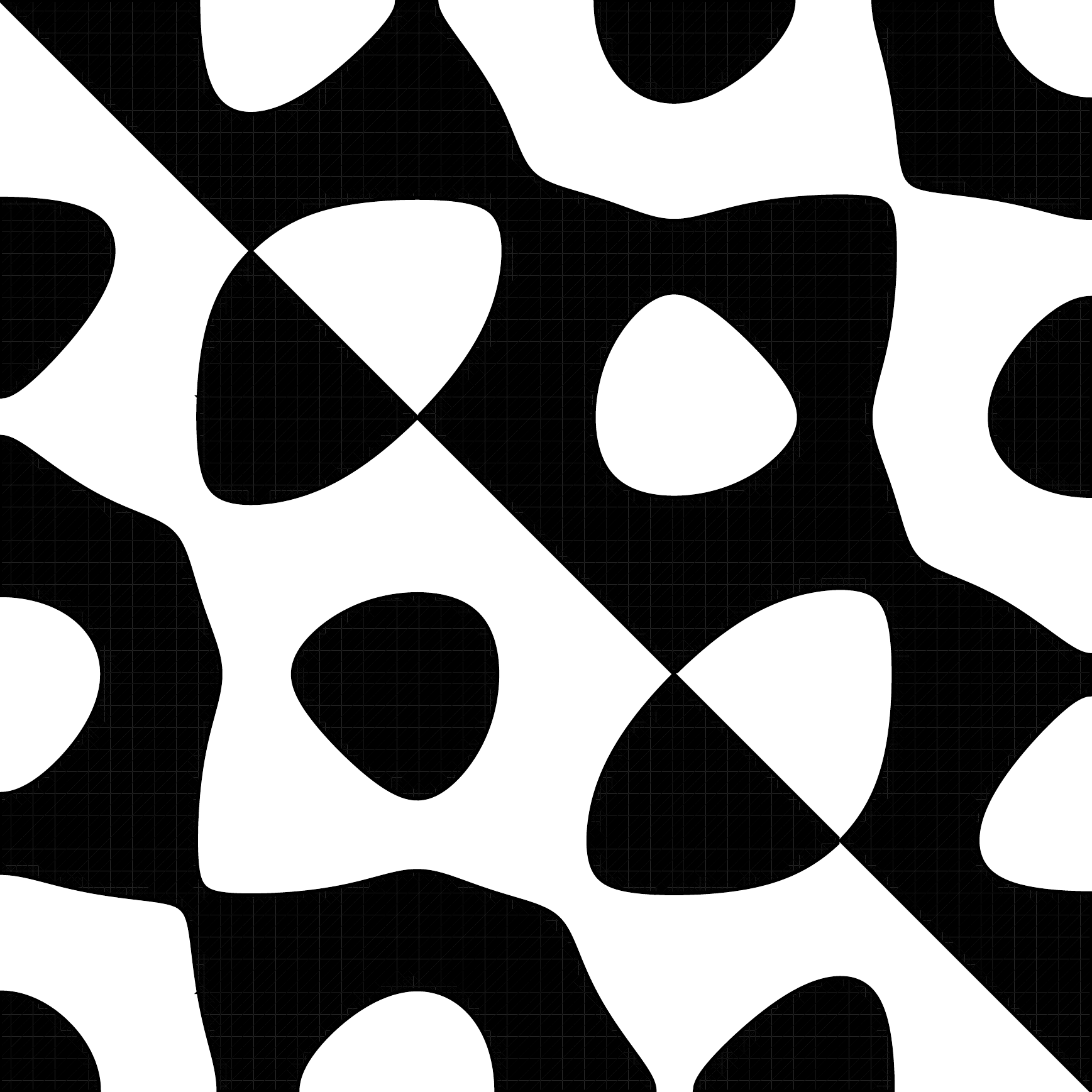}
}
\caption{The graphs show nodal domains in the case $(p,q)=(8,3)$. From upper 
left to lower right, $\theta=0$, $\theta=2\pi/25$, $\theta=\theta_1\approx 0.57$, 
$\theta=19\pi/100$, $\theta=\theta_{13}\approx 0.61$, $\theta=\pi/5$, 
$\theta=\theta_2\approx 0.74$, $\theta=6\pi/25$ and $\theta=\pi/4$.}
\label{fig:8-3}
\end{figure}

\section{Table}
\label{sec:table}

\begin{center}

\tablehead{%
\toprule
$n$ & $p$ & $q$ &  $\lambda_n$ & $n_{\text{ARot}}$  &$n_{\text{SRot}}$ & $n_{\text{AMir}}$ & Comment \\
\midrule
}
\begin{supertabular}{cccccccc}
 1 & 0 & 0 & 0 & & 1 && Courant sharp \\
\midrule
 2--3 & 1 & 0 & 1 & 1--2 & & &Courant sharp \\
  & 0 & 1 & & & & & --- \\
\midrule
 4 & 1 & 1 & 2 &&2& 1 &Courant sharp \\
\midrule
 5--6 & 2 & 0 & 4 &&3--4& & Courant sharp \\
 & 0 & 2 &  && & &--- \\
\midrule
 7--8 & 2 & 1 & 5 &3--4& & &Lemma~\ref{lem:antisymmetric}\\
 & 1 & 2 & && & & --- \\
\midrule
 9 & 2 & 2 & 8 &&5 & &Courant sharp \\
\midrule
 10--11 & 3 & 0 & 9 &5--6&&& Lemma~\ref{lem:p0} \\
 & 0 & 3 & && &&--- \\
\midrule
 12--13 & 3 & 1 & 10 &&6--7& 2--3 & Lemma~\ref{lem:antimirror} \\
 & 1 & 3 & && & & --- \\
\midrule
 14--15 & 3 & 2 & 13 &7--8&& &  Lemma~\ref{lem:3-2} \\
 & 2 & 3 & &&& & --- \\
\midrule
 16--17 & 4 & 0 & 16 &&8--9&& Lemma~\ref{lem:p0} \\
 & 0 & 4 & &&& &--- \\
\midrule
 18--19 & 4 & 1 & 17 &9--10&& & Lemma \ref{lem:4-1}\\
 & 1 & 4 & && && --- \\
\midrule
 20 & 3 & 3 & 18 &&10& 4& Lemma~\ref{lem:pp}\\
\midrule
 21--22 & 4 & 2 & 20 &&11--12& &  Lemma~\ref{lem:4-2} \\
 & 2 & 4 & && && --- \\
\midrule
 23--26 & 5 & 0 & 25 & 11--14&&& Lemma~\ref{lem:antisymmetric} \\
 & 4 & 3 & &&&& --- \\
 & 3 & 4 & &&&& --- \\
 & 0 & 5 & &&&& --- \\
\midrule
 27--28 & 5 & 1 & 26 &&13--14& 5--6& Lemma~\ref{lem:antimirror} \\
 & 1 & 5 & &&&& --- \\
\midrule
 29--30 & 5 & 2 & 29 & 15--16&&& Lemma~\ref{lem:antisymmetric}  \\
 & 2 & 5 & &&&& --- \\
\midrule
 31 & 4 & 4 & 32 &&15 && Lemma~\ref{lem:pp} \\
\midrule
 32--33 & 5 & 3 & 34 &&16--17& 7--8& Lemma~\ref{lem:antimirror} \\
 & 3 & 5 & && && --- \\
\midrule
 34--35 & 6 & 0 & 36 &&18--19& & Lemma~\ref{lem:p0} \\
 & 0 & 6 & &&&& --- \\
\midrule
 36--37 & 6 & 1 & 37 & 17--18 &&& Lemma~\ref{lem:antisymmetric} \\
 & 1 & 6 & &&&& --- \\
\midrule
 38--39 & 6 & 2 & 40 &&20--21&& Lemma~\ref{lem:pandqeven} \\
 & 2 & 6 & &&&& --- \\
\midrule
 40--41 & 5 & 4 & 41 &19--20&&& Lemma~\ref{lem:antisymmetric} \\
 & 4 & 5 & &&&& --- \\
\midrule
 42--43 & 6 & 3 & 45 & 21--22 &&& Lemma~\ref{lem:6-3} \\
 & 3 & 6 & & &&& --- \\
\midrule
 44--45 & 7 & 0 & 49 & 23--24 &&& Lemma~\ref{lem:p0} \\
 & 0 & 7 & &&&& --- \\
\midrule
 46--48 & 7 & 1 & 50 &&22--24& 9--11&Lemma~\ref{lem:antimirror} \\
 & 5 & 5 & &&&& --- \\
 & 1 & 7 & &&&& --- \\
\midrule
 49--50 & 6 & 4 & 52 & &25--26&&  Lemma~\ref{lem:6-4} \\
 & 4 & 6 & &&&& --- \\
\midrule
 51--52 & 7 & 2 & 53 & 25--26 && &Lemma~\ref{lem:antisymmetric}\\
 & 2 & 7 & &&&& --- \\
\midrule
 53--54 & 7 & 3 & 58 &&27--28&12--13 &Lemma~\ref{lem:antimirror} \\
 & 3 & 7 & &&& --- \\
\midrule
 55--56 & 6 & 5 & 61 & 27--28 &&& Lemma~\ref{lem:antisymmetric}\\
 & 5 & 6 & &&&& --- \\
\midrule
 57--58 & 8 & 0 & 64 &&29--30&& Lemma~\ref{lem:p0}\\
 & 0 & 8 & &&&& --- \\
\midrule
 59--62 & 8 & 1 & 65 & 29--32 & && Lemma~\ref{lem:antisymmetric}\\
 & 7 & 4 & && &&--- \\
 & 4 & 7 & && &&--- \\
 & 1 & 8 & &&&& --- \\
\midrule
 63--64 & 8 & 2 & 68 &&31--32&&Lemma~\ref{lem:symmetric}  \\
 & 2 & 8 & &&&& --- \\
\midrule
 65 & 6 & 6 & 72 &&33&& Lemma~\ref{lem:pp} \\
\midrule
 66--67 & 8 & 3 & 73 & 33--34 &&& Lemma~\ref{lem:8-3}  \\
 & 3 & 8 & && && --- \\
\midrule
 68--69 & 7 & 5 & 74 &&34--35 & 14--15 &Lemma~\ref{lem:antimirror} \\
 & 5 & 7 & &&&&  --- \\
\midrule
 70--71 & 8 & 4 & 80 &&36--37&& Lemma~\ref{lem:8-4}  \\
 & 4 & 8 & &&&& --- \\
\midrule
 72--73 & 9 & 0 & 81 & 35--36 &&& Lemma~\ref{lem:p0} \\
 & 0 & 9 & & &&& --- \\
\midrule
 74--75 & 9 & 1 & 82 &&38--39& 16--17 & Lemma~\ref{lem:antimirror} \\
 & 1 & 9 & &&&& --- \\
\midrule
 76--79 & 9 & 2 & 85 & 37--40 && &Lemma~\ref{lem:antisymmetric} \\
 & 7 & 6 & &&&& --- \\
 & 6 & 7 & &&&& --- \\
 & 2 & 9 & &&&& --- \\
\midrule
 80--81 & 8 & 5 & 89 & 41--42 && 18--19 & Lemma~\ref{lemma2.3}  \\
 & 5 & 8 & & &&& --- \\
\midrule
 82--83 & 9 & 3 & 90 && 40--41& 20--21 &Lemma~\ref{lem:antimirror} \\
 & 3 & 9 & && &&--- \\
\midrule
 84--85 & 9 & 4 & 97 & 43--44 &&& Lemma \ref{lem:9-4}  \\
 & 4 & 9 & & &&&  --- \\
\midrule
 86 & 7 & 7 & 98 && 42 & 22 & Lemma~\ref{lem:antimirror}  \\
\midrule
 87--90 & 10 & 0 & 100 &&43--46& & Lemma~\ref{lem:symmetric}  \\
 & 8 & 6 & &&&& --- \\
 & 6 & 8 & &&&& --- \\
 & 0 & 10 & &&&& --- \\
\midrule
 91--92 & 10 & 1 & 101 & 45--46 &&& Lemma~\ref{lem:antisymmetric} \\
 & 1 & 10 & &&&& --- \\
\midrule
 93--94 & 10 & 2 & 104 &&47--48&& Lemma~\ref{lem:pandqeven}\\
 & 2 & 10 & &&&& --- \\
\midrule
 95--96 & 9 & 5 & 106 &&49--50& 23--24 & Lemma~\ref{lemma2.3}  \\
 & 5 & 9 & &&&& ---  \\
\midrule
 97--98 & 10 & 3 & 109 &47--48&&  &Lemma~\ref{lem:antisymmetric}\\
 & 3 & 10 & &&&& --- \\
\midrule
 99--100 & 8 & 7 & 113 & 49--50&& &Lemma~\ref{lemma2.3}\\
 & 7 & 8 & &&& & ---  \\
\midrule
 101--102 & 10 & 4 & 116 &&51--52&& Lemma \ref{lem:10-4} \\
 & 4 & 10 & &&& &--- \\
\midrule
 103--104 & 9 & 6 & 117 & 51--52 & &&Lemma~\ref{lemma2.3}\\
 & 6 & 9 & & & && ---  \\
\midrule
 105--106 & 11 & 0 & 121 & 53--54 & &&Lemma~\ref{lem:p0} \\
 & 0 & 11 & & & && --- \\
\midrule
 107--108 & 11 & 1 & 122 &&53--54& 25--26 & Lemma~\ref{lem:antimirror} \\
 & 1 & 11 & &&&& --- \\
\midrule
 109--112 & 11 & 2 & 125 & 55--58 &&& Lemma~\ref{lem:antisymmetric} \\
 & 10 & 5 & &&&& --- \\
 & 5 & 10 & &&&& --- \\
 & 2 & 11 & &&&& --- \\
\midrule
 113 & 8 & 8 & 128 & &55& &Lemma~\ref{lemma2.3}  \\
\midrule
 114--117 & 11 & 3 & 130 &&56--59& 27--30 & Lemma~\ref{lem:antimirror} \\
 & 9 & 7 & &&&& --- \\
 & 7 & 9 & &&&& --- \\
 & 3 & 11 & &&&& --- \\
\midrule
 118--119 & 10 & 6 & 136 &&60--61&  &Lemma~\ref{lemma2.3}  \\
 & 6 & 10 & &&&&  ---  \\
\midrule
 120--121 & 11 & 4 & 137 & 59--60 && &Lemma~\ref{lemma2.3}\\
 & 4 & 11 & &&&&  ---  \\
\midrule
 122--123 & 12 & 0 & 144 &&62--63& & Lemma~\ref{lem:p0}\\
 & 0 & 12 & &&&& --- \\
\midrule
 124--127 & 12 & 1 & 145 & 61--64 &&& Lemma~\ref{lem:antisymmetric} \\
 & 9 & 8 & &&&& --- \\
 & 8 & 9 & &&&& --- \\
 & 1 & 12 & &&&& --- \\
\midrule
 128--129 & 11 & 5 & 146 &&64--65& 31--32 & Lemma~\ref{lemma2.3}  \\
 & 5 & 11 & &&&&  --- \\
\midrule
 130--131 & 12 & 2 & 148 &&66--67&  &Lemma~\ref{lemma2.3}  \\
 & 2 & 12 & &&&&  --- \\
\midrule
 132--133 & 10 & 7 & 149 & 65--66 && & Lemma~\ref{lemma2.3}\\
 & 7 & 10 & & &&& ---  \\
\midrule
 134--135 & 12 & 3 & 153 & 67--68 && &Lemma~\ref{lemma2.3}  \\
 & 3 & 12 & & && &---  \\
\midrule
 136--137 & 11 & 6 & 157 & 69--70 && &Lemma~\ref{lemma2.3}  \\
 & 6 & 11 & & && & ---  \\
\midrule
 138--139 & 12 & 4 & 160 & &68--69& & Lemma~\ref{lemma2.3}  \\
 & 4 & 12 & & && &---  \\
\midrule
 140 & 9 & 9 & 162 & &70 & 33 & Lemma~\ref{lemma2.3}  \\
\midrule
 141--142 & 10 & 8 & 164 & &71--72& &Lemma~\ref{lemma2.3}  \\
 & 8 & 10 & &&&&  ---  \\
\midrule
 143--146 & 13 & 0 & 169 & 71--74 && &Lemma~\ref{lem:antisymmetric}\\
 & 12 & 5 & &&&& --- \\
 & 5 & 12 & &&&& --- \\
 & 0 & 13 & &&&& --- \\
\midrule
 147--150 & 13 & 1 & 170 & &73--76& 34--37 & Lemma~\ref{lem:antimirror} \\
 & 11 & 7 & &&&& --- \\
 & 7 & 11 & &&&& --- \\
 & 1 & 13 & &&&& --- \\
\midrule
 151--152 & 13 & 2 & 173 & 75--76 && &Lemma~\ref{lemma2.3}\\
 & 2 & 13 & & &&& --- \\
\midrule
 153--154 & 13 & 3 & 178 & &77--78& 38--39 & Lemma~\ref{lemma2.3} \\
 & 3 & 13 & & &&& --- \\
\midrule
 155--156 & 12 & 6 & 180 & &79--80& &Lemma~\ref{lemma2.3} \\
 & 6 & 12 & & &&& --- \\
\midrule
 157--158 & 10 & 9 & 181 & 77--78 &&& Lemma~\ref{lemma2.3}\\
 & 9 & 10 & & &&& --- \\
\midrule
 159--162 & 13 & 4 & 185 & 79--82 &&& Lemma~\ref{lem:antisymmetric}\\
 & 11 & 8 & & &&& --- \\
 & 8 & 11 & & &&& --- \\
 & 4 & 13 & & &&& --- \\
\midrule
 163--164 & 12 & 7 & 193 & 83--84 && &Lemma~\ref{lemma2.3} \\
 & 7 & 12 & & &&& --- \\
\midrule
 165--166 & 13 & 5 & 194 & &81--82& 40--41 & Lemma~\ref{lemma2.3} \\
 & 5 & 13 & & &&& --- \\
\midrule
 167--168 & 14 & 0 & 196 & & 83--84&& Lemma~\ref{lem:p0} \\
 & 0 & 14 & & &&& --- \\
\midrule
 169--170 & 14 & 1 & 197 & 85--86 &&& Lemma~\ref{lemma2.3} \\
 & 1 & 14 & & &&& --- \\
\midrule
 171--173 & 14 & 2 & 200 & &85--87& &Lemma~\ref{lemma2.3} \\
 & 10 & 10 & & &&& --- \\
 & 2 & 14 & & &&& --- \\
\midrule
 174--175 & 11 & 9 & 202 & &88--89& 42--43 & Lemma~\ref{lemma2.3} \\
 & 9 & 11 & & &&& --- \\
\midrule
 176--179 & 14 & 3 & 205 & 87--90 &&& Lemma~\ref{lemma2.3}\\
 & 13 & 6 & & &&& --- \\
 & 6 & 13 & & && &--- \\
 & 3 & 14 & & &&& --- \\
\midrule
 180--181 & 12 & 8 & 208 & &90--91& &Lemma~\ref{lemma2.3} \\
 & 8 & 12 & & &&& --- \\
\midrule
 182--183 & 14 & 4 & 212 & &92--93&& Lemma~\ref{lemma2.3} \\
 & 4 & 14 & & &&& --- \\
\midrule
 184--185 & 13 & 7 & 218 & &94--95& 44--45 & Lemma~\ref{lemma2.3} \\
 & 7 & 13 & & && &--- \\
\midrule
 186--189 & 14 & 5 & 221 & 91--94 && &Lemma~\ref{lemma2.3}\\
 & 11 & 10 & &&&& --- \\
 & 10 & 11 & &&&& --- \\
 & 5 & 14 & &&&& --- \\
\midrule
 190--193 & 15 & 0 & 225 & 96--99 &&& Lemma~\ref{lemma2.3} \\
 & 12 & 9 & & &&& --- \\
 & 9 & 12 & & &&& --- \\
 & 0 & 15 & & &&& --- \\
\midrule
 194--195 & 15 & 1 & 226 & &95--96& 46--47 & Lemma~\ref{lemma2.3} \\
 & 1 & 15 & & &&& --- \\
\midrule
 196--197 & 15 & 2 & 229 & 99--100 &&& Lemma~\ref{lemma2.3} \\
 & 2 & 15 & & &&& --- \\
\midrule
 198--199 & 14 & 6 & 232 & &97--98& &Lemma~\ref{lemma2.3} \\
 & 6 & 14 & & &&& --- \\
\midrule
 200--201 & 13 & 8 & 233 & 101--102 & &&Lemma~\ref{lemma2.3} \\
 & 8 & 13 & & &&& --- \\
\midrule
 202--203 & 15 & 3 & 234 & &100--101& 48--49 &  Lemma~\ref{lemma2.3} \\
 & 3 & 15 & & &&& --- \\
\midrule
 204--205 & 15 & 4 & 241 & 103--104 && &Lemma~\ref{lemma2.3} \\
 & 4 & 15 & & && &--- \\
\midrule
 206 & 11 & 11 & 242 & &102&  50 & Lemma~\ref{lemma2.3} \\
\midrule
 207--208 & 12 & 10 & 244 & & 103--104&&Lemma~\ref{lemma2.3} \\
 & 10 & 12 & & & &&--- \\
\midrule
 $\geq 209$ & & &  $\geq 245$ & & &&Proposition~\ref{prop:red1}\\
\bottomrule
\end{supertabular}
\end{center}

\newpage

\section{Open problems}
From the numerics together with our mathematical analysis for specific 
eigenspaces, it seems reasonable to correct 
some traveling folk conjecture into the following one:
\begin{conjecture}
In a given eigenspace of dimension $2$, the maximal number of nodal domains is 
obtained for at least one eigenfunction $\Phi_{p,q}^\theta$
for some $\theta \in \{0,\frac \pi 4,\frac \pi 2,\frac {3\pi}{4}, \pi\}$.
\end{conjecture}

The numerical work of Corentin L\'{e}na~\cite{Len} devoted to the analysis of spectral
minimal partitions (showing non nodal $k$-minimal partitions starting from 
$k\geq 3$) suggests that there are only two Courant sharp 
situations. The case of the isotropic torus has finally been solved quite 
recently by C.~Lena~\cite{Len2}.  Following the strategy of 
{\AA}.~Pleijel~\cite{Pl}, his proof is a combination of  a lower bound 
(\`a la Weyl) of the counting function with an explicit remainder term and of 
a Faber--Krahn inequality for connected domains on the torus with an explicit 
upper bound on the area.

It is also natural to ask if there are similar results to the results 
concerning the Dirichlet problem on the square considered by A.~Stern and 
B{\'e}rard--Helffer, that is the existence of an infinite sequence of 
eigenvalues such that a corresponding eigenfunction has only two nodal domains. 
We conjecture that it is impossible to find such a sequence. To justify this 
guess,  one can  try to show
that in the Neumann case, the number of nodal lines 
touching the boundary tends to $+\infty$ as the eigenvalue tends to 
$+\infty$.
This has a nice connection with a recent result of T.~Hoffmann-Ostenhof~\cite
{HO}, saying that the only eigenfunction whose nodal set does not touch the 
boundary is the constant one.

At the moment, we can only prove the following:
\begin{prop}
\label{prop:touchingpoints}
Let $(p,q)\in \mathbb N^*\times \mathbb N$ with $p>q$.
Then, for any $\theta\in [0,\pi]$  the nodal lines of the 
eigenfunction $\Phi^\theta_{p,q}$ have 
at least $2p + 2q$ touching points at the boundary.
\end{prop}

We first  prove the proposition, with the additional assumption that $\cos px$ 
and $\cos qx$ have no common zeros in $ [0,\pi]$.

As we have seen in Subsection~\ref{ss5.2}, the analysis of the zeros on 
the boundary is immediately related with the investigation of the solutions of
\[
 \cos px = t \cos qx \quad \text{or}\quad 
 \cos py =t \cos qy 
\]
with $t = \pm \tan \theta$ or $t= \pm \frac{1}{\tan \theta}$.
The result is then a consequence of the following lemma:
\begin{lemma}
 Let $(p,q)\in \mathbb N^*\times \mathbb N$ with $p>q$. Suppose that 
$\cos px$ and $\cos qx$ have no common zeros in $ [0,\pi]$. Then,  
for any $t\in [-1,+1]$, the equation
\[
\cos px =  t \cos qx
\]
has exactly $p$ solutions in $[0,\pi]$.
\end{lemma}

\begin{proof}~\\
{\bf We first observe that there are at most $p$ solutions.} Indeed, if we choose $u=\cos x$ as new 
variable, we obtain a polynomial equation in the variable $u$ of degree $p$,  
\[
P_p (u) = t P_q (u)\,,
\]
where $P_p$ is some  Chebyshev  polynomial defined by
\[
\cos px  = P_p (\cos x)\,.
\]
Hence we get our first observation (counting with multiplicity).

{\bf We now show that there are at least $p$ solutions.} For $t=0$, the solutions are 
the zeros of $x\mapsto \cos px$, that is 
\[
x_k^{(p)}= (2k+1) \frac{\pi} {2p}\,; \quad k=0,\dots, (p-1)\,.
\]
The zeros $x_\ell^{(q)}$ ($\ell=0,\dots, q-1$) of $x\mapsto \cos qx$ will 
play an important role. We will indeed look at the function $f_{p,q}$ introduced 
in~\eqref{deffpq} and they correspond to vertical asymptotes of the graph of 
$f_{p,q}$.

For $t\neq 0$, say $t >0$, we have now to count the number of solutions of 
$f_{p,q}(x)=t$. First we observe that there is (at least) one solution in
$(0,x_0^{(p)})$ and no solution in $(x_{p-1}^{(p)},\pi)$. Moreover $f_{p,q}$ 
is finite there.

We now consider the equation in the interval 
$I_k^{(p)}:=(x_k^{(p)},  x_{k+1}^{(p)})$, for some  $0\leq k \leq p-2$.

For a given interval there are three cases.
\begin{enumerate}
\item 
There is a zero $x_\ell^{(q)}$ in  $I_k^{(p)}$. In this case the range 
of $f_{p,q}$ always contains $(0,+\infty)$ in particular there is always 
at least one point such that $f_{p,q} (x) =t$.
\item 
There are no zeros of $\cos qx$ in $I_k^{(p)}$ and 
$(-1) ^{(k+1)} \cos q x >0$. We observe in this case that at 
$\hat x_k^{(p)}=(k+1) \frac \pi p$ which belongs to $I_k^{(p)}$, we have 
\begin{equation}\label{geq1}
f_{p,q} (\hat x_k^{(p)})= (-1)^{k+1} /\cos (q \hat x_k^{(p)}) \geq 1\,.
\end{equation}
We will 
see below that the inequality is strict when $p$ and $q$ are mutually prime. 
In the case when we have equality, we have $\sin (q \hat x_k^{(p)}) =0$ and 
we get from~\eqref{pt=qt} that we are at a local extremum of $f_{p,q}$.

\item 
There are no zeros of $\cos qx$ in $I_k^{(p)}$ and 
$(-1) ^{(k+1)} \cos q x < 0$.  In this case the guess is that there are 
no zeros. We will get it at the end of the argument but the information is 
not needed for our lower bound.
\end{enumerate}

To complete the lower bound we have simply to verify that for two 
intervals $I_k^{(p)}$, $I_{k+1}^{(p)}$
not containing a zero of $\cos qx$ we are  either in a sequence case (2), 
case (3) or in a sequence case (3), case (2). This implies that we have  for $|t| < 1$
at least two solutions in the union of the two intervals. For $|t|=1$ the argument is the same if the inequality is strict in \eqref{geq1} and we have a double point if there is an equality (this will correspond to a critical point at the boundary).

If one of the intervals, say $I_{k+1}^{(p)}$, contains a zero of $\cos qx$, 
we play the same game as before but with the pair $I_k^{(p)}$, $I_{k+2}^{(p)}$.

Summing up we get the lower bound by $p$.
Hence we have exactly $p$ zeros. 
\end{proof}

\begin{proof}[Proof of Proposition~\ref{prop:touchingpoints} (with additional assumption)]
We can now finish the proof of the proposition under the additional assumption that $\cos px$ and $\cos qx$ have no common zeros in $ [0,\pi]$. The lemma can be 
applied (depending on $\theta$) either to $x=0$ and $x=\pi$ or to $y=0$ and 
$y=\pi$. For the two other cases, we can use that for $|t| >1$, $f_{p,q}(x)=t$ 
has at least one solution.
\end{proof}

Before attacking the general case, note the following lemma:
\begin{lemma}
If $p$ and $q$ are mutually prime, and $0 \leq  k \leq p-2$,
\[
(-1)^{k+1}\, \cos (q \hat x_k^{(p)})  < 1\,.
\]
\end{lemma}

\begin{proof}
Let us assume $k$ odd. We want to show that $\cos  (k+1) \frac{q\pi}{p} \neq 1$. 
By contradiction, we would have
\[
(k+1) \frac{q}{p} = 2 \ell\,.
\]
This can be written in the form
\[
\frac{(k+1) }{2} q = \ell p\,.
\]
By assumption, $p$ and $q$ are mutually prime. This leads to 
$\ell = q \tilde \ell$, $\frac{k+1}{2} = p \tilde \ell$ for an integer 
$\tilde \ell$. We have now to remember that
$\frac{k+1}{2} \in (\frac 12, \frac{p-1}{2})$. This gives a contradiction.

Let us now assume $k$ even. We want to show that 
$\cos  (k+1) \frac{q\pi}{p} \neq -1$. By contradiction, we would have
\[ 
(k+1) q = (2 \ell +1) p\,.
\]
$p$ and $q$ being mutually prime.  This leads to $k+1 = \tilde \ell p$. This is 
again impossible because $1 \leq (k+1)\leq p-1$. 
\end{proof}

\begin{remark}
This implies equality in all the lower 
bounds of the second part. This implies also  $ p-q-1$ local extrema  in 
$(0,\pi)$ for $f_{p,q}$ as can be seen in Figures~\ref{fig:5-2-cos} 
and~\ref{fig:8-3-cos}.
\end{remark}
 
\begin{proof}[End of the proof of the general case.]
We now explain how we can relax the assumption that $\cos px$ and $\cos qx$ 
(or equivalently $P_p $ and $P_q $) have no common zeros in $[0,\pi]$). 
A simple example is $p=3$ and $q=1$, where $x=\frac \pi 2$ is a common zero 
of $\cos 3x$ and $\cos x$. This zero is common to the all family 
$x \mapsto \cos 3x - t \cos x$. Looking at $f_{3,1} (x)$, we observe that we 
can regularize it at $\frac \pi 2$ and that it is enough to 
 apply the previous argument for showing that there is at least two 
solutions of $f_{3,1}(x)=t$ for $|t|\leq 1$
(see Figure \ref{fig:cos-3-1}). This is of course trivial in this case.

\begin{figure}[ht]
\centering
\includegraphics[width=8cm]{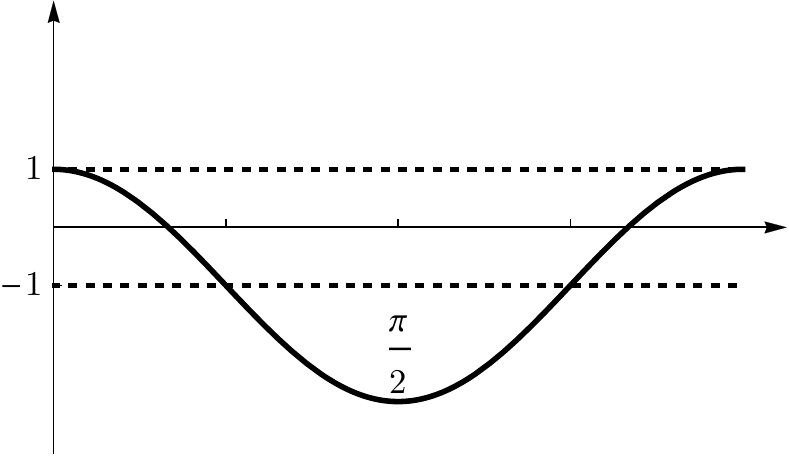}
\caption{The graph of the (regularized) $f_{3,1}$ in the interval 
$0<x<\pi$.}
\label{fig:cos-3-1}
\end{figure}

The general case is similar. We first determine the cardinal $p^*$  of the 
set $I^*_p$ in $\{ 1,\dots, p\}$) of the $k$'s such that $\hat x_k^{(p)}$ 
is a zero of $\cos q x$.

Similarly, we consider the set of the $\ell$ such that  
$\hat x_k^{(p)}=\hat x_\ell^{(q)}$ for some $k$. Another way of 
presentation is to claim the existence of a polynomial $Q$ of degree 
$p^*$ such that
\[
P_p = Q  \hat P_{p-p^*}\,,\, P_q = Q \hat P_{q-p^*}\,,
\]
so that the regularization of $f_{p,q}$ is given by:
\[
f_{p,q} (x) =  \frac{\hat P_{p-p^*}(\cos x)}{\hat P_{q-p^*}(\cos x)}\,.
\]
It suffices to do the same proof as before but keeping only in 
our construction the $x_k^{(p)}=x_\ell^{(q)}$ which are not the $p^*$ common 
zeros of $\cos px$ and $\cos qx$.

The proof is then identical and will give at least $p-p^*$ solutions 
of $f_{p,q}(x) =t$, hence of $P_{p}(x) - t P_q(x)=0$ in addition to the  $p^*$ previously obtained. 
\end{proof}

\begin{prop}
Let $(p,q)\in \mathbb N^*\times \mathbb N$ with $p>q$ and $\theta \in [0,\pi]$. 
Then the eigenfunction $\Psi:=\Phi^\theta_{p,q}$ satisfies
\begin{equation}\label{mupsi}
 \mu (\Psi) \geq p+q +1 \geq \sqrt{p^2 + q^2} +1 = \sqrt{\lambda} + 1\,.
 \end{equation}
\end{prop}

\begin{proof}
This is obtained by using Euler's formula, which implies (forgetting the 
contribution of the critical points inside the square)
\[
\mu(\Psi)\geq  b_1 + \frac 12 \#\{\mbox {boundary  points}\}   \,,
\]
where $b_1$ is the number of connected components of 
$\partial \Omega \cup N(\Psi)$ ($N(\Psi)$ being the zero set of $\Psi$).
Using $b_1\geq 1$ and Proposition \ref{prop:touchingpoints} we 
get~\eqref{mupsi}.
\end{proof}

This gives a rather explicit way to prove  that for a specific family which 
seems relatively generic (any family corresponding to eigenvalues of 
multiplicity at most $2$) the number of nodal domains tends to $+\infty$. Hence we conjecture:
\begin{conjecture}
 For any sequence of eigenfunctions of the Neumann problem in the square associated 
with an infinite sequence of eigenvalues, the number of nodal domains tends to $+\infty$.
\end{conjecture}

\section*{Acknowledgements} The authors would like to thank
P.~B\'erard, C.~L\'ena, J.~Leydold, T.~Hoffmann-Ostenhof for remarks or 
transmission of information.

All numerical calculations and graphs were done
with the computer software Mathematica, except for the images in 
Figure~\ref{fig:5-2alt}, which were created with MetaPost.

\end{document}